\documentclass[UTF8,fontset=windows,final, leqno]{siamltex}
\usepackage{ctex}
\usepackage{amsmath}
\allowdisplaybreaks[4]
\usepackage{epsfig}
\usepackage{graphicx}
\usepackage{subfigure}
\usepackage{amssymb}
\usepackage{CJK}
\usepackage{color}
\usepackage[noadjust]{cite}
\usepackage{hyperref}
\usepackage{float}
\usepackage{caption}
\numberwithin{equation}{section}
\newtheorem{remark}{Remark}[section]

\renewcommand\abstractname{Abstract}

\renewcommand\figurename{Figure}
\renewcommand\tablename{Table}

\normalsize

\renewcommand\refname{References}
\def\lam{{\beta}}

\def\Ome{{\Omega}}

\def\nab{{\nabla}}
\def\vepsi{{\varepsilon}}
\def\p{{\partial}}
\def\reff#1{\eqref{#1}}
\def\norm#1#2{\Vert\,#1\,\Vert_{#2}}

\def\vepsi{\varepsilon}

\def\cT{{\mathcal T}}

\def\no{{\nonumber}}

\def\div{{\mbox{\rm div\,}}}

\def\p{{\partial}}

\def\nab{\nabla}
\def\Ome{\Omega}
\def\lam{\beta}

\newcommand{\bRM}{\mathbf{RM}}
\newcommand{\br}{\mathbf{r}}

\def\ba{\mathbf{a}}
\def\bb{\mathbf{b}}

\def\bC{\mathbf{C}}
\def\bv{\mathbf{v}}

\def\bg{\mathbf{g}}
\def\bn{\mathbf{n}}
\def\bH{\mathbf{H}}
\def\bV{\mathbf{V}}
\def\bL{\mathbf{L}}
\def\bP{\mathbf{P}}

\def\bV{\mathbf{V}}

\def\bX{\mathbf{X}}

\def\R{\mathbb{R}}

\def\bx{{\bf x}}
\CTEXoptions[today=old]

\begin{document}
	
	%    \ifpdf
	%    \DeclareGraphicsExtensions{.pdf, .jpg, .tif, .png}
	%    \else
	%    \DeclareGraphicsExtensions{, .jpg}
	%    \fi
	
	%\title{Multiphysics Finite Element Methods for a Poroelasticity Model}
	\title{A new multiphysics finite element method for a Biot model with secondary consolidation\footnote{Last update: \today}}
	
	\author{
		Zhihao Ge\thanks{School of Mathematics and Statistics, Henan University, Kaifeng 475004, P.R. China ({\tt zhihaoge@henu.edu.cn}).
			The work was supported by the National Natural Science Foundation of China under grant No. 11971150.}
		\and
		Wenlong He\thanks{School of Mathematics and Statistics, Henan University, Kaifeng 475004, P.R. China.}
	}
	
	%%%
	\maketitle
	
	%\vspace{-1.4in}
	%\slugger{sinum}{200x}{xx}{x}{xxx--xxx}
	%\vspace{1.1in}
	
	\setcounter{page}{1}
	
	%\begin{PII}
	%S00000000000
	%\end{PII}
	
	%\large
	\begin{abstract}
	In this paper, we propose a new multiphysics finite element method for a Biot model with secondary consolidation in soil dynamics. To better describe the processes of deformation and diffusion underlying in the original model, we reformulate Biot model by a new multiphysics approach, which transforms the fluid-solid coupled problem to a fluid coupled problem--a generalized Stokes problem and a diffusion problem. Then, we give the energy law and prior error estimate of the weak solution. And we design a fully discrete time-stepping scheme to use mixed finite element method for $P_2-P_1-P_1$ element pairs to approximate the space variables and backward Euler method for the time variable, and we prove the discrete energy laws and the optimal convergence order error estimates. Also, we show some numerical examples to verify the theoretical results. Finally, we draw a conclusion to summarize the main results of this paper.
	\end{abstract}

	\pagestyle{myheadings}
	\thispagestyle{plain}
	\markboth{ZHIHAO GE, WENLONG HE}{NEW MFEM FOR A BIOT MODEL WITH SECONDARY CONSOLIDATION}
	
	%%%%%%%%%%%%%%%%%%%%%%%%%%%%%%
	\section{Introduction}
	Biot model in soil dynamics is widely distributed and plays a particularly important role in the construction of civil engineering, such as industrial and civil buildings, roads and bridges, water conservancy facilities, embankments and ports (cf. \cite{biot,20220106,20220104,20220105,20211214}). Also, the general Biot model is widely used in various fields such as geophysics, biomechanics, chemical engineering, materials science and so on, one can refer to \cite{2,3,6,7,8,10,biot,coussy04,de86}.   Compression deformation of saturated clay is usually based on Terzaghi's consolidation theory and Biot's consolidation theory(cf.\cite{biot,20220105,20220103}). Secondary consolidation is a process in which the volume of saturated clay decreases with time after the completion of primary consolidation, which plays an important role in the study of clay.  In this paper, we consider the following Biot model with secondary consolidation (cf. \cite{20210819}):
	\begin{alignat}{2}\label{1.1} 
		-\lambda^{*}\nabla(\div\pmb\tau)_{t}-\div\sigma(\pmb\tau) + b_0 \nab p &= \mathbf{F}
		&&\qquad \mbox{in } \Ome_T:=\Ome\times (0,T)\subset \mathbf{\R}^d\times (0,T),\\
		(a_0p+b_0 \div \pmb\tau)_t + \div \pmb\zeta_f &=\phi &&\qquad \mbox{in } \Ome_T,\label{1.2}
	\end{alignat}
	where
	\begin{align} 
		&\sigma(\pmb\tau)=\gamma\varepsilon(\pmb\tau)+\beta tr(\varepsilon(\pmb\tau))\mathbf{I},~~~~~ \varepsilon(\pmb\tau)=\dfrac{1}{2}(\nabla\pmb\tau+\nabla^{T}\pmb\tau), \label{1.3}\\
		&\pmb\zeta_f:= -\frac{K}{\theta_f} \bigl(\nab p -\rho_f \bg \bigr). \label{1.4}
	\end{align}
	Here $\pmb\tau$ denotes the displacement vector of the solid and $p$ denotes the pressure of the solvent. $\mathbf{I}$ denotes the $d\times d$ identity matrix and $\vepsi(\pmb\tau)$ is
	known as the deformed Green strain tensor. $\mathbf{F}$ is the body force. The permeability tensor $K=K(x)$ is assumed to be symmetric and uniformly positive definite in the sense that there exists positive constants $K_1$ and $K_2$ such that
	$K_1|\zeta|^2\leq K(x)\zeta\cdot \zeta \leq K_2 |\zeta|^2$ for a.e. $x\in\Omega$ and $\zeta\in \mathbf{\R}^d$; the solvent viscosity $\theta_f$, Biot-Willis constant
	$b_0$, $ \lambda^{*}\geq0 $ and the constrained specific storage coefficient
	$a_0$. In addition, $\sigma(\pmb\tau)$ is called the (effective) stress tensor. $\pmb\zeta_f$ is the volumetric solvent flux and (\ref{1.4}) is called the well-known Darcy's law. $ \beta $ and $ \gamma $ are
	Lam\'e constants, $\widehat{\sigma}(\pmb\tau, p):=\sigma(\pmb\tau)-b_0 p \mathbf{I}$ is the total stress tensor. We assume that $\rho_f\not\equiv 0$, which is a realistic assumption.
	
	As for Biot model with primary consolidation, there are more research results, for example, Phillips and Wheeler propose and analyze a continuous-in-time linear poroelasticity model in \cite{pw07};  Feng, Ge and Li in \cite{fgl14} propose  propose  a stable finite element method by a multiphysics approach, and so on.   The secondary consolidation was introduced and developed by Cushman and Murad in \cite{20211214}. Showalter find the term of $ \lambda^{*}\nabla(\div\pmb\tau)_{t} $ has a  effect for the momentum equation when $ \lambda^{*}>0 $ similar to that of $ a_0>0 $ for the diffusion equation in \cite{20210819}. Gaspar introduce a stabilized method for Biot Model with secondary
	consolidation by using the finite difference method on staggered grids in \cite{20220101}. Lewis and Schrefler use the finite element to study Biot model with secondary consolidation but not overcome the ``locking phenomenon" in \cite{20220102}. In this paper, we reformulate the Biot model by introducing of variables $ q=\div\pmb\tau,\varpi=a_{0}p+b_{0}q,\delta:=b_{0}p-\beta q-\lambda^{*}q_{t} $, which is different from the introduced variables in \cite{fgl14}. We successfully transformed the fluid-solid coupled problem \reff{1.1}-\reff{1.4} to a fluid coupled problem \reff{2.5}-\reff{2.10}.  We also give the energy estimates and prior error estimate and prove that the time-stepping method has the optimal convergence order.  One  can see that the pressure produce the numerical oscillation in Figure \ref{figure_9} and Figure \ref{figure_11} by using the element pair $P_2-P_1$ to solve the original model when $ a_{0} $ trends to $ 0 $, but the pressure is stable in in Figure \ref{figure_10} and Figure \ref{figure_12} by using the element pair $P_2-P_1-P_1$ in the reformulated model when $ a_{0} $ trends to $ 0 $.
	
	The remainder of this paper is organized as follows. In Section \ref{sec-2}, we reformulate the original model based on a multiphysics approach to a fluid-fluid coupling system and give the definition of weak solution to the original model and the reformulated model. Besides, we give the energy law and prior error estimate. In Section \ref{sec-3}, we propose and analyze the coupled and decoupled time stepping methods based on the multiphysics approach and prove that the time-stepping has the optimal convergence order. In Section \ref{sec-4}, we provide some numerical experiments to verify the theoretical results of the proposed approach and methods. Finally, we draw a conclusion to summarize the main results of this paper.
	
	\section{Multiphysics reformulation and PDE analysis}\label{sec-2}
	%\subsection{Preliminaries}
	To close the above system, we set the following boundary and initial conditions in this paper:
	\begin{alignat}{2} \label{2.1}
		\widehat{\sigma}(\pmb\tau,p)\bn=\lambda^{*}(\div\pmb\tau)_{t}\bn+\sigma(\pmb\tau)\bn-b_0 p \bn &= \mathbf{F}_1
		&&\qquad \mbox{on } \p\Ome_T:=\p\Ome\times (0,T),\\
		\pmb\zeta_f\cdot\bn= -\frac{K}{\theta_f} \bigl(\nab p -\rho_f \bg \bigr)\cdot \bn
		&=\phi_1 &&\qquad \mbox{on } \p\Ome_T, \label{2.2} \\
		\pmb\tau=\pmb\tau_0,\qquad p&=p_0 &&\qquad \mbox{in } \Ome\times\{t=0\}. \label{2.3}
	\end{alignat}
	Introduce new variables
	\[
	q:=\div \pmb\tau,\quad \varpi:=a_0p+b_0 q,\quad \delta:=b_0 p -\lam q-\lambda^{*}q_{t}.
	\]
	
	In some engineering literature, Lam\'e constant
	$\gamma$ is also called the {\em shear modulus} and denoted by $G$, and
	$B:=\lam +\frac23 G$ is called the {\em bulk modulus}. $\lam,~\gamma$ and $B$
	are computed from the {\em Young's modulus} $E$ and the {\em Poisson ratio}
	$\nu$ by the following formulas
	\[
	\lam=\frac{E\nu}{(1+\nu)(1-2\nu)},\qquad \gamma=G=\frac{E}{2(1+\nu)}, \qquad
	B=\frac{E}{3(1-2\nu)}.
	\]
	
	It is easy to check that
	\begin{align}\label{2.4}
		p=\chi_1 \delta + \chi_2 \varpi+\lambda^{*}\chi_1q_{t}, \qquad q=\chi_1 \varpi-\chi_3 \delta-\lambda^{*}\chi_3q_{t},
	\end{align}
	where $\chi_1= \frac{b_0}{b_0^2+\lam a_0},
	\chi_2=\frac{\lam}{b_0^2+\lam a_0},
	\chi_3=\frac{a_0}{b_0^2+\lam a_0}$.
	
	Then the problem \reff{1.1}-\reff{1.4} can be rewritten as
	\begin{alignat}{2} \label{2.5}
		-\gamma\div\varepsilon(\pmb\tau) + \nab \delta &= \mathbf{F} &&\qquad \mbox{in } \Ome_T,\\
		\chi_3\delta +\div \pmb\tau+\lambda^{*}\chi_3\div\pmb\tau_{t} &=\chi_1\varpi &&\qquad \mbox{in } \Ome_T, \label{2.6}\\
		\varpi_t - \frac{1}{\theta_f} \div[K (\nab (\chi_1 \delta + \chi_2 \varpi+\lambda^{*}\chi_1q_{t})-\rho_f\bg)]&=\phi
		&&\qquad \mbox{in } \Ome_T. \label{2.7}
	\end{alignat}
	The boundary and initial conditions \reff{2.1}-\reff{2.3} can be rewritten as
	\begin{alignat}{2} \label{2.8}
		\lambda^{*}(\div\pmb\tau)_{t}\bn+\sigma(\pmb\tau)\bn-b_0 (\chi_1 \delta + \chi_2 \varpi+\lambda^{*}\chi_1q_{t}) \bn &= \mathbf{F}_1
		&&~\mbox{on } \p\Ome_T:=\p\Ome\times (0,T),\\
		-\frac{K}{\theta_f} \bigl(\nab (\chi_1 \delta + \chi_2 \varpi+\lambda^{*}\chi_1q_{t}) -\rho_f \bg \bigr)\cdot \bn
		&=\phi_1 &&~ \mbox{on } \p\Ome_T, \label{2.9} \\
		\pmb\tau=\pmb\tau_0,\qquad p&=p_0 &&~\mbox{in } \Ome\times\{t=0\}. \label{2.10}
	\end{alignat}
%\subsection{Definition of weak solution}\label{subsec-3}

In this paper, $\Omega \subset \R^d \,(d=2,3)$ denotes a bounded polygonal domain with the boundary
$\p\Ome$. The standard function space notation is adopted in this paper, their
precise definitions can be found in \cite{bs08,cia,temam}.
In particular, $(\cdot,\cdot)$ and $\langle \cdot,\cdot\rangle$
denote respectively the standard $L^2(\Ome)$ and $L^2(\p\Ome)$ inner products. For any Banach space $B$, we let $\mathbf{B}=[B]^d$, and use $\mathbf{B}^\prime$ to denote its dual space. In particular, we use $(\cdot,\cdot)_{\small\rm dual}$ %and $\langle \cdot,\cdot \rangle_{\small\rm dual}$
to denote the dual product on $\bH^1(\Ome)' \times \bH^1(\Ome)$, and $\norm{\cdot}{L^p(B)}$ is a shorthand notation for
$\norm{\cdot}{L^p((0,T);B)}$.\\
We also introduce the function spaces
\begin{align*}
	&L^2_0(\Omega):=\{q\in L^2(\Omega);\, (q,1)=0\}, \qquad \bX:= \bH^1(\Ome).
\end{align*}
From \cite{temam}, it is well known  that the following  inf-sup condition holds in the space $\bX\times L^2_0(\Ome)$:
\begin{align}\label{2.11}
	\sup_{\bv\in \bX}\frac{(\div \bv,\varphi)}{\norm{\bv}{H^1(\Ome)}}
	\geq \alpha_0 \norm{\varphi}{L^2(\Ome)} \qquad \forall
	\varphi\in L^2_0(\Ome),\quad \alpha_0>0.
\end{align}
Let
\[
\bRM:=\{\br:=\ba+\bb \times x;\, \ba, \bb, x\in \R^d\}
\]
denote the space of infinitesimal rigid motions. It is well known \cite{bs08, gra, temam} that $\bRM$ is the kernel of
the strain operator $\vepsi$, that is, $\br\in \bRM$ if and only if
$\vepsi(\br)=0$. Hence, we have
\begin{align}
	\vepsi(\br)=0,\quad \div \br=0 \qquad\forall \br\in \bRM. \label{2.12}
\end{align}
Let $\bL^2_\bot(\p\Ome)$ and $\bH^1_\bot(\Ome)$ denote respectively the subspaces of $\bL^2(\p\Ome)$ and $\bH^1(\Ome)$ which are orthogonal to $\bRM$, that is,
\begin{align*}
	&\bH^1_\bot(\Ome):=\{\bv\in \bH^1(\Ome);\, (\bv,\br)=0\,\,\forall \br\in \bRM\},
	\\
	&\bL^2_\bot(\p\Ome):=\{\bg\in \bL^2(\p\Ome);\,\langle \bg,\br\rangle=0\,\,
	\forall \br\in \bRM \}.
\end{align*}
It is well known \cite{dautray} that there exists a constant $c_1>0$ such that
\begin{eqnarray}
	\inf_{\br\in \bRM}\|\bv+\br\|_{L^2(\Ome)}
	\le c_1\|\vepsi(\bv)\|_{L^2(\Ome)} \qquad\forall \bv\in\bH^1(\Ome).\label{2.13}
\end{eqnarray}
From \cite{fgl14}, we know that for each $\bv\in \bH^1_\bot(\Ome)$ there holds the following alternative version of the inf-sup condition
\begin{eqnarray}
	\sup_{\bv\in \bH^1_\bot(\Ome)}\frac{(\div \bv,\varphi)}{\norm{\bv}{H^1(\Ome)}}
	\geq \alpha_1 \norm{\varphi}{L^2(\Ome)} \qquad \forall
	\varphi\in L^2_0(\Ome),\quad \alpha_1>0.\label{2.14}
\end{eqnarray}

For convenience, we assume that $\mathbf{F},~ \mathbf{F}_1,~ \phi$
and $\phi_1$ all are independent of $t$ in the remaining of the paper. We note that all the results of this paper can be easily extended to the case of time-dependent
source functions.
\begin{definition}\label{weak1}
	Let $\pmb\tau_0\in\bH^1(\Ome),~ \mathbf{F}\in\bL^2(\Omega),~
	\mathbf{F}_1\in \bL^2(\p\Ome),~ p_0\in L^2(\Ome),~ \phi\in L^2(\Ome)$,
	and $\phi_1\in  L^2(\p\Ome)$.  Assume $a_0>0$ and
	$(\mathbf{F},\bv)+\langle \mathbf{F}_1,~ \bv \rangle =0$ for any $\bv\in \mathbf{RM}$.
	Given $T > 0$, a tuple $(\pmb\tau,p)$ with
	\begin{alignat*}{2}
		&\pmb\tau\in L^\infty\bigl(0,T; \bH_\perp^1(\Ome)),
		&&\qquad p\in L^\infty(0,T; L^2(\Omega))\cap L^2 \bigl(0,T; H^1(\Omega)\bigr), \\
		&p_t, (\div\pmb\tau)_t \in L^2(0,T;H^{1}(\Ome)')
		&&\qquad %a_0^{\frac12} p\in L^\infty\bigl(0,T; L^2(\Ome)),
	\end{alignat*}
	is called a weak solution to the problem \reff{1.1}--\reff{1.4}, if there hold for almost every $t \in [0,T]$
	\begin{alignat}{2}\label{2.15}
		&\lambda^{*}((\div\pmb\tau)_{t},\div\bv)+\gamma\bigl( \varepsilon(\pmb\tau), \vepsi(\bv) \bigr)&& \\
		&\hskip 0.5in
		+\lam\bigl(\div\pmb\tau, \div\bv \bigr)
		-b_0 \bigl( p, \div \bv \bigr)=(\mathbf{F}, \bv)+\langle \mathbf{F}_1,\bv\rangle
		&&\quad\forall \bv\in \bH^1(\Ome), \no \\
		&\bigl((a_0 p +b_0\div\pmb\tau)_t, \varphi \bigr)_{\rm dual}
		+ \frac{1}{\theta_f} \bigl( K(\nab p-\rho_f\bg), \nab \varphi \bigr)
		\label{2.16} \\
		&\hskip 2in =\bigl(\phi,\varphi\bigr)
		+\langle \phi_1,\varphi \rangle
		&&\quad\forall \varphi \in H^1(\Ome), \no  \\
		&\pmb\tau(0) = \pmb\tau_0,\qquad p(0)=p_0.  && \label{2.17}
	\end{alignat}
\end{definition}
Similarly, we can define the weak solution to the problem \reff{2.5}-\reff{2.10} as follows:
\begin{definition}\label{weak2}
	Let $\pmb\tau_0\in \bH^1(\Ome), \mathbf{F} \in \bL^2(\Omega),
	\mathbf{F}_1 \in \bL^2(\p\Ome), p_0\in L^2(\Ome), \phi\in L^2(\Ome)$,
	and $\phi_1\in L^2(\p\Ome)$.  Assume $a_0>0$ and
	$(\mathbf{F},\bv)+\langle \mathbf{F}_1, \bv \rangle =0$ for any $\bv\in \mathbf{RM}$.
	Given $T > 0$, a $5$-tuple $(\pmb\tau,\delta,\varpi,p,q)$ with
	\begin{alignat*}{2}
		&\pmb\tau\in L^\infty\bigl(0,T; \bH_\perp^1(\Ome)), &&\qquad
		\delta\in L^\infty \bigl(0,T; L^2(\Omega)\bigr), \\
		&\varpi\in L^\infty\bigl(0,T; L^2(\Omega)\bigr)
		\cap H^1\bigl(0,T; H^{1}(\Omega)'\bigr),
		&&\qquad q\in L^\infty(0,T;L^2(\Ome)), \\
		&p\in L^\infty \bigl(0,T; L^2(\Omega)\bigr) \cap L^2 \bigl(0,T; H^1(\Omega)\bigr)  &&
	\end{alignat*}
	is called a weak solution to the problem \reff{2.5}-\reff{2.7},
	if there hold for almost every $t \in [0,T]$
	\begin{alignat}{2}\label{2.18}
		&\gamma\bigl( \varepsilon(\pmb\tau), \vepsi(\bv) \bigr)-\bigl( \delta, \div\bv \bigr)= (\mathbf{F}, \bv)+\langle \mathbf{F}_1,\bv\rangle
		&&\quad\forall \bv\in \bH^1(\Ome),\\
		&\chi_3 \bigl( \delta, \varphi \bigr) +\bigl(\div\pmb\tau, \varphi \bigr)+\lambda^{*}\chi_3(\div\pmb\tau_{t},\varphi)
		= \chi_1\bigl(\varpi, \varphi \bigr) &&\quad\forall \varphi \in L^2(\Ome), \label{2.19}  \\
		&\bigl(\varpi_t, \psi \bigr)_{\rm dual}
		+\frac{1}{\theta_f} \bigl(K(\nab (\chi_1\delta +\chi_2\varpi+\lambda^{*}\chi_1\div\pmb\tau_{t}) -\rho_f\bg), \nab \psi \bigr) \label{2.20} \\
		&\hskip 1in= (\phi, \psi)+\langle \phi_1,\psi\rangle &&\quad\forall \psi \in H^1(\Ome) , \no  \\
		&p:=\chi_1\delta +\chi_2\varpi+\lambda^{*}\chi_1\div\pmb\tau_{t}, ~~
		q:=\chi_1\varpi-\chi_3\delta-\lambda^{*}\chi_3\div\pmb\tau_{t}, \label{2.21} \\
		%\pmb\tau(0) = \pmb\tau_0, \qquad p(0) &=p_0, && \label{e2.8} \\
		%q(0)=q_0:=\div \pmb\tau_0,\quad \quad
		&\varpi(0)= \varpi_0:=a_0p_0+b_0 q_0,  && \label{2.22}
	\end{alignat}
	where $q_0:=\div \pmb\tau_0$, $u_0$ and $p_0$ are same as in Definition \reff{weak1}.
\end{definition}
\begin{remark}
	The introduced new variables $ \delta $ and $ \varpi $ transform \reff{2.15} into a more concise form \reff{2.18}, which is a significant advantage than original model.
\end{remark}

\begin{lemma}\label{lem2.2}
	Every weak solution $(\pmb\tau,\delta,\varpi)$ of the problem \reff{2.18}--\reff{2.22} satisfies
	the following energy law
	\begin{eqnarray}\label{2.26}
		&&\qquad J(s)+\lambda^{*}\int_{0}^{s}\left\|(\div\pmb\tau)_{t} \right\|_{L^{2}(\Omega)}^{2}dt+ \frac{1}{\theta_f} \int_0^s \bigl(K(\nab(\chi_1\delta+\chi_2\varpi+\lambda^{*}q_{t})-\rho_f\bg),
		\\
		&&\quad \nab (\chi_1\delta+\chi_2\varpi+\lambda^{*}q_{t})\bigr)\, dt -\int_0^s \bigl(\phi, \chi_1\delta+\chi_2\varpi+\lambda^{*}q_{t})\, dt-\int_0^s \langle \phi_1, \chi_1\delta+\chi_2\varpi+\lambda^{*}q_{t} \rangle\, dt = J(0)\no
	\end{eqnarray}
	for all $t\in [0, T]$,  where
	\begin{align}\label{2.27}
		J(s):&= \frac12 \Bigl[\gamma\left\|\varepsilon(\pmb\tau(s)) \right\|_{L^{2}(\Omega)}^{2}+\chi_2 \norm{\varpi(s)}{L^2(\Ome)}^2+\chi_3 \norm{\delta(s)}{L^2(\Ome)}^2 \no\\
		&+(\lambda^{*})^{2}\chi_3\left\| \div\pmb\tau_{t}(s)\right\| _{L^{2}(\Omega)}^{2}+\lambda^{*}\chi_3(\div\pmb\tau_{t}(s),\delta(s))-2\bigl(\mathbf{F},\pmb\tau(s)\bigr) -2\langle \mathbf{F}_1, \pmb\tau(s) \rangle \Bigr].
	\end{align}
	
	Moreover, there holds
	\begin{align}\label{2.28}
		\norm{\varpi_t}{L^2(0.T;H^{1}(\Ome)')}
		\leq\frac{1}{\theta_f} \norm{K\nab (\chi_1\delta+\chi_2\varpi+\lambda^{*}q_{t})-\rho_f \bg}{L^2(\Ome)} 
		+ \|\phi\|_{L^2(\Ome)} + \|\phi_1\|_{L^2(\p\Ome)} < \infty.
	\end{align}
	
\end{lemma}
\begin{proof}
	We only consider the case of $\pmb\tau_t\in L^2(0,T; L^2(\Ome))$. Setting $\bv=\pmb\tau_t$ in \reff{2.18},
	differentiating \reff{2.19} with respect to $t$ followed by taking $\varphi=\delta$, and setting
	$\psi=p=\chi_1\delta +\chi_2\varpi+\lambda^{*}\chi_1\div\pmb\tau_{t}$ in \reff{2.20}, we have
	\begin{eqnarray}
		&\gamma \bigl(\varepsilon(\pmb\tau), \vepsi(\pmb\tau_t) \bigr)-\bigl( \delta, \div \pmb\tau_{t} \bigr)= (\mathbf{F}, \pmb\tau_{t})+\langle \mathbf{F}_1,\pmb\tau_{t}\rangle
		,\label{2.29}\\
		&\chi_3 \bigl( \delta_{t}, \delta \bigr) +\bigl(\div\pmb\tau_{t}, \delta \bigr)+\lambda^{*}\chi_3(\div\pmb\tau_{tt},\delta)
		= \chi_1\bigl(\varpi_{t}, \delta \bigr) ,   \\
		&\qquad\bigl(\varpi_t, p \bigr)_{\rm dual}
		+\frac{1}{\theta_f} \bigl(K(\nab (\chi_1\delta +\chi_2\varpi+\lambda^{*}\chi_1\div\pmb\tau_{t}) -\rho_f\bg), \nab p \bigr) = (\phi, p)+\langle \phi_1,p\rangle.
	\end{eqnarray}
	
	Adding the resulting equations and integrating in
	$t$ over $ (0,t) $ for any $s\in(0, T]$, we get
	\begin{eqnarray}
		&&\dfrac{\gamma}{2}\left\| \varepsilon(\pmb\tau(s)) \right\|_{L^{2}(\Omega)}^{2}+\dfrac{\chi_2}{2}\left\|\varpi(s) \right\|_{L^{2}(\Omega)}^{2}+\dfrac{\chi_3}{2}\left\|\delta(s) \right\|_{L^{2}(\Omega)}^{2} \label{2.32}\\
		&&+ \frac{1}{\theta_f} \int_0^s \bigl(K (\nab p-\rho_f\bg), \nab p\bigr)\, dt
		-\int_0^s \bigl(\phi, p\bigr)\, dt
		-\int_0^s \langle \phi_1, p \rangle\, dt  \no\\
		&&-\bigl(\mathbf{F},\pmb\tau(s)\bigr) -\langle \mathbf{F}_1, \pmb\tau(s) \rangle+\int_{0}^{s}\left[ \lambda^{*}\chi_3(\div\pmb\tau_{tt},\delta)+\lambda^{*}\chi_1(\varpi_{t},\div\pmb\tau_{t})\right] \,dt\no\\
		&&=\dfrac{\gamma}{2}\left\|\varepsilon(\pmb\tau(0)) \right\|_{L^{2}(\Omega)}^{2}+\dfrac{\chi_2}{2}\left\|\varpi(0) \right\|_{L^{2}(\Omega)}^{2}+\dfrac{\chi_3}{2}\left\|\delta(0) \right\|_{L^{2}(\Omega)}^{2}\no\\
		&&-\bigl(\mathbf{F},\pmb\tau(0)\bigr) -\langle \mathbf{F}_1, \pmb\tau(0) \rangle .\no
	\end{eqnarray}
 Using the equality $ q_{t}=\chi_1\varpi_{t}-\chi_3\delta_{t}-\lambda^{*}\chi_3\div\pmb\tau_{tt} $, we have
 \begin{eqnarray}
 	&&\qquad\lambda^{*}\chi_3(\div\pmb\tau_{tt},\delta)+\lambda^{*}\chi_1(\varpi_{t},\div\pmb\tau_{t})\label{2.33}\\
 	&&=\lambda^{*}\chi_3\left[ \dfrac{d}{dt}(\div\pmb\tau_{t},\delta)-(\div\pmb\tau_{t},\delta_{t})\right]+\lambda^{*}\chi_1(\varpi_{t},\div\pmb\tau_{t}) \no\\
 	&&=\lambda^{*}\chi_3\dfrac{d}{dt}(\div\pmb\tau_{t},\delta)+\lambda^{*}(\div\pmb\tau_{t},\chi_1\varpi_{t}-\chi_3\delta_{t})\no\\
 	&&=\lambda^{*}\chi_3\dfrac{d}{dt}(\div\pmb\tau_{t},\delta)+\lambda^{*}(\div\pmb\tau_{t},\div\pmb\tau_{t}+\lambda^{*}\chi_3\div\pmb\tau_{tt})\no\\
 	&&=\lambda^{*}\chi_3\dfrac{d}{dt}(\div\pmb\tau_{t},\delta) +\lambda^{*}(\div\pmb\tau_{t},\div\pmb\tau_{t})+\dfrac{(\lambda^{*})^{2}\chi_3}{2}\dfrac{d}{dt}(\div\pmb\tau_{t},\div\pmb\tau_{t}).\no
 \end{eqnarray}
Using \reff{2.33} and \reff{2.32}, we get
	\begin{eqnarray}
	&&\qquad\dfrac{\gamma}{2}\left\| \varepsilon(\pmb\tau(s)) \right\|_{L^{2}(\Omega)}^{2}+\dfrac{\chi_2}{2}\left\|\varpi(s) \right\|_{L^{2}(\Omega)}^{2}+\dfrac{\chi_3}{2}\left\|\delta(s) \right\|_{L^{2}(\Omega)}^{2}+\lambda^{*}\chi_3(\div\pmb\tau_{t}(s),\delta(s)) \label{2.34}\\
	&&+ \frac{1}{\theta_f} \int_0^s \bigl(K (\nab p-\rho_f\bg), \nab p\bigr)\, dt
	-\int_0^s \bigl(\phi, p\bigr)\, dt
	-\int_0^s \langle \phi_1, p \rangle\, dt-2\bigl(\mathbf{F},\pmb\tau(s)\bigr)  \no\\
	&& -2\langle \mathbf{F}_1, \pmb\tau(s) \rangle+\int_{0}^{s}\lambda^{*}\left\| \div\pmb\tau_{t}\right\| _{L^{2}(\Omega)}^{2}\,dt+\dfrac{(\lambda^{*})^{2}\chi_3}{2}\left\| \div\pmb\tau_{t}(s)\right\| _{L^{2}(\Omega)}^{2}\no\\
	&&=\dfrac{\gamma}{2}\left\|\varepsilon(\pmb\tau(0)) \right\|_{L^{2}(\Omega)}^{2}+\dfrac{\chi_2}{2}\left\|\varpi(0) \right\|_{L^{2}(\Omega)}^{2}+\dfrac{\chi_3}{2}\left\|\delta(0) \right\|_{L^{2}(\Omega)}^{2}-2\bigl(\mathbf{F},\pmb\tau(0)\bigr)\no\\
	&& -2\langle \mathbf{F}_1, \pmb\tau(0)+\lambda^{*}\chi_3(\div\pmb\tau_{t}(0),\delta(0))+\dfrac{(\lambda^{*})^{2}\chi_3}{2}\left\| \div\pmb\tau_{t}(0)\right\| _{L^{2}(\Omega)}^{2} \rangle,\no
\end{eqnarray}
which implies that \reff{2.26} holds. The inequality \reff{2.28} follows immediately from the following inequality
	\begin{eqnarray}
		&&\quad\qquad(\varpi_{t},\varphi)=-\dfrac{1}{\theta_f}(K(\nabla (\chi_1\delta+\chi_2\varpi+\lambda^{*}q_{t})-\rho_{f}\bg),\nabla\varphi)+(\phi,\psi)+\left\langle\phi_{1},\psi \right\rangle\\
		&&\leq\dfrac{1}{\theta_f}\left\|K(\nabla (\chi_1\delta+\chi_2\varpi+\lambda^{*}q_{t})-\rho_{f}\bg) \right\|_{L^{2}(\Omega)}\left\|\nabla\psi \right\|_{L^{2}(\Omega)}+\left\|\phi \right\|_{L^{2}(\Omega)}\left\|\psi \right\|_{L^{2}(\Omega)}+\left\|\phi_{1} \right\|_{L^{2}(\partial\Omega)}\left\|\psi \right\|_{L^{2}(\partial\Omega)}\no\\
		&&\leq\dfrac{1}{\theta_f}\left\|K(\nabla (\chi_1\delta+\chi_2\varpi+\lambda^{*}q_{t})-\rho_{f}\bg) \right\|_{L^{2}(\Omega)}\left\|\nabla\psi \right\|_{L^{2}(\Omega)}+\left\|\phi \right\|_{L^{2}(\Omega)}\left\|\psi \right\|_{L^{2}(\Omega)}+\left\|\phi_{1} \right\|_{L^{2}(\partial\Omega)}\left\|\psi \right\|_{L^{2}(\Omega)}\no.       
	\end{eqnarray}
	and the definition of the $ H^{1}(\Omega)^{'} $-norm. The proof is complete.
\end{proof}

 Likewise, the weak solution of \reff{2.15}--\reff{2.17} satisfies a similar energy law which is a rewritten version of \reff{2.26}.
\begin{lemma}\label{lem2.1}
	Every weak solution $(\pmb\tau,p)$ of the problem \reff{2.15}--\reff{2.17} satisfies
	the following energy law:
	\begin{align}\label{2.23}
		&E(s)+\lambda^{*}\int_{0}^{s}\left\|(\div\pmb\tau)_{t} \right\|_{L^{2}(\Omega)}^{2}dt+\frac{1}{\theta_f} \int_0^s \bigl( K(\nab p-\rho_f\bg), \nab p\bigr)\, dt
		\\
		&\quad-\int_0^s \bigl(\phi, p\bigr)\, dt-\int_0^s \langle \phi_1, p \rangle\, dt =E(0)\no
	\end{align}
	for all $t\in [0,T]$,  where
	\begin{align}\label{2.24}
		E(s):&= \frac12 \left[\gamma\left\|\varepsilon(\pmb\tau(s)) \right\|_{L^{2}(\Omega)}^{2}+\lam \norm{\div \pmb\tau(s)}{L^2(\Ome)}^2\right.\\
		&\left.+a_0\norm{p(s)}{L^2(\Ome)}^2-2\bigl(\mathbf{F},\pmb\tau(s)\bigr) -2\langle \mathbf{F}_1, \pmb\tau(s) \rangle \right] .\no
	\end{align}
	Moreover, there holds
	\begin{align}\label{2.25}
		\norm{(a_0p+b_0 \div \pmb\tau)_t}{L^2(0.T;H^{1}(\Ome)')}
		&\leq\frac{1}{\theta_f} \norm{K\nab p-\rho_f \bg}{L^2(\Ome)}  \\
		&\qquad
		+ \|\phi\|_{L^2(\Ome)} + \|\phi_1\|_{L^2(\p\Ome)} < \infty. \no
		%\leq E(0)^{\frac12}.
	\end{align}
	
\end{lemma}

\begin{lemma}\label{lem2.3}
	Every weak solution $(\pmb\tau,\delta,\varpi)$ of the problem \reff{2.18}--\reff{2.22} satisfies
	the following inequality
	\begin{eqnarray}\label{2.35}
		&&\qquad\widehat{J}(s)+\lambda^{*}\int_{0}^{s}\left\|(\div\pmb\tau)_{t} \right\|_{L^{2}(\Omega)}^{2}dt+ \frac{1}{\theta_f} \int_0^s \bigl(K(\nab(\chi_1\delta+\chi_2\varpi+\lambda^{*}q_{t})-\rho_f\bg),
		\\
		&& \nab (\chi_1\delta+\chi_2\varpi+\lambda^{*}q_{t})\bigr)\, dt -\int_0^s \bigl(\phi, \chi_1\delta+\chi_2\varpi+\lambda^{*}q_{t})\, dt-\int_0^s \langle \phi_1, \chi_1\delta+\chi_2\varpi+\lambda^{*}q_{t} \rangle\, dt \leq\widehat{J}(0)\no
	\end{eqnarray}
	for all $t\in [0, T]$,  where
	\begin{eqnarray}\label{2.36}
		&&\widehat{J}(s):= \frac12 \Bigl[\gamma\left\|\varepsilon(\pmb\tau(s)) \right\|_{L^{2}(\Omega)}^{2}+\chi_2 \norm{\varpi(s)}{L^2(\Ome)}^2-2\bigl(\mathbf{F},\pmb\tau(s)\bigr) -2\langle \mathbf{F}_1, \pmb\tau(s) \rangle \Bigr],\\
		&&\widehat{J}(0):= \frac12 \Bigl[\gamma\left\|\varepsilon(\pmb\tau(0)) \right\|_{L^{2}(\Omega)}^{2}+\chi_2 \norm{\varpi(0)}{L^2(\Ome)}^2+\chi_3 \norm{\delta(0)}{L^2(\Ome)}^2 \label{2.37}\\
		&&+(\lambda^{*})^{2}\chi_3\left\| \div\pmb\tau_{t}(0)\right\| _{L^{2}(\Omega)}^{2}+\lambda^{*}\chi_3\left\|\div\pmb\tau_{t}(0) \right\|_{L^{2}(\Omega)}^{2}+\lambda^{*}\chi_3\left\|\delta(0) \right\|_{L^{2}(\Omega)}^{2}\no\\
		&&-2\bigl(\mathbf{F},\pmb\tau(0)\bigr) -2\langle \mathbf{F}_1, \pmb\tau(0) \rangle \Bigr].\no
	\end{eqnarray}
\end{lemma}
\begin{lemma}\label{lem2.9}
	Every weak solution $(\pmb\tau,\delta,\varpi,p, q)$ to the problem \reff{2.18}-\reff{2.22} satisfies
	the following energy laws
	\begin{align}\label{2.62}
		&C_\varpi(t):=\bigl(\varpi(\cdot, t),1\bigr)
		=\bigl(\varpi_0,1\bigr) + \bigl[ (\phi,1) + \langle \phi_1, 1\rangle \bigr] t,\quad t\geq 0,\\
		&C_\delta(t):=\bigl( \delta(\cdot,t), 1\bigr)=\frac{\chi_1\gamma}{d+\chi_3\gamma}C_{\varpi}(t)-\frac{d\lambda^{*}\chi_3\chi_1\gamma}{(d+\chi_3\gamma)^{2}}\bigl[ (\phi,1) + \langle \phi_1, 1\rangle \bigr]\label{2.63}\\
		&~+\frac{1}{d+\chi_3\gamma}\left(-(\mathbf{F},\bx)-\left\langle\mathbf{F}_{1},\bx \right\rangle \right)\no,\\
		&C_q(t):=\bigl( q(\cdot,t), 1\bigr)
		=\chi_1 C_\varpi(t)-\chi_3 C_\delta(t), \label{2.64}\\
		&C_p(t):=\bigl( p(\cdot,t), 1\bigr)
		=\chi_1 C_\delta(t)+\chi_2 C_\varpi(t), \label{2.65}\\
		&C_{\pmb\tau}(t) := \bigl\langle \pmb\tau(\cdot,t)\cdot \bn, 1 \bigr\rangle  =C_q(t). \label{2.66}
	\end{align}
\end{lemma}
\begin{proof}
We first notice that \reff{2.62} follows immediately from taking $\psi\equiv 1$
	in \reff{2.20}. To prove \reff{2.63}, taking $\bv=\bx$ in \reff{2.18} and $\varphi=1$ in \reff{2.19}, which are
	valid test functions, and using the identities $\nabla \bx=\mathbf{I},~ \div \bx=d$, and $\varepsilon(\bx)=\mathbf{I}$, we get
	\begin{align*}
		\gamma\Bigl(\varepsilon(\pmb\tau), \mathbf{I}\Bigr)&=\gamma(\div\pmb\tau,1) = d\bigl( \delta, 1\bigr)+\bigl( \mathbf{F}, \bx\bigr)+ \langle \mathbf{F}_1,\bx\rangle,\\
		\bigl( \div \pmb\tau, 1\bigr) &=\chi_1(\varpi,1)-\chi_3(\delta, 1)-\lambda^{*}\chi_3(\div\pmb\tau_{t},1).
	\end{align*}
	It is easy to check that
	\begin{eqnarray}
		d\lambda^{*}\chi_3(\delta_{t},1)+(d+\chi_3\gamma)\bigl( \delta, 1\bigr)=\chi_1\gamma(\varpi,1)-(\mathbf{F},\bx)-\left\langle\mathbf{F}_{1},\bx \right\rangle. \no
	\end{eqnarray}
	According to ordinary differential equation theory, we get
\begin{eqnarray}
	( \delta, 1\bigr)&&=e^{-\int\frac{d+\chi_3\gamma}{d\lambda^{*}\chi_3}dt}\left(\int e^{\int\frac{d+\chi_3\gamma}{d\lambda^{*}\chi_3}dt}\left(\frac{\chi_1\gamma}{d\lambda^{*}\chi_3}(\varpi,1)-\frac{1}{d\lambda^{*}\chi_3}(\mathbf{F},\bx)-\frac{1}{d\lambda^{*}\chi_3}\left\langle\mathbf{F}_{1},\bx \right\rangle \right)dt+C  \right)\no\\
	&&=\frac{d\lambda^{*}\chi_3\chi_1\gamma}{(d+\chi_3\gamma)d\lambda^{*}\chi_3}(\varpi,1)-(\frac{d\lambda^{*}\chi_3}{d+\chi_3\gamma})^{2}\frac{\chi_1\gamma}{d\lambda^{*}\chi_3} \bigl[ (\phi,1) + \langle \phi_1, 1\rangle \bigr]\no\\
	&&~+e^{-\frac{d+\chi_3\gamma}{d\lambda^{*}\chi_3}t}\left(\int e^{\frac{d+\chi_3\gamma}{d\lambda^{*}\chi_3}t}\left(-\frac{1}{d\lambda^{*}\chi_3}(\mathbf{F},\bx)-\frac{1}{d\lambda^{*}\chi_3}\left\langle\mathbf{F}_{1},\bx \right\rangle \right)dt+C  \right),\no
\end{eqnarray} 
	which implies that \reff{2.63} holds.
	
	Finally, since $q=\chi_1\varpi-\chi_3\delta$,$p=\chi_1\delta + \chi_2 \varpi$, \reff{2.64} and \reff{2.65} follow from \reff{2.62}
	and \reff{2.63}. \reff{2.66} is an immediate consequence of $q=\div \pmb\tau$ and the result $ \bigl\langle \pmb\tau(\cdot,t)\cdot \bn, 1 \bigr\rangle=(\div\pmb\tau,1) $ of appealing to Gauss divergence theorem. The proof is complete.
\end{proof}

Using Lemma \ref{lem2.2}, Lemma \ref{lem2.1} and Lemma \ref{lem2.3}, we have the following solution estimates.
\begin{lemma}\label{estimates}
	There exists a positive constant
	$ \acute{C}_1=\acute{C}_1\bigl(\|\pmb\tau_0\|_{H^1(\Ome)}, \|p_0\|_{L^2(\Ome)},$
	$\|\mathbf{F}\|_{L^2(\Ome)},\|\mathbf{F}_1\|_{L^2(\p \Ome)},\|\phi\|_{L^2(\Ome)}, \|\phi_1\|_{L^2(\p\Ome)} \bigr)$ and $ \acute{C}_2=\acute{C}_2\bigl(\acute{C}_1,\|\nab p_0\|_{L^2(\Ome)}\bigr)$
	such that
	\begin{eqnarray}\label{2.41}
		&\sqrt{\lambda^{*}}\|(\div\pmb\tau)_{t}\|_{L^2(0,T;L^2(\Ome))}+\sqrt{\gamma}\|\varepsilon(\pmb\tau)\|_{L^\infty(0,T;L^2(\Ome))}+\sqrt{\chi_2} \|\varpi\|_{L^\infty(0,T;L^2(\Ome))} \\
		&
		+\sqrt{\chi_3} \|\lambda^{*}\div\pmb\tau_{t}+\delta\|_{L^\infty(0,T;L^2(\Ome))}
		+\sqrt{\frac{K_1}{\theta_f}} \|\nab p \|_{L^2(0,T;L^2(\Ome))} \leq \acute{C}_1, \no \\
		&\|\pmb\tau\|_{L^\infty(0,T;L^2(\Ome))}\leq \acute{C}_1, \quad
		\|p\|_{L^\infty(0,T;L^2(\Ome))} \leq \acute{C}_2 \bigl( \chi_2^{\frac12} + \chi_1 \chi_3^{-\frac12}
		\bigr), \label{2.42} \\
		&\|p\|_{L^2(0,T; L^2(\Ome))} \leq \acute{C}_1,~ \quad
		\|\delta\|_{L^2(0,T;L^2(\Ome))} \leq \acute{C}_2\chi_1^{-1} \bigl(1+ \chi_2^{\frac12} \bigr).
		\label{2.43}
	\end{eqnarray}
\end{lemma}
\begin{proof}
	Using the identity
	\begin{eqnarray}
		&\dfrac{1}{2}\dfrac{d}{dt}(\delta,\delta)+\dfrac{d}{dt}(\lambda^{*}\div\pmb\tau_{t},\delta)+\dfrac{1}{2}\dfrac{d}{dt}(\lambda^{*}\div\pmb\tau_{t},\lambda^{*}\div\pmb\tau_{t})\no\\
		&=\dfrac{1}{2}\dfrac{d}{dt}(\lambda^{*}\div\pmb\tau_{t}+\delta,\lambda^{*}\div\pmb\tau_{t}+\delta),\no
	\end{eqnarray}
and \reff{2.34}, we get 
\begin{eqnarray}
	&&\qquad\dfrac{\gamma}{2}\left\| \varepsilon(\pmb\tau(s)) \right\|_{L^{2}(\Omega)}^{2}+\dfrac{\chi_2}{2}\left\|\varpi(s) \right\|_{L^{2}(\Omega)}^{2}+\dfrac{\chi_3}{2}\left\|\lambda^{*}\div\pmb\tau_{t}(s)+\delta(s) \right\|_{L^{2}(\Omega)}^{2} \label{2.44}\\
	&&+ \frac{1}{\theta_f} \int_0^s \bigl(K (\nab p-\rho_f\bg), \nab p\bigr)\, dt
	-\int_0^s \bigl(\phi, p\bigr)\, dt
	-\int_0^s \langle \phi_1, p \rangle\, dt-2\bigl(\mathbf{F},\pmb\tau(s)\bigr)  \no\\
	&& -2\langle \mathbf{F}_1, \pmb\tau(s) \rangle+\int_{0}^{s}\lambda^{*}\left\| \div\pmb\tau_{t}\right\| _{L^{2}(\Omega)}^{2}\,dt=\dfrac{\gamma}{2}\left\|\varepsilon(\pmb\tau(0)) \right\|_{L^{2}(\Omega)}^{2}\no\\
	&&+\dfrac{\chi_2}{2}\left\|\varpi(0) \right\|_{L^{2}(\Omega)}^{2}+\dfrac{\chi_3}{2}\left\|\lambda^{*}\div\pmb\tau_{t}(0)+\delta(0) \right\|_{L^{2}(\Omega)}^{2}-2\bigl(\mathbf{F},\pmb\tau(0)\bigr)-2\langle \mathbf{F}_1, \pmb\tau(0).\no
\end{eqnarray}
Using \reff{2.44}, we have
\begin{eqnarray}
	&&\sqrt{\lambda^{*}}\|(\div\pmb\tau)_{t}\|_{L^2(0,T;L^2(\Ome))}+\sqrt{\gamma}\|\varepsilon(\pmb\tau)\|_{L^\infty(0,T;L^2(\Ome))}+\sqrt{\chi_2} \|\varpi\|_{L^\infty(0,T;L^2(\Ome))}\label{2.45} \\
	&&+\sqrt{\chi_3} \|\lambda^{*}\div\pmb\tau_{t}+\delta\|_{L^\infty(0,T;L^2(\Ome))}
	+\sqrt{\frac{K_1}{\theta_f}} \|\nab p \|_{L^2(0,T;L^2(\Ome))}\no\\
	&&\leq C\left(\left\|\varepsilon(\pmb\tau(0)) \right\|_{L^{2}(\Omega)}+\left\|\varpi(0) \right\|_{L^{2}(\Omega)}+\left\|\div\pmb\tau_{t}(0)\right\|_{L^{2}(\Omega)} +\left\| \delta(0) \right\|_{L^{2}(\Omega)} \right. \no\\
	&&\left.+\|\mathbf{F}\|_{L^2(\Ome)}+\|\mathbf{F}_1\|_{L^2(\p \Ome)}+\|\phi\|_{L^2(\Ome)}+\|\phi_1\|_{L^2(\p\Ome)}\right),\no
\end{eqnarray}
which implies the \reff{2.41} holds.  It's easy to check that \reff{2.42} holds from \reff{2.41} and the relation $p=\chi_1\delta +\chi_2\varpi+\lambda^{*}q_{t}$. We note that \reff{2.43} follows from \reff{2.41}, \reff{2.13}, the Poincar$\acute{e}$ inequality, \reff{2.65} and the relation $p=\chi_1\delta +\chi_2\varpi+\lambda^{*}q_{t}$. The proof is complete.
\end{proof}
\begin{theorem}\label{smooth}
	Suppose that $\pmb\tau_0$ and $p_0$ are sufficiently smooth, then
	there exist positive constants~ $ \acute{C}_2=\acute{C}_2\bigl(\acute{C}_1,\|\nab p_0\|_{L^2(\Ome)}\bigr)$ and $ \acute{C}_3=\acute{C}_3\bigl(\acute{C}_1,\acute{C}_2,\|\pmb\tau_0\|_{H^2(\Ome)},
	\|p_0\|_{H^2(\Ome)} \bigr)$ such that
	\begin{align}\label{2.46}
		&\sqrt{\lambda^{*}}\left\|(\div\pmb\tau)_{t} \right\|_{L^2(0,T;L^2(\Ome))}+ \sqrt{\gamma}\|\varepsilon(\pmb\tau_{t})\|_{L^2(0,T;L^2(\Ome))}
		+\sqrt{\chi_2} \|\varpi_t\|_{L^2(0,T;L^2(\Ome))} \\
		&\qquad
		+\sqrt{\chi_3} \|\delta_t\|_{L^2(0,T;L^2(\Ome))}
		+\sqrt{\frac{K_1}{\theta_f}} \|\nab p \|_{L^\infty(0,T;L^2(\Ome))}\no\\
		&+\sqrt{\gamma\lambda^{*}\chi_3}\|\varepsilon(\pmb\tau_{t})\|_{L^\infty(0,T;L^2(\Ome))}+\lambda^{*}\sqrt{\chi_3}\|\div\pmb\tau_{t}\|_{L^\infty(0,T;L^2(\Ome))} \leq \acute{C}_2, \no \\
		&\sqrt{\lambda^{*}}\left\|(\div\pmb\tau)_{tt} \right\|_{L^2(0,T;L^2(\Ome))}+\sqrt{\gamma}\|\varepsilon(\pmb\tau_{t})\|_{L^\infty(0,T;L^2(\Ome))}
		\label{2.47} \\
		&\quad
		+\sqrt{\chi_2} \|\varpi_t\|_{L^\infty(0,T;L^2(\Ome))}+\sqrt{\chi_3} \|\lambda^{*}\div\pmb\tau_{tt}+\delta_t\|_{L^\infty(0,T;L^2(\Ome))}\no\\
		&\quad+\sqrt{\frac{K_1}{\theta_f}} \|\nab p_{t} \|_{L^2(0,T;L^2(\Ome))} \leq \acute{C}_3, \no \\
		&\|\varpi_{tt}\|_{L^2(H^{1}(\Ome)')} \leq \sqrt{\frac{K_2}{\theta_f}}\acute{C}_3. \label{2.48}
	\end{align}
\end{theorem}
\begin{proof}
	To show \reff{2.46}, first differentiating \reff{2.18} one time with respect to $t$ and setting
	$\bv=\pmb\tau_{t}$, differentiating \reff{2.19} one time with respect to $t$ and setting
	$\varphi=\delta_t$, taking $ \psi=p_{t}=\chi_1\delta_{t}+\chi_2\varpi_{t}+\lambda^{*}\chi_1\div\pmb\tau_{t} $ in \reff{2.20}, we get
	\begin{eqnarray}
	&\gamma\left\|\varepsilon(\pmb\tau_{t}) \right\|_{L^{2}(\Omega)}^{2}-\bigl(\div\pmb\tau_{t},\delta_t \bigr)
	=0,\\
	&\chi_3\|\delta_t\|_{L^2(\Ome)}^2+(\div\pmb\tau_{t},\delta_{t})+\lambda^{*}\chi_3(\div\pmb\tau_{tt},\delta_{t})=\chi_1\bigl(\varpi_{t},\delta_t\bigr),\\
	&\qquad\chi_1 \bigl(\varpi_{t},\delta_t \bigr) + \chi_2 \|\varpi_t\|_{L^2(\Ome)}^2
	+\lambda^{*}\chi_1(\varpi_{t},\div\pmb\tau_{tt})+\frac{1}{\theta_f}(K(\nab p-\rho_{f}\bg),\nab p_{t})=\dfrac{d}{dt}\left[ (\phi,p)-\left\langle\phi_{1},p \right\rangle \right].
	\end{eqnarray}
	Adding the above equations and integrating in $ t $ from $ 0 $ to $ s $, we get
	\begin{align}\label{2.52}
		&\int_{0}^{s}\gamma\left\|\varepsilon(\pmb\tau_{t}) \right\|_{L^{2}(\Omega)}^{2}+\chi_2 \|\varpi_t\|_{L^2(\Ome)}^2+\chi_3 \|\delta_t\|_{L^2(\Ome)}^2+\lambda^{*}\chi_3(\div\pmb\tau_{tt},\delta_{t})\,dt \\
		&+ \frac{K}{2\theta_f}\|\nab p(s)\|_{L^2(\Ome)}^2+\int_{0}^{s}\lambda^{*}\chi_1(\varpi_{t},\div\pmb\tau_{tt})\,dt\no\\
		&=\frac{K}{2\theta_f}\|\nab p(0)\|_{L^2(\Ome)}^2+(\phi,p(s)-p(0))+\left\langle\phi_{1},p(s)-p(0) \right\rangle.\no
	\end{align}
Differentiating \reff{2.18} one time with respect to $t$ and setting
$\bv=\pmb\tau_{tt}$, we get
\begin{eqnarray}
	\dfrac{\gamma}{2}\dfrac{d}{dt}\left\| \varepsilon(\pmb\tau_{t})\right\|_{L^{2}(\Omega)}^{2}-(\delta_{t},\div\pmb\tau_{tt})=0.\label{2.53}
\end{eqnarray}
Using the equality $ q_{t}=\chi_1\varpi_{t}-\chi_3\delta_{t}-\lambda^{*}\chi_3\div\pmb\tau_{tt} $ and \reff{2.53}, we have
\begin{eqnarray}
	&&\lambda^{*}\chi_3(\div\pmb\tau_{tt},\delta_{t})+\lambda^{*}\chi_1(\varpi_{t},\div\pmb\tau_{tt})\label{2.54}\\
	&&=\lambda^{*}\chi_3(\div\pmb\tau_{tt},\delta_{t})+\lambda^{*}(\div\pmb\tau_{tt},\chi_3\delta_{t}+\div\pmb\tau_{t}+\lambda^{*}\chi_3\div\pmb\tau_{tt})\no\\
	&&=2\chi_3\lambda^{*}(\div\pmb\tau_{tt},\delta_{t})+\lambda^{*}(\div\pmb\tau_{t},\div\pmb\tau_{t})+(\lambda^{*})^{2}\chi_3(\div\pmb\tau_{t},\div\pmb\tau_{tt})\no\\
	&&=\gamma\lambda^{*}\chi_3\dfrac{d}{dt}\left\| \varepsilon(\pmb\tau_{t})\right\|_{L^{2}(\Omega)}^{2} +\lambda^{*}\left\|\div\pmb\tau_{t} \right\|_{L^{2}(\Omega)}^{2}+\frac{(\lambda^{*})^{2}\chi_3}{2} \dfrac{d}{dt}\left\|\div\pmb\tau_{t} \right\|_{L^{2}(\Omega)}^{2}.\no
\end{eqnarray}
Taking \reff{2.54} in \reff{2.52}, we obtain
\begin{align}\label{2.55}
	&\int_{0}^{s}\gamma\left\|\varepsilon(\pmb\tau_{t}) \right\|_{L^{2}(\Omega)}^{2}+\chi_2 \|\varpi_t\|_{L^2(\Ome)}^2+\chi_3 \|\delta_t\|_{L^2(\Ome)}^2+\lambda^{*}\left\|\div\pmb\tau_{t} \right\|_{L^{2}(\Omega)}^{2}\,dt \\
	&+\gamma\lambda^{*}\chi_3\left\| \varepsilon(\pmb\tau_{t}(s))\right\|_{L^{2}(\Omega)}^{2}+\frac{(\lambda^{*})^{2}\chi_3}{2} \left\|\div\pmb\tau_{t}(s) \right\|_{L^{2}(\Omega)}^{2}+ \frac{K}{2\theta_f}\|\nab p(s)\|_{L^2(\Ome)}^2\no\\
	&=\gamma\lambda^{*}\chi_3\left\| \varepsilon(\pmb\tau_{t}(0))\right\|_{L^{2}(\Omega)}^{2}+\frac{(\lambda^{*})^{2}\chi_3}{2}\left\|\div\pmb\tau_{t}(0) \right\|_{L^{2}(\Omega)}^{2}+\frac{K}{2\theta_f}\|\nab p(0)\|_{L^2(\Ome)}^2\no\\
	&+(\phi,p(s)-p(0))+\left\langle\phi_{1},p(s)-p(0) \right\rangle,\no
\end{align}
which implies \reff{2.46} holds. Differentiating \reff{2.19} twice with respect to $t$ and setting
$\varphi=\delta_t$. Differentiating \reff{2.20} one time with respect to $t$  and setting $\psi=p_{t}=\chi_1\delta_{t} + \chi_2\varpi_{t}$ in \reff{2.20}, we get 
\begin{eqnarray}\label{2.56}
	&\dfrac{\chi_3}{2}\dfrac{d}{dt}\|\delta_t\|_{L^2(\Ome)}^2+(\div\pmb\tau_{tt},\delta_{t})+\lambda^{*}\chi_3(\div\pmb\tau_{ttt},\delta_{t})=\chi_1\bigl(\varpi_{tt},\delta_t\bigr),\\
	&\qquad\chi_1 \bigl(\varpi_{tt},\delta_{t} \bigr)+\lambda^{*}\chi_1(\varpi_{tt},\div\pmb\tau_{tt}) + \dfrac{\chi_2}{2}\frac{d}{dt}\|\varpi_t\|_{L^2(\Ome)}^2
	+\frac{K}{\theta_f} \|\nab p_{t}\|_{L^2(\Ome)}^2
	=0.\label{2.57}
\end{eqnarray}
Adding \reff{2.53}, \reff{2.56} and \reff{2.57}, and integrating in $t$ we get for $s\in[0,T]$, we have
\begin{eqnarray}\label{2.58}
	&& \qquad\dfrac{\gamma}{2}\left\|\varepsilon(\pmb\tau_{t})(s) \right\|_{L^2(\Ome)}^2+ \dfrac{\chi_2}{2}\|\varpi_t(s)\|_{L^2(\Ome)}^2 + \dfrac{\chi_3}{2}\|\delta_t(s)\|_{L^2(\Ome)}^2 +\frac{K}{\theta_f}\int_0^s  \|\nab p_{t}\|_{L^2(\Ome)}^2\,dt\\
	&&+\lambda^{*}\chi_3(\div\pmb\tau_{ttt},\delta_{t})+\lambda^{*}\chi_1(\varpi_{tt},\div\pmb\tau_{tt})=\dfrac{\gamma}{2}\left\|\varepsilon(\pmb\tau_{t})(0) \right\|_{L^2(\Ome)}^2+ \dfrac{\chi_2}{2}\|\varpi_t(0)\|_{L^2(\Ome)}^2 + \dfrac{\chi_3}{2}\|\delta_t(0)\|_{L^2(\Ome)}^2. \no
\end{eqnarray}
Using the equality $ q_{tt}=\chi_1\varpi_{tt}-\chi_3\delta_{tt}-\lambda^{*}\chi_3\div q_{ttt} $, we have
\begin{eqnarray}
	&&\qquad\lambda^{*}\chi_3(\div\pmb\tau_{ttt},\delta_{t})+\lambda^{*}\chi_1(\varpi_{tt},\div\pmb\tau_{tt})\label{2.59}\\
	&&=\lambda^{*}\chi_3\left[ \dfrac{d}{dt}(\div\pmb\tau_{tt},\delta_{t})-(\div\pmb\tau_{tt},\delta_{tt})\right]+\lambda^{*}\chi_1(\varpi_{tt},\div\pmb\tau_{tt}) \no\\
	&&=\lambda^{*}\chi_3\dfrac{d}{dt}(\div\pmb\tau_{tt},\delta_{t})+\lambda^{*}(\div\pmb\tau_{tt},\chi_1\varpi_{tt}-\chi_3\delta_{tt})\no\\
	&&=\lambda^{*}\chi_3\dfrac{d}{dt}(\div\pmb\tau_{tt},\delta_{t})+\lambda^{*}(\div\pmb\tau_{tt},\div\pmb\tau_{tt})+\lambda^{*}\chi_3\div\pmb\tau_{ttt})\no\\
	&&=\lambda^{*}\chi_3\dfrac{d}{dt}(\div\pmb\tau_{tt},\delta_{t}) +\lambda^{*}(\div\pmb\tau_{tt},\div\pmb\tau_{tt})+\dfrac{(\lambda^{*})^{2}\chi_3}{2}\dfrac{d}{dt}(\div\pmb\tau_{tt},\div\pmb\tau_{tt}).\no
\end{eqnarray}
	Taking the identity
\begin{eqnarray}
	&\dfrac{1}{2}\dfrac{d}{dt}(\delta_{t},\delta_{t})+\dfrac{d}{dt}(\lambda^{*}\div\pmb\tau_{tt},\delta_{t})+\dfrac{1}{2}\dfrac{d}{dt}(\lambda^{*}\div\pmb\tau_{tt},\lambda^{*}\div\pmb\tau_{tt})\label{2.60}\\
	&=\dfrac{1}{2}\dfrac{d}{dt}(\lambda^{*}\div\pmb\tau_{tt}+\delta_{t},\lambda^{*}\div\pmb\tau_{tt}+\delta_{t}),\no
\end{eqnarray}
and \reff{2.59} in \reff{2.58}, we get
\begin{align}\label{2.61}
	&\int_0^s \lambda^{*}\left\|(\div\pmb\tau)_{tt} \right\|_{L^2(\Ome)}^2\,dt+ \dfrac{\gamma}{2}\left\|\varepsilon(\pmb\tau_{t})(s) \right\|_{L^2(\Ome)}^2\\
	&+ \dfrac{\chi_2}{2}\|\varpi_t(s)\|_{L^2(\Ome)}^2 + \dfrac{\chi_3}{2}\|\lambda^{*}\div\pmb\tau_{tt}(s)+\delta_t(s)\|_{L^2(\Ome)}^2 +\frac{K}{\theta_f}\int_0^s  \|\nab p_{t}\|_{L^2(\Ome)}^2\,dt\no\\
	&=\dfrac{\gamma}{2}\left\|\varepsilon(\pmb\tau_{t})(0) \right\|_{L^2(\Ome)}^2+ \dfrac{\chi_2}{2}\|\varpi_t(0)\|_{L^2(\Ome)}^2 + \dfrac{\chi_3}{2}\|\lambda^{*}\div\pmb\tau_{tt}(0)+\delta_t(0)\|_{L^2(\Ome)}^2, \no
\end{align}
which implies that \reff{2.47} holds. 
\reff{2.48} follows immediately from the following inequality
\begin{align*}
	\bigl( \varpi_{tt}, \psi \bigr) = -\frac{1}{\theta_f} \bigl( K\nab p_{t}, \nab \psi\bigr)
	\leq \frac{K_{2}}{\theta_f} \|\nab p_{t}\|_{L^2(\Ome)} \|\nab \psi\|_{L^2(\Ome)},
\end{align*}
\reff{2.47} and the definition of the $H^{1}(\Omega)'$-norm. The proof is complete.
\end{proof}

\begin{theorem}\label{thm2.8}
	Let $\pmb\tau_0\in\bH^1(\Ome), \mathbf{F}\in\bL^2(\Omega),
	\mathbf{F}_1\in \bL^2(\p\Ome), p_0\in L^2(\Ome), \phi\in L^2(\Ome)$, and $\phi_1\in L^2(\p\Ome)$. Suppose $a_0>0$ and $(\mathbf{F},\bv)+\langle \mathbf{F}_1, \bv \rangle =0$ for any $\bv\in \mathbf{RM}$. Then there exists a unique weak solution to the problem \reff{1.1}-\reff{1.4} in the sense of Definition \ref{weak1}. Likewise, there exists a unique weak solution to the problem
	\reff{2.5}-\reff{2.10} in the sense of Definition \ref{weak2}.
\end{theorem}
\begin{proof}
	The existence of weak solution can be easily proved by using the standard Galerkin method and the compactness argument (cf. \cite{temam}). Lemma \ref{lem2.2}, Lemma \ref{lem2.1} and Lemma \ref{lem2.3} provide the required uniform estimates for the Galerkin approximate solutions, since the derivation is standard, here we omit the details. 
	
	Next, we prove the uniqueness of the weak solution of the problem \reff{2.5}-\reff{2.10}. Lemma \ref{estimates} and Theorem \ref{smooth} gives  the priori estimates for the weak solution. Since $ C_{\varpi}(t)=(\varpi(\cdot,t),1)=(\varpi_{0},1)+\left[(\phi,1)+\left\langle \phi_{1},1\right\rangle  \right] $. It's easy to check that $ \varpi $ is unique. We assume that $ (\pmb\tau_{1},\delta_{1},\varpi_{1}) $ and $ (\pmb\tau_{2},\delta_{2},\varpi_{2}) $ are the different solutions of \reff{2.18}-\reff{2.22}. Using \reff{2.18} and \reff{2.19}, we obtain
	\begin{align}
		\gamma(\varepsilon(\pmb\tau_{1})-\varepsilon(\pmb\tau_{2}),\varepsilon(\bv))-(\delta_{1}-\delta_{2},\div\bv)&=0~~~~\forall~\bv\in \bH^{1}(\varOmega),\label{2.72}\\
		\chi_3(\delta_{1}-\delta_{2},\varphi)+(\div\pmb\tau_{1}-\div\pmb\tau_{2},\varphi)+\lambda^{*}\chi_3((\div\pmb\tau_{1}-\div\pmb\tau_{2})_{t},\varphi)&=0~~~~~\forall \varphi\in L^{2}(\varOmega).\label{2.73}
	\end{align}
	Adding \reff{2.72} and \reff{2.73}, letting $ \bv=\pmb\tau_{1}-\pmb\tau_{2},~ \varphi=\delta_{1}-\delta_{2} $, we have
	\begin{align}
		\dfrac{\lambda^{*}\chi_3\gamma}{2}\dfrac{d}{dt}\left\|\varepsilon(\pmb\tau_{1})-\varepsilon(\pmb\tau_{2}) \right\|^{2}_{L^{2}(\Omega)}+\gamma\left\|\varepsilon(\pmb\tau_{1})-\varepsilon(\pmb\tau_{2}) \right\|^{2}_{L^{2}(\Omega)}+\chi_3\left\|\delta_{1}-\delta_{2} \right\|^{2}_{L^{2}(\Omega)} =0.\label{2.74} 
	\end{align}
	Using \reff{2.74} and the initial value $ \pmb\tau_{0} $, we obtain
	\begin{align*}
		\pmb\tau_{1}=\pmb\tau_{2},~~~\delta_{1}=\delta_{2}.	
	\end{align*}
	Since $ p=\chi_1\delta+\chi_2\varpi ,~q=\chi_1\varpi-\chi_3\delta$, so we have
	\begin{align*}
		p_{1}=p_{2},~~~~q_{1}=q_{2}.
	\end{align*}	
	Hence, the solution of the problem \reff{2.18}-\reff{2.22} is unique. The proof is complete.
\end{proof}
\begin{remark}
	We can prove the existence and uniqueness of problem \reff{2.15}-\reff{2.17} by using the same argument as Theorem \ref{thm2.8}. 
\end{remark}
\section{Fully discrete finite element methods}\label{sec-3}
\subsection{Formulation of fully discrete finite element methods}
Let $\mathcal{T}_h$ be a 
quasi-uniform triangulation or rectangular partition of $\Omega$ with maximum mesh size $h$, and $\bar{\Omega}=\bigcup_{\mathcal{K}\in\mathcal{T}_h}\bar{\mathcal{K}}$. The time interval $[0, T]$ is divided into $N$ equal intervals, denoted by $[t_{n-1}, t_{n}], n=1,2,...N$,  and $\Delta t=\frac{T}{N}$, then $t_n=n\Delta t$. In this work, we use backward Euler method and denote $ d_{t} v^{n}:=\frac{v^{n}-v^{n-1}}{\Delta t}$. 

Also, let $(\bX_h, M_h)$ be a stable mixed finite element pair, that is, $\bX_h\subset \bH^1(\Omega)$ and $M_h\subset L^2(\Omega)$ 
satisfy the inf-sup condition
\begin{alignat}{2}\label{3.1}
	\sup_{\bv_h\in \bX_h}\frac{({\rm div} \bv_h, \varphi_h)}{\|\bv_h\|_{H^1(\Ome)}}
	\geq \beta_0\|\varphi_h\|_{L^2(\Ome)} &&\quad \forall\varphi_h\in M_{0h}:=M_h\cap L_0^2(\Omega),\ \beta_0>0.
\end{alignat}

A number of stable mixed finite element spaces $(\bX_h, M_h)$ have been known in the literature
\cite{brezzi}. A well-known example is the following
so-called Taylor-Hood element (cf. \cite{ber,brezzi}):
\begin{align*}
	\bX_h &=\bigl\{\bv_h\in \bC^0(\overline{\Ome});\,
	\bv_h|_\mathcal{K}\in \bP_2(\mathcal{K})~~\forall \mathcal{K}\in \cT_h \bigr\}, \\
	M_h &=\bigl\{\varphi_h\in C^0(\overline{\Ome});\, \varphi_h|_\mathcal{K}\in P_1(\mathcal{K})
	~~\forall \mathcal{K}\in \cT_h \bigr\}.
\end{align*}

Finite element approximation space $W_h$ for $\varpi$ variable can be chosen independently, any piecewise polynomial space is acceptable provided that
$W_h \supset M_h$, the most convenient choice is $W_h =M_h$.

Define
\begin{equation}\label{3.2}
	\bV_h:=\bigl\{\bv_h\in \bX_h;\,  (\bv_h,\br)=0\,\,
	\forall \br\in \bRM \bigr\},
\end{equation}
it is easy to check that $\bX_h=\bV_h\bigoplus \bRM$. It was proved in \cite{fh10}
that there holds the following inf-sup condition:
\begin{align}\label{3.3}
	\sup_{\bv_h\in \bV_h}\frac{(\div\bv_h,\varphi_h)}{\norm{\bv_h}{H^1(\Ome)}} 
	\geq \beta_1 \norm{\varphi_h}{L^2(\Ome)} \quad \forall \varphi_h\in M_{0h}, \quad \beta_1>0.
\end{align}

Also, we recall the following inverse inequality for polynomial functions \cite{bs08,brezzi,cia} :
\begin{align}
	\left\|\nabla\varphi_{h} \right\|_{L^{2}(\mathcal{K})}\leq c_{1}h^{-1}\left\|\varphi_{h} \right\|_{L^{2}(\mathcal{K})}~~~~~\forall\varphi_{h}\in P_{r}(\mathcal{K}), \mathcal{K}\in T_{h}.\label{3.4} 
\end{align}

Now, we give the fully discrete multiphysics finite element algorithm for the problem \reff{2.5}-\reff{2.7}.

{\bf Multiphysics Finite Element Algorithm (MFEA)} 
\begin{itemize}
	\item[(i)]
	Compute $\pmb\tau^0_h\in \bV_h$ and $q^0_h\in W_h$ by $\pmb\tau^0_h =\pmb\tau_0, p^0_h =p_0$.
	\item[(ii)] For $n=0,1,2, \cdots$,  do the following two steps.
	
	{\em Step 1:} Solve for $(\pmb\tau^{n+1}_h,\delta^{n+1}_h, \varpi^{n+1}_h)\in \bV_h\times M_h \times  W_h$ such that
	\begin{eqnarray}
		&\qquad\gamma\bigl(\varepsilon(\pmb\tau_{h}^{n+1}), \vepsi(\bv_h) \bigr)-\bigl( \delta^{n+1}_h, \div \bv_h \bigr)
		= (\mathbf{F}, \bv_h)+\langle \mathbf{F}_1,\bv_h\rangle &~ \forall \bv_h\in \bV_h, \label{3.5}\\
		&\chi_3\bigl(\delta^{n+1}_h, \varphi_h \bigr) +\bigl(\div\pmb\tau^{n+1}_h, \varphi_h \bigr)+\lambda^{*}\chi_3(d_{t}\div\pmb\tau_{h}^{n+1},\varphi_{h})&\label{3.6}\\
		&\hskip 1in=\chi_1\bigl( \varpi^{n+\theta}_h, \varphi_h \bigr)
		&~ \forall \varphi_h \in M_h, \no \\
		&\bigl(d_t\varpi^{n+1}_h, \psi_h \bigr)
		+\frac{1}{\theta_f} \bigl(K(\nab (\chi_1\delta^{n+1}_h +\chi_2\varpi^{n+1}_h &\label{3.7}\\
		&\hskip 0.6in
		+\lambda^{*}\chi_1d_{t}\div\pmb\tau_{h}^{n+1}-\rho_f\bg,\nab\psi_h \bigr)=(\phi, \psi_h)+\langle \phi_1,\psi_h\rangle &~  \forall \psi_h\in W_h,  \no
	\end{eqnarray}
	where $ d_{t}(\div\pmb\tau_{h}^{n+1})=\dfrac{\div\pmb\tau_{h}^{n+1}-\div\pmb\tau_{h}^{n}}{\Delta t},d_{t}(\varpi_{h}^{n+1})=\dfrac{\varpi_{h}^{n+1}-\varpi_{h}^{n}}{\Delta t} $ and $\theta=0,1$.
	
	{\em Step 2:} Update $p^{n+1}_h$ and $q^{n+1}_h$ by
	\begin{alignat}{2}
		&p^{n+1}_h=\chi_1\delta^{n+1}_h +\chi_2\varpi^{n+\theta}_h+\lambda^{*}\chi_1d_{t}\div_{h}^{n+1},\label{3.8}\\
		&q^{n+1}_h=\chi_1\varpi^{n+\theta}_h-\chi_3\delta^{n+1}_h-\lambda^{*}\chi_3d_{t}\div_{h}^{n+1}.\label{3.9}
	\end{alignat}
\end{itemize}
%\subsection{Stability analysis}
\begin{lemma}\label{lma3.1}
	Let $\{(\pmb\tau_h^{n}, \delta_h^{n}, \varpi_h^{n})\}_{n\geq 0}$ be defined by the (MFEA), then there hold
	\begin{alignat}{2}
		(\varpi^{n}_h, 1)&=C_\varpi(t_n) &&\qquad\mbox{for } n=0, 1, 2, \cdots,\label{3.10}\\
		(\delta^{n}_h, 1)&=C_\delta(t_{n-1+\theta}) =\dfrac{1}{\frac{d\lambda^{*}\chi_3}{\Delta t}+\chi_3\gamma+d}\left[ \frac{d\lambda^{*}\chi_3}{\Delta t}\bigl(\delta_h^{n-1}, 1\bigr) \right.\label{3.11}\\
		&\left.+\chi_1\gamma C_\varpi(t_{n-1+\theta})-(\mathbf{F},\bx)-\left\langle\mathbf{F}_{1},\bx\right\rangle\right] &&\qquad\mbox{for }  n=1-\theta, 1, 2, \cdots,  \no \\
		\langle\pmb\tau^{n}_h\cdot\bn, 1\rangle &=C_{\pmb\tau}(t_{n-1+\theta})
		&&\qquad\mbox{for } n=1-\theta, 1, 2, \cdots.\label{3.12}
	\end{alignat}
\end{lemma}
\begin{proof}
	Taking $\psi_h=1$ in \reff{3.7}, we have
	\begin{eqnarray}
		\bigl(d_t\varpi_h^{n+1}, 1\bigr) =(\phi,1) + \langle \phi_1, 1\rangle.\label{3.13}
	\end{eqnarray}
	Summing \reff{3.13} over $n$ from $0$ to $\ell \,(\geq 0)$, we get
	\begin{eqnarray}
		(\varpi_h^{\ell+1},1)=
		=(\varpi_0,1) + \bigl[(\phi,1) + \langle \phi_1, 1\rangle\bigr] t_{\ell+1}
		=C_{\varpi}(t_{\ell+1}), ~\ell=0,1,2,\cdots,
	\end{eqnarray}
which implies that \reff{3.10} holds.
	
	Taking $\bv_h=\bx$ in \reff{3.5} and $\varphi_h=1$ in \reff{3.6}, we get
	\begin{align}
		\gamma\bigl(\varepsilon(\pmb\tau_{h}^{n+1}), \mathbf{I} \bigr) -d\bigl( \delta_h^{n+1}, 1\bigr) 
		&=\bigl( \mathbf{F}, \bx\bigr)+ \langle \mathbf{F}_1,\bx\rangle,\label{3.15}\\
		\chi_3\bigl(\delta_h^{n+1}, 1\bigr) +\bigl( \div \pmb\tau_h^{n+1}, 1\bigr)+\lambda^{*}\chi_3(d_{t}\div\pmb\tau_{h}^{n+1},1) &=\chi_1 C_\varpi(t_{n+\theta}). \label{3.16}
	\end{align}
	Substituting \eqref{3.15} into \eqref{3.16}, we have
	\begin{eqnarray}
		&d\lambda^{*}\chi_3\bigl(d_{t}\delta_h^{n+1}, 1\bigr)+(\chi_3\gamma+d)(\delta_{h}^{n+1},1) = \chi_1\gamma C_\varpi(t_{n+\theta})-(\mathbf{F},\bx) -\left\langle\mathbf{F}_{1},\bx\right\rangle.\no  
	\end{eqnarray}
It is easy to check that
\begin{eqnarray}
	\left(\dfrac{d\lambda^{*}\chi_3}{\Delta t}+\chi_3\gamma+d\right) (\delta_{h}^{n+1},1)=\dfrac{d\lambda^{*}\chi_3}{\Delta t}\bigl(\delta_h^{n}, 1\bigr)+\chi_1\gamma C_\varpi(t_{n+\theta})-(\mathbf{F},\bx)-\left\langle\mathbf{F}_{1},\bx\right\rangle,\no
\end{eqnarray} 
	which implies \reff{3.11} holds for all $n\geq 1-\theta$. 
	
	Using \reff{3.10}, \reff{3.11}, \eqref{3.16} and Gauss divergence theorem, we deduce that \reff{3.12} holds.  The proof is complete.
\end{proof}
\begin{lemma}\label{lma3.3}
	Let $ \left\lbrace (\pmb\tau_{h}^{n},\delta_{h}^{n},\varpi_{h}^{n})\right\rbrace_{n\geq0}  $ be defined by the (MFEA), then there holds the following inequality:
	\begin{align}
		J_{h,\theta}^{l+1}+S_{h,\theta}^{l+1}= J_{h,\theta}^{0}~~~~~~~~~~for~l\geq0,~\theta=0,1, \label{3.17}
	\end{align}
	where
	\begin{align*}
		&J_{h,\theta}^{l+1}:=\dfrac{1}{2}\left[\gamma\left\|\varepsilon(\pmb\tau_{h}^{l+1}) \right\|_{L^{2}(\Omega)}^{2}+\chi_2\left\|\varpi_{h}^{l+\theta}\right\|_{L^{2}(\Omega)}^{2}+\chi_3\left\|\lambda^{*}d_{t}\div\pmb\tau_{h}^{l+1}+\delta_{h}^{l+1}\right\|_{L^{2}(\varOmega)}^{2}\right.\\
		&\left.+\dfrac{\gamma\lambda^{*}\chi_3\Delta t}{2}\left\|d_{t}\varepsilon(\pmb\tau_{h}^{l+1}) \right\|_{L^{2}(\Omega)}^{2}-2(\mathbf{F},\pmb\tau_{h}^{l+1})-2\left\langle\mathbf{F}_{1},\pmb\tau_{h}^{l+1} \right\rangle \right],\\
		&S_{h,\theta}^{l+1}:=\varDelta t\sum_{n=0}^{l}\left[\lambda^{*}\left\|d_{t}\div \pmb\tau_{h}^{n+1}\right\|_{L^{2}(\Omega)}^{2}+\dfrac{\gamma\varDelta t}{2}\left\|d_{t}\varepsilon(\pmb\tau_{h}^{n+1})\right\|_{L^{2}(\Omega)}^{2}\right.\\
		&+\dfrac{1}{\theta_f}(K\nabla p_{h}^{n+1}-K\rho_{f}g,\nabla p_{h}^{n+1})+\dfrac{\chi_2\varDelta t}{2}\left\|d_{t}\varpi_{h}^{n+\theta}\right\|_{L^{2}(\Omega)}^{2}+\dfrac{\chi_3\varDelta t}{2}\left\|d_{t}\delta_{h}^{n+1}\right\|_{L^{2}(\Omega)}^{2}  \\
		&+\dfrac{\chi_3\varDelta t}{2}\left\|\lambda^{*}d_{t}^{2}\div\pmb\tau_{h}^{n+1}\right\|_{L^{2}(\Omega)}^{2}+\dfrac{\gamma\lambda^{*}\chi_3(\Delta t)^{2}}{2}\left\|d_{t}^{2}\varepsilon(\pmb\tau_{h}^{n+1}) \right\|_{L^{2}(\Omega)}^{2}\no\\
		&-(\phi,p_{h}^{n+1})-\left\langle\phi_{1},p_{h}^{n+1}\right\rangle-(1-\theta)\dfrac{\chi_1\varDelta t}{\theta_f}(Kd_{t}\nabla\delta_{h}^{n+1},\nabla p_{h}^{n+1})\no\\
		&\left.-(1-\theta)\dfrac{\chi_1\lambda^{*}\varDelta t}{\theta_f}(Kd_{t}^{2}\nabla\div\pmb\tau_{h}^{n+1},\nabla p_{h}^{n+1})\right]. 
	\end{align*}
\end{lemma}
\begin{proof}
(i) When $ \theta=0 $, based on \reff{3.5}, we can  define $ \varpi_{h}^{-1} $ by
\begin{align}
	\chi_1(\varpi_{h}^{-1},\varphi_{h})=\chi_3(\delta_{h}^{0},\varphi_{h})+(\div\pmb\tau_{h}^{0},\varphi_{h})+\lambda^{*}\chi_{3}(d_{t}\div\pmb\tau_{h}^{0},\varphi_{h}).\label{3.18}
\end{align}
Setting $ \bv_{h}=d_{t}\pmb\tau_{h}^{n+1} $ in \reff{3.5}, $ \varphi_{h}=\delta_{h}^{n+1} $ in\reff{3.6} and $ \psi_{h}=p_{h}^{n+1} $ in \reff{3.7}  after lowing the super-index from $ n+1 $ to $ n $ on both sides of \reff{3.7}, we get
\begin{eqnarray}
	&\gamma(\varepsilon(\pmb\tau_{h}^{n+1}),\varepsilon(d_{t}\pmb\tau_{h}^{n+1}))-(\delta_{h}^{n+1},d_{t}\div \pmb\tau_{h}^{n+1})=\left(\mathbf{F},d_{t}\pmb\tau_{h}^{n+1}\right)+\left\langle \mathbf{F}_{1},d_{t}\pmb\tau_{h}^{n+1}\right\rangle,\label{3.19}\\
	&\qquad\chi_3(d_{t}\delta_{h}^{n+1},\delta_{h}^{n+1})+(d_{t}\div \pmb\tau_{h}^{n+1},\delta_{h}^{n+1})+\lambda^{*}\chi_3(d_{t}^{2}\div\pmb\tau_{h}^{n+1},\delta_{h}^{n+1})=\chi_1(d_{t}\varpi_{h}^{n},\delta_{h}^{n+1}),\label{3.20}\\
    &\qquad(d_{t}\varpi_{h}^{n},p_{h}^{n+1})+\dfrac{1}{\theta_f}(K(\nabla(\chi_1\delta_{h}^{n}+\chi_2\varpi_{h}^{n}+\lambda^{*}\chi_1d_{t}\div\pmb\tau_{h}^{n})-\rho_{f}\bg),\nabla p_{h}^{n+1})=(\phi,p_{h}^{n+1})+\left\langle\phi_{1},p_{h}^{n+1} \right\rangle.\label{3.21}
\end{eqnarray}
Adding \reff{3.19}-\reff{3.21}, we have
\begin{eqnarray}
	&\gamma(\varepsilon(\pmb\tau_{h}^{n+1}),\varepsilon(d_{t}\pmb\tau_{h}^{n+1}))+\chi_3(d_{t}\delta_{h}^{n+1},\delta_{h}^{n+1})+\lambda^{*}\chi_3(d_{t}^{2}\div\pmb\tau_{h}^{n+1},\delta_{h}^{n+1})\label{3.22}\\
	&+\chi_1(d_{t}\varpi_{h}^{n},\varpi_{h}^{n})+(d_{t}\varpi_{h}^{n},\lambda^{*}\chi_1d_{t}\div\pmb\tau_{h}^{n+1})+\dfrac{1}{\theta_f}(K\nabla p_{h}^{n+1}-K\rho_{f}g,\nabla p_{h}^{n+1})\no\\
	&-(1-\theta)\dfrac{\chi_1\varDelta t}{\theta_f}(Kd_{t}\nabla\delta_{h}^{n+1},\nabla p_{h}^{n+1})-(1-\theta)\dfrac{\chi_1\varDelta t}{\theta_f}(Kd_{t}^{2}\nabla\div\pmb\tau_{h}^{n+1},\nabla p_{h}^{n+1})\no\\
	&=\left(\mathbf{F},d_{t}\pmb\tau_{h}^{n+1}\right)+\left\langle \mathbf{F}_{1},d_{t}\pmb\tau_{h}^{n+1}\right\rangle+(\phi,p_{h}^{n+1})+\left\langle\phi_{1},p_{h}^{n+1} \right\rangle.\no
\end{eqnarray}
Using the equality $ d_{t}q^{n+1}_h=\chi_1d_{t}\varpi^{n}_h-d_{t}\chi_3\delta^{n+1}_h-\lambda^{*}d_{t}^{2}\chi_3\div_{h}^{n+1} $ to deal with the term $ \lambda^{*}\chi_3(d_{t}^{2}\div\pmb\tau_{h}^{n+1},\delta_{h}^{n+1}) $ and $ (d_{t}\varpi_{h}^{n+1},\lambda^{*}\chi_1d_{t}\div\pmb\tau_{h}^{n+1}) $ in \reff{3.22}, we have
\begin{eqnarray}
	&&\qquad\lambda^{*}\chi_3(d_{t}^{2}\div\pmb\tau_{h}^{n+1},\delta_{h}^{n+1})+(d_{t}\varpi_{h}^{n},\lambda^{*}\chi_1d_{t}\div\pmb\tau_{h}^{n+1})\label{3.23}\\
	&&=\lambda^{*}\chi_3d_{t}(d_{t}\div\pmb\tau_{h}^{n+1},\delta_{h}^{n+1})-\lambda^{*}\chi_3(d_{t}\div\pmb\tau_{h}^{n},d_{t}\delta_{h}^{n+1})+(d_{t}\varpi_{h}^{n},\lambda^{*}\chi_1d_{t}\div\pmb\tau_{h}^{n+1})\no\\
	&&=\lambda^{*}\chi_3d_{t}(d_{t}\div\pmb\tau_{h}^{n+1},\delta_{h}^{n+1})-\lambda^{*}\chi_3(d_{t}\div\pmb\tau_{h}^{n},d_{t}\delta_{h}^{n+1})+\lambda^{*}\chi_3(d_{t}\div\pmb\tau_{h}^{n+1},d_{t}\delta_{h}^{n+1})\no\\
	&&-\lambda^{*}\chi_3(d_{t}\div\pmb\tau_{h}^{n+1},d_{t}\delta_{h}^{n+1})+(d_{t}\varpi_{h}^{n},\lambda^{*}\chi_1d_{t}\div\pmb\tau_{h}^{n+1})\no\\
	&&=\lambda^{*}\chi_3d_{t}(d_{t}\div\pmb\tau_{h}^{n+1},\delta_{h}^{n+1})+\lambda^{*}\chi_3\Delta t(d_{t}^{2}\div\pmb\tau_{h}^{n+1},d_{t}\delta_{h}^{n+1})\no\\
	&&\qquad+\lambda^{*}(d_{t}\div\pmb\tau_{h}^{n+1},\chi_1d_{t}\varpi_{h}^{n}-\chi_3d_{t}\delta_{h}^{n+1})\no\\
	&&=\lambda^{*}\chi_3d_{t}(d_{t}\div\pmb\tau_{h}^{n+1},\delta_{h}^{n+1})+\lambda^{*}\chi_3\Delta t(d_{t}^{2}\div\pmb\tau_{h}^{n+1},d_{t}\delta_{h}^{n+1})\no\\
	&&\qquad+\lambda^{*}(d_{t}\div\pmb\tau_{h}^{n+1},d_{t}\div\pmb\tau_{h}^{n+1}+\lambda^{*}\chi_3d_{t}^{2}\div\pmb\tau_{h}^{n+1})\no\\
	&&=\lambda^{*}\chi_3d_{t}(d_{t}\div\pmb\tau_{h}^{n+1},\delta_{h}^{n+1})+\lambda^{*}\chi_3\Delta t(d_{t}^{2}\div\pmb\tau_{h}^{n+1},d_{t}\delta_{h}^{n+1})\no\\
	&&\qquad+\lambda^{*}(d_{t}\div\pmb\tau_{h}^{n+1},d_{t}\div\pmb\tau_{h}^{n+1})+\chi_3(\lambda^{*}d_{t}\div\pmb\tau_{h}^{n+1},\lambda^{*}d_{t}^{2}\div\pmb\tau_{h}^{n+1}).\no
\end{eqnarray}
Taking $ \bv_{h}=d_{t}^{2}\pmb\tau_{h}^{n+1} $ in \reff{3.5}, we get 
\begin{eqnarray}
	\gamma\Delta t(\varepsilon(d_{t}\pmb\tau_{h}^{n+1}),\varepsilon(d_{t}^{2}\pmb\tau_{h}^{n+1}))-\Delta t(d_{t}\delta_{h}^{n+1},d_{t}^{2}\div\pmb\tau_{h}^{n+1})=0.\label{3.24}
\end{eqnarray}
Using the equality $ (d_{t}a_{h}^{n+1},a_{h}^{n+1})=\dfrac{\Delta t}{2}\left\| d_{t}a_{h}^{n+1}\right\|_{L^{2}(\Omega)}^{2}+\dfrac{1}{2}d_{t}\left\| a_{h}^{n+1}\right\|_{L^{2}(\Omega)}^{2}$, \reff{3.23} and \reff{3.24} in\reff{3.22}, where $ a=\varepsilon(\pmb\tau),\delta,\varpi$ and $\div\pmb\tau $, we get
\begin{eqnarray}
	&&\qquad\dfrac{\gamma}{2}d_{t}\left\|\varepsilon(\pmb\tau_{h}^{n+1}) \right\|_{L^{2}(\Omega)}^{2}+\dfrac{\gamma\Delta t}{2}\left\|d_{t}\varepsilon(\pmb\tau_{h}^{n+1}) \right\|_{L^{2}(\Omega)}^{2}+\dfrac{\chi_3\varDelta t}{2}\left\|d_{t}\delta_{h}^{n+1} \right\|_{L^{2}(\Omega)}^{2}\label{3.25}\\
	&&~+\dfrac{\chi_3\varDelta t}{2}\left\|\lambda^{*}d_{t}^{2}\div\pmb\tau_{h}^{n+1} \right\|_{L^{2}(\Omega)}^{2}+\dfrac{\chi_3}{2}d_{t}\left\|\lambda^{*}d_{t}\div\pmb\tau_{h}^{n+1}+\delta_{h}^{n+1} \right\|_{L^{2}(\Omega)}^{2}\no\\
	&&~+\dfrac{\chi_2\varDelta t}{2}\left\|d_{t}\varpi_{h}^{n} \right\|_{L^{2}(\Omega)}^{2}+\dfrac{\chi_2}{2}d_{t}\left\|\varpi_{h}^{n} \right\|_{L^{2}(\Omega)}^{2}+\lambda^{*}\left\|d_{t}\div\pmb\tau_{h}^{n+1}\right\|_{L^{2}(\Omega)}^{2}\no\\
	&&~+\dfrac{\gamma\lambda^{*}\chi_3\Delta t}{2}d_{t}\left\|d_{t}\varepsilon(\pmb\tau_{h}^{n+1}) \right\|_{L^{2}(\Omega)}^{2}+\dfrac{\gamma\lambda^{*}\chi_3(\Delta t)^{2}}{2}\left\|d_{t}^{2}\varepsilon(\pmb\tau_{h}^{n+1}) \right\|_{L^{2}(\Omega)}^{2}\no\\
	&&~+\dfrac{1}{\theta_f}(K\nabla p_{h}^{n+1}-K\rho_{f}\bg,\nabla p_{h}^{n+1})-(1-\theta)\dfrac{\chi_1\varDelta t}{\theta_f}(Kd_{t}\nabla\delta_{h}^{n+1},\nabla p_{h}^{n+1})\no\\
	&&~-(1-\theta)\dfrac{\chi_1\varDelta t}{\theta_f}(Kd_{t}^{2}\nabla\div\pmb\tau_{h}^{n+1},\nabla p_{h}^{n+1})\no\\
	&&=\left(\mathbf{F},d_{t}\pmb\tau_{h}^{n+1}\right)+\left\langle \mathbf{F}_{1},d_{t}\pmb\tau_{h}^{n+1}\right\rangle+(\phi,p_{h}^{n+1})+\left\langle\phi_{1},p_{h}^{n+1} \right\rangle\no
\end{eqnarray}
and applying the summation operator $ \Delta t\sum_{n=0}^{l} $ to the both sides of \reff{3.25}, we see that \reff{3.17}  holds for $ \theta=0$.

(ii) When $ \theta=1 $, we can prove \reff{3.17} by using the similar process for the case $ \theta=0 $, so here we omit the more details. The proof is complete.
\end{proof}

\begin{lemma}
	Let $ \left\lbrace (\pmb\tau_{h}^{n},\delta_{h}^{n},\varpi_{h}^{n})\right\rbrace_{n\geq0}  $ be defined by the (MFEA) with $ \theta=0 $, then there holds the following inequality:
	\begin{align}
		J_{h,0}^{l+1}+\hat{S}_{h,0}^{l+1}\leq J_{h,0}^{0},~~~~~~~~~~for~l\geq0, \label{3.26}
	\end{align}
	provided that $ \varDelta t=O(h^{2}) $. Here
	\begin{align*}
		&\hat{S}_{h,0}^{l+1}:=\varDelta t\sum_{n=0}^{l}\left[\lambda^{*}\left\|d_{t}\div \pmb\tau_{h}^{n+1}\right\|_{L^{2}(\Omega)}^{2}+\dfrac{\gamma\varDelta t}{4}\left\|d_{t}\varepsilon(\pmb\tau_{h}^{n+1})\right\|_{L^{2}(\Omega)}^{2}\right.\\
		&+\dfrac{1}{\theta_f}(K\nabla p_{h}^{n+1}-K\rho_{f}\bg,\nabla p_{h}^{n+1})+\dfrac{\chi_2\varDelta t}{2}\left\|d_{t}\varpi_{h}^{n}\right\|_{L^{2}(\Omega)}^{2}  +\dfrac{\chi_3\varDelta t}{2}\left\|d_{t}\delta_{h}^{n+1} \right\|_{L^{2}(\Omega)}^{2}\\
		&~\left.+ \dfrac{\chi_3\varDelta t}{4} \left\|\lambda^{*}d_{t}^{2}\div\pmb\tau_{h}^{n+1} \right\|_{L^{2}(\Omega)}^{2}+ \dfrac{\gamma\lambda^{*}\chi_3}{2} \left\|d_{t}^{2}\varepsilon(\pmb\tau_{h}^{n+1}) \right\|_{L^{2}(\Omega)}^{2}-(\phi,p_{h}^{n+1})-\left\langle\phi_{1},p_{h}^{n+1}\right\rangle\right]. 
	\end{align*}
\end{lemma}
\begin{proof}
	By Cauchy-Schwarz inequality and inverse inequality \reff{3.4}, we get
	\begin{eqnarray}
		&&\dfrac{\chi_1\varDelta t}{\theta_f}(Kd_{t}\nabla\delta_{h}^{n+1},\nabla p_{h}^{n+1})\label{3.27}\\
		&&\leq \dfrac{K_{2}^{2}\chi_1^{2}}{K_{1}\theta_f}\left\|\nabla\delta_{h}^{n+1}-\nabla\delta_{h}^{n} \right\|_{L^{2}(\Omega)}^{2}+\dfrac{K_{1}}{4\theta_f}\left\|\nabla p_{h}^{n+1} \right\|_{L^{2}(\Omega)}^{2}\no\\
		&&\leq\dfrac{K_{2}^{2}\chi_1^{2}}{K_{1}\theta_fh^{2}}\left\|\delta_{h}^{n+1}-\delta_{h}^{n} \right\|_{L^{2}(\Omega)}^{2}+\dfrac{K_{1}}{4\theta_f}\left\|\nabla p_{h}^{n+1} \right\|_{L^{2}(\Omega)}^{2},\nonumber\\
		&&\dfrac{\chi_1\lambda^{*}\varDelta t}{\theta_f}(K\nabla d_{t}^{2}\div\pmb\tau_{h}^{n+1},\nabla p_{h}^{n+1})\label{3.28}\\
		&&\leq\dfrac{\chi_1K_{2}\varDelta t}{\theta_f}\left\|\lambda^{*}\nabla d_{t}^{2}\div\pmb\tau_{h}^{n+1} \right\|_{L^{2}(\Omega)}\left\|\nabla p_{h}^{n+1} \right\|_{L^{2}(\Omega)}\no\\
		&&\leq \dfrac{K_{2}^{2}\chi_1^{2}(\Delta t)^{2}}{K_{1}\theta_f h^{2}}\left\|\lambda^{*}d_{t}^{2}\div\pmb\tau_{h}^{n+1} \right\|_{L^{2}(\Omega)}^{2}+\dfrac{K_{1}}{4\theta_f}\left\|\nabla p_{h}^{n+1} \right\|_{L^{2}(\Omega)}^{2}.\no
	\end{eqnarray}
	To bound the first term on the right-hand side of \reff{3.27}, using \reff{3.3}, we have
	\begin{align}
		&\left\|\delta_{h}^{n+1}-\delta_{h}^{n}\right\|_{L^{2}(\Omega)}\leq\dfrac{1}{\beta_{1}}\sup_{\bv_{h}\in\bV_{h}}\dfrac{(\div\bv_{h},\delta_{h}^{n+1}-\delta_{h}^{n})}{\left\|\nabla\bv_{h} \right\|_{L^{2}(\Omega)} }\label{3.29}\\
		&\leq\dfrac{1}{\beta_{1}}\sup_{\bv_{h}\in\bV_{h}}\dfrac{\gamma(\varepsilon(\pmb\tau_{h}^{n+1}-\pmb\tau_{h}^{n}),\varepsilon(\bv_{h}))}{\left\|\nabla\bv_{h} \right\|_{L^{2}(\Omega)} }\nonumber\\
		&\leq\frac{\gamma\varDelta t}{\beta_{1}}\left\|d_{t}\varepsilon(\pmb\tau^{n+1}_{h}) \right\|_{L^{2}(\Omega)}.\nonumber
	\end{align}
	Substituting \reff{3.29} into \reff{3.27} and combining it with \reff{3.17} and \reff{3.28} imply that \reff{3.26} holds if $\varDelta t\leq \min\left\lbrace \dfrac{K_{1}\theta_f\beta_{1}^{2}h^{2}}{4\gamma\chi_1^{2}K_{2}^{2}},\dfrac{K_{1}\theta_f\chi_{3}h^{2}}{4\chi_1^{2}K_{2}^{2}}\right\rbrace  $. The proof is complete.
\end{proof}

\subsection{Error estimates}\label{sec-3.3}
To derive the optimal order error estimates of the fully discrete multiphysics finite element method
for any $\varphi \in L^{2}(\Omega)$, we firstly define $L^{2}(\Omega)$-projection operators $\mathcal{Q}_{h}: L^{2}(\Omega)\rightarrow X^{k}_{h}$  by
\begin{eqnarray}
	(\mathcal{Q}_{h}\varphi,\psi_{h})=(\varphi,\psi_{h})~~~~~\psi_{h}\in X^{k}_{h},\label{3.30}
\end{eqnarray}
where $X^{k}_{h}:=\{\psi_{h}\in C^{0};~\psi_{h}|_{E}\in P_{k}(E)~\forall E\in\mathcal{T}_{h}\}$, $ k $ is the degree of piecewise polynomial on E.

Next, for any $\varphi\in H^{1}(\Omega)$, we define its elliptic projection $\mathcal{S}_{h}: H^{1}(\Omega)\rightarrow X^{k}_{h}$ by
\begin{align}
	(K\nabla \mathcal{S}_{h}\varphi,&\nabla\varphi_{h})=(K\nabla\varphi,\nabla\varphi_{h})~~~~~\forall \varphi_{h}\in X^{k}_{h},\\
	&(\mathcal{S}_{h}\varphi,1)=(\varphi,1).
\end{align}
Finally, for any $\textbf{v}\in\textbf{H}^{1}(\Omega)$, we define its elliptic projection $\mathcal{R}_{h}: \textbf{H}^{1}(\Omega)\rightarrow\textbf{V}^{k}_{h}$ by
\begin{eqnarray}
	(\varepsilon(\mathcal{R}_{h}\textbf{v}),\varepsilon(\textbf{w}_{h}))=(\varepsilon(\textbf{v}),\varepsilon(\textbf{w}_{h}))~~~~~
	\forall\textbf{w}_{h}\in \textbf{V}^{k}_{h},
\end{eqnarray}
where $\textbf{V}^k_{h}:=\{\textbf{v}_{h} \in \textbf{C}^{0};~\textbf{v}_{h}|_{\mathcal{K}}\in \textbf{P}_{k}(\mathcal{K}), (\textbf{v}_{h}, \textbf{r})=0 ~\forall \textbf{r} \in \textbf{RM}\}$, $k$ is the degree of the piecewise polynomial on $\mathcal{K}$. From \cite{bs08}, we know that  $\mathcal{Q}_{h},\mathcal{S}_{h}$ and $\mathcal{R}_{h}$ satisfy
\begin{align}
	\|\mathcal{Q}_{h}\varphi-\varphi\|_{L^{2}(\Omega)}+&h\|\nabla(\mathcal{Q}_{h}\varphi-\varphi)\|_{L^{2}(\Omega)}\label{3.34}\\
	&\leq Ch^{s+1}\|\varphi\|_{H^{s+1}(\Omega)}~~\forall \varphi\in H^{s+1}(\Omega),~~  0\leq s\leq k,\no\\ \|\mathcal{S}_{h}\varphi-\varphi\|_{L^{2}(\Omega)}+&h\|\nabla(\mathcal{S}_{h}\varphi-\varphi)\|_{L^{2}(\Omega)}\label{3.35}\\ &\leq Ch^{s+1}\|\varphi\|_{H^{s+1}(\Omega)}~~\forall \varphi\in H^{s+1}(\Omega),~~  0\leq s\leq k,\no\\
	\|\mathcal{R}_{h}\textbf{v}-\textbf{v}\|_{L^{2}(\Omega)}+&h\|\nabla(\mathcal{R}_{h}\textbf{v}-\textbf{v})\|_{L^{2}(\Omega)}\label{3.36}\\
	&\leq Ch^{s+1}\|\textbf{v}\|_{H^{s+1}(\Omega)}~~\forall \textbf{v}\in \textbf{H}^{s+1}(\Omega),~~  0\leq s\leq k.\no
\end{align}
To derive the error estimates, we introduce the following notations
\begin{align*}	&E_{\pmb\tau}^{n}=\pmb\tau(t_{n})-\pmb\tau_{h}^{n},~~~E_{\delta}^{n}=\delta(t_{n})-\delta_{h}^{n},~~~~E_{\varpi}^{n}=\varpi(t_{n})-\varpi_{h}^{n},~~~~\\
	&E_{p}^{n}=p(t_{n})-p_{h}^{n},~~~~E_{q}^{n}=q(t_{n})-q_{h}^{n}.	
\end{align*}
It is easy to check out
\begin{align}
	E_{p}^{n}=\chi_1E_{\delta}^{n}+\chi_2E_{\varpi}^{n}+\lambda^{*}\chi_1d_{t}\div E_{\pmb\tau}^{n},~~~~E_{q}^{n}=\chi_1E_{\varpi}^{n}-\chi_3E_{\delta}^{n}-\lambda^{*}\chi_3d_{t}\div E_{\pmb\tau}^{n}.
\end{align}
Also, we denote
\begin{align*}
	&E_{\pmb\tau}^{n}=\pmb\tau(t_{n})-\mathcal{R}_{h}(\pmb\tau(t_{n}))+\mathcal{R}_{h}(\pmb\tau(t_{n}))-\pmb\tau_{h}^{n}:=Y_{\pmb\tau}^{n}+Z_{\pmb\tau}^{n},\\
	&E_{\delta}^{n}=\delta(t_{n})-\mathcal{S}_{h}(\delta(t_{n}))+\mathcal{S}_{h}(\delta(t_{n}))-\delta_{h}^{n}:=Y_{\delta}^{n}+Z_{\delta}^{n},\\
	&E_{\varpi}^{n}=\varpi(t_{n})-\mathcal{S}_{h}(\varpi(t_{n}))+\mathcal{S}_{h}(\varpi(t_{n}))-\varpi_{h}^{n}:=Y_{\varpi}^{n}+Z_{\varpi}^{n},\\
	&E_{p}^{n}=p(t_{n})-\mathcal{S}_{h}(p(t_{n}))+\mathcal{S}_{h}(p(t_{n}))-p_{h}^{n}:=Y_{p}^{n}+Z_{p}^{n},\\
	&E_{\delta}^{n}=\delta(t_{n})-\mathcal{Q}_{h}(\delta(t_{n}))+\mathcal{Q}_{h}(\delta(t_{n}))-\delta_{h}^{n}:=F_{\delta}^{n}+G_{\delta}^{n},\\
	&E_{\varpi}^{n}=\varpi(t_{n})-\mathcal{Q}_{h}(\varpi(t_{n}))+\mathcal{Q}_{h}(\varpi(t_{n}))-\varpi_{h}^{n}:=F_{\varpi}^{n}+G_{\varpi}^{n},	\\
	&E_{p}^{n}=p(t_{n})-\mathcal{Q}_{h}(p(t_{n}))+\mathcal{S}_{h}(p(t_{n}))-p_{h}^{n}:=F_{p}^{n}+G_{p}^{n}.
\end{align*}
It is easy to check that $
	G_{p}^{n}=\chi_1G_{\delta}^{n}+\chi_2G_{\varpi}^{n}+\lambda^{*}\chi_1d_{t}\div Z_{\pmb\tau}^{n}, Z_{p}^{n}=\chi_1Z_{\delta}^{n}+\chi_2Z_{\varpi}^{n}+\lambda^{*}\chi_1d_{t}\div \textsc{Z}_{\pmb\tau}^{n},  G_{q}^{n}=Z_{q}^{n}=\div Z_{\pmb\tau}^{n}:=\chi_1G_{\varpi}^{n}-\chi_3G_{\delta}^{n}-\lambda^{*}\chi_3d_{t}\div Z_{\pmb\tau}^{n}$.
\begin{lemma}
	Let $ \left\lbrace (\pmb\tau_{h}^{n}, \delta_{h}^{n}, \varpi_{h}^{n})\right\rbrace_{n\geq0}  $ be generated by the (MFEA) and $Y_{\pmb\tau}^{n}, Z_{\pmb\tau}^{n}, Y_{\delta}^{n}, Z_{\delta}^{n}, Y_{\varpi}^{n}$ and $Z_{\varpi}^{n}$ be defined as above. Then there holds
	\begin{eqnarray}
		&&\quad\qquad\mathcal{E}_{h}^{l+1}+\Delta t\sum_{n=0}^{l}\left[\dfrac{\Delta t}{2}\left(\chi_3\left\|\lambda^{*}d_{t}^{2}\div Z_{\pmb\tau}^{n+1} \right\|_{L^2(\Omega)}^{2}+\gamma\left\|d_{t}\varepsilon(Z_{\pmb\tau}^{n+1}) \right\|_{L^2(\Omega)}^{2}\right.\right.\label{3.38}\\
		&&\left.+\chi_3\left\| d_{t}G_{\delta}^{n+1} \right\|_{L^2(\Omega)}^{2}+\chi_2\left\| d_{t}G_{\varpi}^{n+\theta} \right\|_{L^2(\Omega)}^{2}+\gamma\chi_3\lambda^{*}\Delta t\left\|d_{t}^{2}\varepsilon(Z_{\pmb\tau}^{n+1}) \right\|_{L^{2}(\Omega)}^{2}\right)\no\\
		&&\left.+\lambda^{*}\left\|d_{t}\div Z_{\pmb\tau}^{n+1} \right\|_{L^{2}(\Omega)}^{2}+\dfrac{1}{\theta_f}(K(\nabla \hat{Z}_{p}^{n+1},\nabla\hat{Z}_{p}^{n+1})\right] \no\\
		&&=\mathcal{E}_{h}^{0}+\Delta t\sum_{n=0}^{l}\left[\chi_{3}\lambda^{*}\Delta t(d_{t}F_{\delta}^{n+1},\div d_{t}^{2}Z_{\pmb\tau}^{n+1})+ (F_{\delta}^{n+1},\div d_{t}Z_{\pmb\tau}^{n+1})+\lambda^{*}\chi_3(\mathbf{R}_{h}^{n+1},G_{\delta}^{n+1})\right. \no\\
		&&\left.-(\div d_{t}Y_{\pmb\tau}^{n+1},G_{\delta}^{n+1})\right]+\Delta t\sum_{n=0}^{l}\left[(R_{h}^{n+\theta},\hat{Z}_{p}^{n+1})+\chi_1(1-\theta)\varDelta t(d_{t}^{2}\varpi(t_{n+1}),G_{\delta}^{n+1}) \right. \no\\
		&&\left.+(1-\theta)\dfrac{\chi_1\varDelta t}{\theta_f}\left(Kd_{t}\nabla Z_{\delta}^{n+1},\nabla\hat{Z}_{p}^{n+1}\right)\right]+\varDelta t\sum_{n=0}^{l}(d_{t}G_{\varpi}^{n+\theta},Y_{p}^{n+1}-F_{p}^{n+1})\no\\
		&&+\varDelta t\sum_{n=0}^{l}\left[(1-\theta)\dfrac{\chi_1\varDelta t}{\theta_f}\left(Kd_{t}\nabla \div Z_{\pmb\tau}^{n+1},\nabla\hat{Z}_{p}^{n+1}  \right)-\lambda^{*}\chi_3(d_{t}^{2}\div Y_{\pmb\tau}^{n+1},G_{\delta}^{n+1}) \right],\no 
	\end{eqnarray}
	where
	\begin{align*}
		&\hat{Z}_{p}^{n+1}:=F_{p}^{n+1}-Y_{p}^{n+1}+\chi_1G_{\delta}^{n+1}+\chi_2G_{\varpi}^{n+\theta},\\
		&\mathcal{E}_{h}^{l+1}:=\dfrac{1}{2}\left[\gamma\left\| \varepsilon(Z_{\pmb\tau}^{l+1}) \right\|_{L^2(\Omega)}^{2}+\chi_2\left\|G_{\varpi}^{l+\theta} \right\|_{L^{2}(\varOmega)}^{2}+\chi_3\left\|\lambda^{*}d_{t}\div Z_{\pmb\tau}^{l+1}+G_{\delta}^{l+1} \right\|_{L^{2}(\varOmega)}^{2}\right.\\
		&\left.+\gamma\chi_3\lambda^{*}\Delta t\left\|d_{_{t}}\varepsilon(Z_{\pmb\tau}^{l+1}) \right\|_{L^{2}(\Omega)}^{2}  \right], \\
		&R_{h}^{n+1}:=-\dfrac{1}{\varDelta t}\int_{t_{n}}^{t_{n+1}}(s-t_{n})\varpi_{tt}(s)ds,\\
		&\mathbf{R}_{h}^{n+1}:=-\dfrac{1}{\varDelta t}\int_{t_{n}}^{t_{n+1}}(s-t_{n})\div\pmb\tau_{tt}(s)ds.
	\end{align*}
\end{lemma}
\begin{proof}
	Subtracting \reff{3.5} from \reff{2.18}, \reff{3.6} from \reff{2.19}, \reff{3.7}  from \reff{2.20}, respectively, we get
	\begin{eqnarray}
		&&\gamma(\varepsilon(E_{\pmb\tau}^{n+1}),\varepsilon(\bv_{h}))-(E_{\delta}^{n+1},\div\bv_{h})=0 \quad  \forall~\bv_{h}\in \bV_{h},\\
		&&\chi_3(E_{\delta}^{n+1},\varphi_{h})+(\div E_{\pmb\tau}^{n+1},\varphi_{h})+\lambda^{*}\chi_3(d_{t}\div E_{\pmb\tau}^{n+1},\div\bv_{h})\\
		&&=\chi_1(E_{\varpi}^{n+\theta},\varphi_{h})+\lambda^{*}\chi_3(\mathbf{R}_{h}^{n+1},\div\bv_{h})+\chi_1(1-\theta)\varDelta t(d_{t}\varpi(t_{n+1}),\varphi_{h}) \quad\forall \varphi_{h}\in M_{h},\no\\
		&&(d_{t}E_{\varpi}^{n+\theta},\psi_{h})+\dfrac{1}{\theta_f}(K(\nabla E_{p}^{n+1},\nabla\psi_{h})-(1-\theta)\dfrac{\chi_1\varDelta t}{\theta_f}\left(Kd_{t}\nabla E_{\delta}^{n+1},\nabla\psi_{h}  \right) \\
		&&~~-(1-\theta)\dfrac{\chi_1\lambda^{*}\varDelta t}{\theta_f}(Kd_{t}^{2}\div E_{\pmb\tau}^{n+1},\nabla \psi_{h})=(R_{h}^{n+\theta},\psi_{h}) \quad\forall\psi_{h}\in W_{h},\no\\
		&&E_{\pmb\tau}^{0}=0, E_{\delta}^{0}=0, E_{\varpi}^{-1}=0.
	\end{eqnarray}
Using the definitions of the projection operators $ \mathcal{Q}_{h},\mathcal{S}_{h},\mathcal{R}_{h}$, we have
\begin{eqnarray}
	&&\gamma(\varepsilon(Z_{\pmb\tau}^{n+1}),\varepsilon(\bv_{h}))-(G_{\delta}^{n+1},\div\bv_{h})=(F_{\delta}^{n+1},\div\bv_{h}) \quad  \forall~\bv_{h}\in \bV_{h},\label{3.43}\\
	&&\chi_3(G_{\delta}^{n+1},\varphi_{h})+(\div Z_{\pmb\tau}^{n+1},\varphi_{h})+\lambda^{*}\chi_3(d_{t}\div Z_{\pmb\tau}^{n+1},\div\bv_{h})\label{3.44}\\
	&&=\chi_1(G_{\varpi}^{n+\theta},\varphi_{h})-(\div Y_{\pmb\tau}^{n+1},\varphi_{h})-\lambda^{*}\chi_3(d_{t}\div Y_{\pmb\tau}^{n+1},\div\bv_{h})\no\\
	&&~~+\lambda^{*}\chi_3(\mathbf{R}_{h}^{n+1},\div\bv_{h})+\chi_1(1-\theta)\varDelta t(d_{t}\varpi(t_{n+1}),\varphi_{h}) \quad\forall \varphi_{h}\in M_{h},\no\\
	&&(d_{t}G_{\varpi}^{n+\theta},\psi_{h})+\dfrac{1}{\theta_f}(K(\nabla \hat{Z}_{p}^{n+1},\nabla\psi_{h})-(1-\theta)\dfrac{\chi_1\varDelta t}{\theta_f}\left(Kd_{t}\nabla Z_{\delta}^{n+1},\nabla\psi_{h}  \right)\label{3.45} \\
	&&~~-(1-\theta)\dfrac{\chi_1\lambda^{*}\varDelta t}{\theta_f}(Kd_{t}^{2}\nabla\div Z_{\pmb\tau}^{n+1},\nabla \psi_{h})=(R_{h}^{n+\theta},\psi_{h}) \quad\forall\psi_{h}\in W_{h}.\no
\end{eqnarray}
Taking $ \bv_{h}=d_{t}Z_{\pmb\tau}^{n+1}$ in \reff{3.43}, $\varphi_{h}=G_{\delta}^{n+1} $ after applying the difference operator $ d_{t} $ to \reff{3.44} and $ \psi_{h}=\hat{Z}_{p}^{n+1}=F_{p}^{n+1}-Y_{p}^{n+1}+\chi_1G_{\delta}^{n+1}+\chi_2G_{\varpi}^{n+\theta}+\lambda^{*}\chi_1d_{t}\div Z_{\pmb\tau}^{n+1} $ in \reff{3.45} , we have
\begin{eqnarray}
	&&\quad\quad~~~\gamma(\varepsilon(Z_{\pmb\tau}^{n+1}),d_{t}\varepsilon(Z_{\pmb\tau}^{n+1}))-(G_{\delta}^{n+1},\div d_{t}Z_{\pmb\tau}^{n+1})=(F_{\delta}^{n+1},\div d_{t}Z_{\pmb\tau}^{n+1}),\label{3.46}\\
	&&\quad\quad~~~\chi_3(d_{t}G_{\delta}^{n+1},G_{\delta}^{n+1})+(d_{t}\div Z_{\pmb\tau}^{n+1},G_{\delta}^{n+1})+\lambda^{*}\chi_3(d_{t}^{2}\div Z_{\pmb\tau}^{n+1},G_{\delta}^{n+1})\label{3.47}\\
	&&=\chi_1(d_{t}G_{\varpi}^{n+\theta},G_{\delta}^{n+1})-(d_{t}\div Y_{\pmb\tau}^{n+1},G_{\delta}^{n+1})-\lambda^{*}\chi_3(d_{t}^{2}\div Y_{\pmb\tau}^{n+1},G_{\delta}^{n+1})\no\\
	&&~~+\lambda^{*}\chi_3(\mathbf{R}_{h}^{n+1},G_{\delta}^{n+1})+\chi_1(1-\theta)\varDelta t(d_{t}\varpi(t_{n+1}),G_{\delta}^{n+1}),\no\\
	&&\quad\quad~~~(d_{t}G_{\varpi}^{n+\theta},\hat{Z}_{p}^{n+1})-(1-\theta)\dfrac{\chi_1\varDelta t}{\theta_f}\left(Kd_{t}\nabla Z_{\delta}^{n+1},\nabla\hat{Z}_{p}^{n+1}  \right)\label{3.48} \\
	&&-(1-\theta)\dfrac{\chi_1\lambda^{*}\varDelta t}{\theta_f}(Kd_{t}^{2}\nabla\div Z_{\pmb\tau}^{n+1},\nabla \hat{Z}_{p}^{n+1})+\dfrac{1}{\theta_f}(K(\nabla \hat{Z}_{p}^{n+1},\nabla\hat{Z}_{p}^{n+1})=(R_{h}^{n+\theta},\hat{Z}_{p}^{n+1}).\no
\end{eqnarray}
Adding \reff{3.46}-\reff{3.48} and using the identity  $ (d_{t}a_{h}^{n+1},a_{h}^{n+1})=\dfrac{\Delta t}{2}\left\| d_{t}a_{h}^{n+1}\right\|_{L^{2}(\Omega)}^{2}+\dfrac{1}{2}d_{t}\left\| a_{h}^{n+1}\right\|_{L^{2}(\Omega)}^{2}$ for the result equation, where
$ a_{h}^{n+1}=\varepsilon(Z_{\pmb\tau}^{n+1}),G_{\delta}^{n+1},G_{\varpi}^{n+\theta} $, we get
\begin{eqnarray}
	&&\qquad\qquad\dfrac{\gamma\Delta t}{2}\left\|d_{t}\varepsilon(Z_{\pmb\tau}^{n+1}) \right\|_{L^{2}(\Omega)}^{2}+ \dfrac{\gamma}{2}d_{t}\left\|\varepsilon(Z_{\pmb\tau}^{n+1}) \right\|_{L^{2}(\Omega)}^{2}+\dfrac{\chi_3\Delta t}{2}\left\|d_{t}G_{\delta}^{n+1} \right\|_{L^{2}(\Omega)}^{2}\label{3.49}\\
	&&+\dfrac{\chi_3}{2}d_{t}\left\|G_{\delta}^{n+1} \right\|_{L^{2}(\Omega)}^{2}+\lambda^{*}\chi_3(d_{t}^{2}\div Z_{\pmb\tau}^{n+1},G_{\delta}^{n+1})+\dfrac{\chi_2\Delta t}{2}\left\|d_{t}G_{\varpi}^{n+1} \right\|_{L^{2}(\Omega)}^{2}\no\\
	&&+\dfrac{\chi_2}{2}d_{t}\left\|G_{\varpi}^{n+\theta} \right\|_{L^{2}(\Omega)}^{2}+(d_{t}G_{\varpi}^{n+\theta},\lambda^{*}\chi_1d_{t}\div Z_{\pmb\tau}^{n+1})+\dfrac{1}{\theta_f}(K(\nabla \hat{Z}_{p}^{n+1},\nabla\hat{Z}_{p}^{n+1})\no\\
	&&=(F_{\delta}^{n+1},\div d_{t}Z_{\pmb\tau}^{n+1})-(d_{t}\div Y_{\pmb\tau}^{n+1},G_{\delta}^{n+1})-\lambda^{*}\chi_3(d_{t}^{2}\div Y_{\pmb\tau}^{n+1},G_{\delta}^{n+1})\no\\
	&&+\chi_1(1-\theta)\varDelta t(d_{t}\varpi(t_{n+1}),G_{\delta}^{n+1})+(1-\theta)\dfrac{\chi_1\varDelta t}{\theta_f}\left(Kd_{t}\nabla Z_{\delta}^{n+1},\nabla\hat{Z}_{p}^{n+1}  \right)\no\\
	&&+(1-\theta)\dfrac{\chi_1\lambda^{*}\varDelta t}{\theta_f}(Kd_{t}^{2}\nabla\div Z_{\pmb\tau}^{n+1},\nabla \hat{Z}_{p}^{n+1})+(R_{h}^{n+\theta},\hat{Z}_{p}^{n+1})+(d_{t}G_{\varpi}^{n+\theta},Y_{p}^{n+1}-F_{p}^{n+1})\no\\
	&&+\lambda^{*}\chi_3(\mathbf{R}_{h}^{n+1},G_{\delta}^{n+1}).\no
\end{eqnarray}
Using the equality $ d_{t}G_{q}^{n}=\chi_1d_{t}G_{\varpi}^{n}-\chi_3d_{t}G_{\delta}^{n}-\lambda^{*}\chi_3d_{t}^{2}\div Z_{\pmb\tau}^{n} $ to deal with the term $ \lambda^{*}\chi_3(d_{t}^{2}\div Z_{\pmb\tau}^{n+1},G_{\delta}^{n+1}) $ and $ (d_{t}G_{\varpi}^{n},\lambda^{*}\chi_1d_{t}\div Z_{\pmb\tau}^{n+1}) $ in \reff{3.49}, we have
\begin{eqnarray}
	&&\qquad\qquad\lambda^{*}\chi_3(d_{t}^{2}\div Z_{\pmb\tau}^{n+1},G_{\delta}^{n+1})+ (d_{t}G_{\varpi}^{n},\lambda^{*}\chi_1d_{t}\div Z_{\pmb\tau}^{n+1})\label{3.50}\\
	&&=\lambda^{*}\chi_3d_{t}(d_{t}\div Z_{\pmb\tau}^{n+1},G_{\delta}^{n+1})-\lambda^{*}\chi_3(d_{t}\div Z_{\pmb\tau}^{n},d_{t}G_{\delta}^{n+1})+ (d_{t}G_{\varpi}^{n},\lambda^{*}\chi_1d_{t}\div Z_{\pmb\tau}^{n+1})\no\\
	&&=\lambda^{*}\chi_3d_{t}(d_{t}\div Z_{\pmb\tau}^{n+1},G_{\delta}^{n+1})-\lambda^{*}\chi_3(d_{t}\div Z_{\pmb\tau}^{n},d_{t}G_{\delta}^{n+1})+\lambda^{*}\chi_3(d_{t}\div Z_{\pmb\tau}^{n+1},d_{t}G_{\delta}^{n+1})\no\\
	&&~~-\lambda^{*}\chi_3(d_{t}\div Z_{\pmb\tau}^{n+1},d_{t}G_{\delta}^{n+1})+(d_{t}G_{\varpi}^{n},\lambda^{*}\chi_1d_{t}\div Z_{\pmb\tau}^{n+1})\no\\
	&&=\lambda^{*}\chi_3d_{t}(d_{t}\div Z_{\pmb\tau}^{n+1},G_{\delta}^{n+1})+\lambda^{*}\chi_3\Delta t(d_{t}^{2}\div Z_{\pmb\tau}^{n+1},d_{t}G_{\delta}^{n+1})\no\\
	&&~~+(\lambda^{*}d_{t}\div Z_{\pmb\tau}^{n+1},\chi_1d_{t}G_{\varpi}^{n}-\chi_3d_{t}G_{\delta}^{n+1})\no\\
	&&=\lambda^{*}\chi_3d_{t}(d_{t}\div Z_{\pmb\tau}^{n+1},G_{\delta}^{n+1})+\lambda^{*}\chi_3\Delta t(d_{t}^{2}\div Z_{\pmb\tau}^{n+1},d_{t}G_{\delta}^{n+1})\no\\
	&&~~+(\lambda^{*}d_{t}\div Z_{\pmb\tau}^{n+1},d_{t}\div Z_{\pmb\tau}^{n+1}+\lambda^{*}\chi_3d_{t}^{2}\div Z_{\pmb\tau}^{n+1})\no\\
	&&=\lambda^{*}\chi_3d_{t}(d_{t}\div Z_{\pmb\tau}^{n+1},G_{\delta}^{n+1})+\lambda^{*}\chi_3\Delta t(d_{t}^{2}\div Z_{\pmb\tau}^{n+1},d_{t}G_{\delta}^{n+1})\no\\
	&&~~+(\lambda^{*}d_{t}\div Z_{\pmb\tau}^{n+1},d_{t}\div Z_{\pmb\tau}^{n+1})+(\lambda^{*}d_{t}\div Z_{\pmb\tau}^{n+1},\lambda^{*}\chi_3d_{t}^{2}\div Z_{\pmb\tau}^{n+1}).\no
\end{eqnarray}
Taking $ \bv_{h}=d_{t}Z_{\pmb\tau}^{n+1}$ in \reff{3.43}, we obtain
\begin{eqnarray}
	\gamma(\varepsilon(d_{t}Z_{\pmb\tau}^{n+1}),d_{t}^{2}\varepsilon(Z_{\pmb\tau}^{n+1}))-(d_{t}G_{\delta}^{n+1},\div d_{t}^{2}Z_{\pmb\tau}^{n+1})=(d_{t}F_{\delta}^{n+1},\div d_{t}^{2}Z_{\pmb\tau}^{n+1}).\label{3.51}
\end{eqnarray}
Using \reff{3.50} and \reff{3.51} in \reff {3.49}, we get
\begin{eqnarray}
	&&\qquad\qquad\dfrac{\gamma\Delta t}{2}\left\|d_{t}\varepsilon(Z_{\pmb\tau}^{n+1}) \right\|_{L^{2}(\Omega)}^{2}+ \dfrac{\gamma}{2}d_{t}\left\|\varepsilon(Z_{\pmb\tau}^{n+1}) \right\|_{L^{2}(\Omega)}^{2}+\dfrac{\chi_3\Delta t}{2}\left\|d_{t}G_{\delta}^{n+1} \right\|_{L^{2}(\Omega)}^{2}\label{3.52}\\
	&&+\dfrac{\chi_3}{2}d_{t}\left\|\lambda^{*}d_{t}\div Z_{\pmb\tau}^{n+1}+G_{\delta}^{n+1} \right\|_{L^{2}(\Omega)}^{2}+\dfrac{\chi_3\Delta t}{2}\left\|\lambda^{*}d_{t}^{2}\div Z_{\pmb\tau}^{n+1} \right\|_{L^{2}(\Omega)}^{2}\no\\
	&&+ \dfrac{\gamma\chi_3\lambda^{*}\Delta t}{2}d_{t}\left\|d_{_{t}}\varepsilon(Z_{\pmb\tau}^{n+1}) \right\|_{L^{2}(\Omega)}^{2}+\dfrac{\gamma\chi_3\lambda^{*}(\Delta t)^{2}}{2}\left\|d_{t}^{2}\varepsilon(Z_{\pmb\tau}^{n+1}) \right\|_{L^{2}(\Omega)}^{2}+\lambda^{*}\left\|d_{t}\div Z_{\pmb\tau}^{n+1} \right\|_{L^{2}(\Omega)}^{2}\no\\
	&&+\dfrac{\chi_2\Delta t}{2}\left\|d_{t}G_{\varpi}^{n+1} \right\|_{L^{2}(\Omega)}^{2}+\dfrac{\chi_2}{2}d_{t}\left\|G_{\varpi}^{n+\theta} \right\|_{L^{2}(\Omega)}^{2}-\chi_3\lambda^{*}\Delta t(d_{t}F_{\delta}^{n+1},\div d_{t}^{2}Z_{\pmb\tau}^{n+1})\no\\
	&&+\dfrac{1}{\theta_f}(K(\nabla \hat{Z}_{p}^{n+1},\nabla\hat{Z}_{p}^{n+1})\no\\
	&&=(F_{\delta}^{n+1},\div d_{t}Z_{\pmb\tau}^{n+1})-(d_{t}\div Y_{\pmb\tau}^{n+1},G_{\delta}^{n+1})-\lambda^{*}\chi_3(d_{t}^{2}\div Y_{\pmb\tau}^{n+1},G_{\delta}^{n+1})\no\\
	&&+\chi_1(1-\theta)\varDelta t(d_{t}\varpi(t_{n+1}),G_{\delta}^{n+1})+(1-\theta)\dfrac{\chi_1\varDelta t}{\theta_f}\left(Kd_{t}\nabla Z_{\delta}^{n+1},\nabla\hat{Z}_{p}^{n+1}  \right)\no\\
	&&+(1-\theta)\dfrac{\chi_1\lambda^{*}\varDelta t}{\theta_f}(Kd_{t}^{2}\nabla\div Z_{\pmb\tau}^{n+1},\nabla \hat{Z}_{p}^{n+1})+(R_{h}^{n+\theta},\hat{Z}_{p}^{n+1})+(d_{t}G_{\varpi}^{n+\theta},Y_{p}^{n+1}-F_{p}^{n+1})\no\\
	&&+\lambda^{*}\chi_3(\mathbf{R}_{h}^{n+1},G_{\delta}^{n+1}).\no
\end{eqnarray}
Applying the summation operator $ \Delta t\sum_{n=0}^{l} $ to both sides of \reff{3.52}, the \reff{3.38} holds. The proof is complete.
\end{proof}
\begin{theorem}\label{2021}
	Let $ \left\lbrace (\pmb\tau_{h}^{n},\delta_{h}^{n},\varpi_{h}^{n})\right\rbrace_{n\geq0}  $ be defined by the (MFEA), then there holds
	\begin{align}
		&\max_{0\leq n \leq l}\left[\sqrt{\gamma}\left\|\varepsilon(Z_{\pmb\tau}^{n+1}) \right\|_{L^{2}(\Omega)}+\sqrt{\chi_2}\left\|G_{\varpi}^{n+\theta} \right\|_{L^{2}(\Omega)}\right.\label{3.53} \\
		&\left.+\sqrt{\chi_3}\left\|\lambda^{*}d_{t}\div Z_{\pmb\tau}^{l+1}+G_{\delta}^{l+1} \right\|_{L^{2}(\Omega)}\right]+\left[  \varDelta t\sum_{n=1}^{l}\dfrac{K}{\theta_f}\left\|\nabla\hat{Z}_{p}^{n+1} \right\|_{L^{2}(\Omega)}^{2}  \right]^{\frac{1}{2}}\leq \hat{C}_{1}(T)\varDelta t+\hat{C}_{2}(T)h^{2} \nonumber
	\end{align}
	provided that $ \varDelta t=O(h^{2}) $ when $ \theta=0 $ and $ \varDelta t>0 $ when $ \theta=1 $. Here
	\begin{eqnarray}
		&&\hat{C}_{1}(T)=\hat{C}\left\|\varpi_{t} \right\|_{L^{2}((0,T);L^{2}(\Omega))}+\hat{C}\left\|\varpi_{tt}\right\|_{L^{2}((0,T);H^{1}(\Omega)^{'})}\\
		&&+\hat{C}\left\|\div\pmb\tau_{tt}\right\|_{L^{2}((0,T);H^{1}(\Omega)^{'})},\no\\
		&& \hat C_{2}(T)=\hat{C}\left\|\delta_{t} \right\|_{L^{2}((0,T);H^{2}(\Omega))}+\hat{C}\left\|\delta \right\|_{L^{\infty}((0,T);H^{2}(\Omega))}\\
		&&+\hat{C}\left\|\pmb\tau \right\|_{L^{2}((0,T);H^{3}(\Omega))}+\hat{C}\left\|\nabla\cdot\pmb\tau_{t} \right\|_{L^{2}((0,T);H^{2}(\Omega))}.\no
	\end{eqnarray}
\end{theorem}
\begin{proof}
	Using \reff{3.38} and the fact of $ Z_{\boldsymbol{u}}^{0}=\boldsymbol{0},~Z_{\delta}^{0}=0$ and $ Z_{\varpi}^{-1}=0$, we have
	\begin{eqnarray}
		&&\quad\qquad\mathcal{E}_{h}^{l}+\Delta t\sum_{n=0}^{l}\left[\dfrac{\Delta t}{2}\left(\chi_3\left\|\lambda^{*}d_{t}^{2}\div Z_{\pmb\tau}^{n+1} \right\|_{L^2(\Omega)}^{2}+\gamma\left\|d_{t}\varepsilon(Z_{\pmb\tau}^{n+1}) \right\|_{L^2(\Omega)}^{2}\right.\right.\label{3.56}\\
		&&\left.+\chi_3\left\| d_{t}G_{\delta}^{n+1} \right\|_{L^2(\Omega)}^{2}+\chi_2\left\| d_{t}G_{\varpi}^{n+\theta} \right\|_{L^2(\Omega)}^{2}+\gamma\chi_3\lambda^{*}\Delta t\left\|d_{t}^{2}\varepsilon(Z_{\pmb\tau}^{n+1}) \right\|_{L^{2}(\Omega)}^{2}\right)\no\\
		&&\left.+\lambda^{*}\left\|d_{t}\div Z_{\pmb\tau}^{n+1} \right\|_{L^{2}(\Omega)}^{2}+\dfrac{1}{\theta_f}(K(\nabla \hat{Z}_{p}^{n+1},\nabla\hat{Z}_{p}^{n+1})\right] \no\\
		&&=\Phi_{1}+\Phi_{2}+\Phi_{3}+\Phi_{4}+\Phi_{5}+\Phi_{6}+\Phi_{7}+\Phi_{8}+\Phi_{9},\no 
	\end{eqnarray}
where
\begin{eqnarray}
	&&\Phi_{1}=\chi_{3}\lambda^{*}(\Delta t)^{2}\sum_{n=0}^{l}(d_{t}F_{\delta}^{n+1},\div d_{t}^{2}Z_{\pmb\tau}^{n+1}),\no\\
	&&\Phi_{2}=\varDelta t\sum_{n=0}^{l}\left[
	(F_{\delta}^{n+1},\div d_{t}Z_{\pmb\tau}^{n+1})-(\div d_{t}Y_{\pmb\tau}^{n+1},G_{\delta}^{n+1})\right],\no\\
	&&\Phi_{3}=\Delta t\sum_{n=0}^{l}\chi_1(1-\theta)\varDelta t(d_{t}^{2}\varpi(t_{n+1}),G_{\delta}^{n+1}),\no\\
	&&\Phi_{4}=\Delta t\sum_{n=0}^{l}(1-\theta)\dfrac{\chi_1\varDelta t}{\theta_f}\left(Kd_{t}\nabla Z_{\delta}^{n+1},\nabla\hat{Z}_{p}^{n+1}\right),\no\\
	&&\Phi_{5}=\Delta t\sum_{n=0}^{l}(d_{t}G_{\varpi}^{n+\theta},Y_{p}^{n+1}-F_{p}^{n+1}),\no\\
	&&\Phi_{6}=\Delta t\sum_{n=0}^{l}(1-\theta)\dfrac{\chi_1\lambda^{*}\varDelta t}{\theta_f}(Kd_{t}^{2}\nabla\div Z_{\pmb\tau}^{n+1},\nabla \hat{Z}_{p}^{n+1}),\no\\
	&&\Phi_{7}=\Delta t\sum_{n=0}^{l}-\lambda^{*}\chi_3(d_{t}^{2}\div Y_{\pmb\tau}^{n+1},G_{\delta}^{n+1}) ,\no\\
	&&\Phi_{8}=\Delta t\sum_{n=0}^{l}(R_{h}^{n+\theta},\hat{Z}_{p}^{n+1}),\no\\
	&&\Phi_{9}=\Delta t\sum_{n=0}^{l}\lambda^{*}\chi_3(\mathbf{R}_{h}^{n+1},G_{\delta}^{n+1}).\no
\end{eqnarray}
Next, we estimate each term on the right-hand of \reff{3.56}. For the boundness of $ \Phi_{2}, \Phi_{3}, \Phi_{4}$ and $ \Phi_{8} $, one can refer to \cite{fgl14}. As for the boundness of $ \Phi_{5} $, one can refer to \cite{201912097}. Here we omit the more details. For $\Phi_{1}$ and $ \Phi_{7} $, using Cauchy-Schwarz inequality and Young inequality, we have
\begin{eqnarray}
	&&\qquad\Phi_{1}=\chi_{3}\lambda^{*}(\Delta t)^2\sum_{n=0}^{l}(d_{t}F_{\delta}^{n+1},\div d_{t}^{2}Z_{\pmb\tau}^{n+1})\label{eq20220205-0}\\
	&&\leq\chi_{3}\lambda^{*}(\Delta t)^2\sum_{n=0}^{l}\left\|d_{t}F_{\delta}^{n+1}
	\right\|_{L^{2}(\Omega)}\left\|\div d_{t}^{2}Z_{\pmb\tau}^{n+1} \right\|_{L^{2}(\Omega)}\no\\
	&&\leq(\Delta t)^{2} \sum_{n=0}^{l}\left[\chi_3\left\|d_{t}F_{\delta}^{n+1} \right\|_{L^{2}(\Omega)}^{2}+\dfrac{\chi_3(\lambda^{*})^{2}}{4}\left\|\div d_{t}^{2}Z_{\pmb\tau}^{n+1} \right\|_{L^{2}(\Omega)}^{2} \right],\no\\ 
	&&\Phi_{7}=\Delta t\sum_{n=0}^{l}-\lambda^{*}\chi_3(d_{t}^{2}\div Y_{\pmb\tau}^{n+1},G_{\delta}^{n+1})\label{eq20220205-1}\\
	&&\leq\lambda^{*}\chi_3\Delta t\sum_{n=0}^{l}\left\|d_{t}^{2}\div Y_{\pmb\tau}^{n+1} \right\|_{L^{2}(\Omega)}\left\|G_{\delta}^{n+1} \right\|_{L^{2}(\Omega)}\no\\
	&&\leq\lambda^{*}\chi_3\Delta t\sum_{n=0}^{l}\left[\dfrac{1}{2}\left\|d_{t}^{2}\div Y_{\pmb\tau}^{n+1} \right\|_{L^{2}(\Omega)}^{2}+\dfrac{1}{2}\left\|G_{\delta}^{n+1} \right\|_{L^{2}(\Omega)}^{2} \right].\no 	 
	\end{eqnarray}
When $ \theta=0 $, using Cauchy-Schwarz inequality and Young inequality for $ \Phi_{6} $, we get
\begin{eqnarray}
	&&\Phi_{6}=\Delta t\sum_{n=0}^{l}\dfrac{\chi_1\lambda^{*}\varDelta t}{\theta_f}(Kd_{t}^{2}\nabla\div Z_{\pmb\tau}^{n+1},\nabla \hat{Z}_{p}^{n+1})\label{eq20220205-2}\\
	&&\leq(\Delta t)^{2}\sum_{n=0}^{l}\dfrac{\chi_1\lambda^{*}K_{2}}{h\theta_f}\left\|d_{t}^{2}\div Z_{\pmb\tau}^{n+1} \right\|_{L^{2}(\Omega)}\left\|\nabla\hat{Z}_{p}^{n+1} \right\|_{L^{2}(\Omega)}\no\\
	&&\leq(\Delta t)^{2}\sum_{n=0}^{l}\left[ \dfrac{\chi_1^{2}(\lambda^{*})^{2}K_{2}^{2}\Delta t}{h^{2}\theta_fK_{1}}\left\|d_{t}^{2}\div Z_{\pmb\tau}^{n+1} \right\|_{L^{2}(\Omega)}^{2}+\dfrac{K_{1}}{4\theta_f\Delta t}\left\|\nabla\hat{Z}_{p}^{n+1} \right\|_{L^{2}(\Omega)}\right]. \no
\end{eqnarray}
As for $ \Phi_{9} $, using the fact of $
	\left\|\mathbf{R}_{h}^{n+1} \right\|_{H^{1}(\varOmega)^{'}}^{2}\leq\dfrac{\varDelta  t}{3}\int_{t_{n}}^{t_{n+1}}\left\|\div\pmb\tau_{tt} \right\|_{H^{1}(\varOmega)^{'}}^{2}dt$, the Cauchy-Schwarz inequality and Young inequality, we get
\begin{align}
	&\Delta t\sum_{n=0}^{l}\left[ \lambda^{*}\chi_3(\mathbf{R}_{h}^{n+1},G_{\delta}^{n+1})\right.\leq\lambda^{*}\chi_3\Delta t\sum_{n=0}^{l}\left[\left\|\mathbf{R}_{h}^{n+1} \right\|_{H^{1}(\Omega)^{'}}^{2}+\dfrac{1}{4}\left\|G_{\delta}^{n+1} \right\|_{L^{2}(\Omega)}^{2}  \right]\label{eq20220205-3}\\
	&\leq\lambda^{*}\chi_3\sum_{n=0}^{l}\Delta t\left[\dfrac{\Delta t}{3}\left\|\div\pmb\tau_{tt} \right\|_{L^{2}(t_{n},t_{n+1};H^{1}(\Omega)^{'})}^{2}+\dfrac{1}{4}\left\|G_{\delta}^{n+1} \right\|_{L^{2}(\Omega)}^{2}\right] .\no
\end{align}
 Adding \reff{eq20220205-0}-\reff{eq20220205-3}  and applying the discrete Gronwall inequality (cf. \cite{Shen1990}), we have
 \begin{eqnarray}
 	&&\gamma\left\| \varepsilon(Z_{\pmb\tau}^{l+1}) \right\|_{L^2(\Omega)}^{2}+\chi_2\left\|G_{\varpi}^{l+\theta} \right\|_{L^{2}(\Omega)}^{2}+\chi_3\left\|\lambda^{*}d_{t}\div Z_{\pmb\tau}^{l+1}+G_{\delta}^{l+1} \right\|_{L^{2}(\Omega)}^{2}\no\\
 	&&%+\gamma\chi_3\lambda^{*}\Delta t\left\|d_{_{t}}\varepsilon(Z_{\pmb\tau}^{l+1}) \right\|_{L^{2}(\Omega)}^{2}
 	+\varDelta t\sum_{n=0}^{l}\dfrac{K_{1}}{\theta_f}\left\|\nabla\hat{Z}_{p}^{n+1} \right\|_{L^{2}(\Omega)}^{2}\no\\
 	&&\leq\hat{C}\left[ \dfrac{4\gamma\chi_1^{2}}{\beta_{1}^{2}}\left\|\varpi_{t} \right\|_{L^{2}((0,T);L^{2}(\Omega))}^{2}+\dfrac{\theta_f(\Delta t)^{2}}{3K_{1}}\left\|\varpi_{tt}\right\|_{L^{2}((0,T);H^{1}(\Omega)^{'})}^{2}\right.\no\\
 	&&+\dfrac{\lambda^{*}\chi_3(\Delta t)^{2}}{3}\left\|\div\pmb\tau_{tt}\right\|_{L^{2}((0,T);H^{1}(\Omega)^{'})}^{2}+\left\|F_{\delta}^{l+1} \right\|_{L^{2}(\Omega)}^{2}\no\\
 	&&\left.+\Delta t\sum_{n=0}^{l}\left\|d_{t}F_{\delta}^{l+1} \right\|_{L^{2}(\Omega)}^{2}+\Delta t\sum_{n=0}^{l}\left\|d_{t}\div Y_{\pmb\tau}^{l+1} \right\|_{L^{2}(\Omega)}^{2}\right] \no \\
 	&&\leq\hat{C}(\Delta t)^{2}\left( \left\|\varpi_{t} \right\|_{L^{2}((0,T);L^{2}(\Omega))}^{2}+\left\|\varpi_{tt}\right\|_{L^{2}((0,T);H^{1}(\Omega)^{'})}^{2}+\left\|\div\pmb\tau_{tt}\right\|_{L^{2}((0,T);H^{1}(\Omega)^{'})}^{2}\right) \no\\
 	&& +\hat{C}h^{4}\left( \left\|\delta_{t} \right\|_{L^{2}((0,T);H^{2}(\Omega))}^{2}+\left\|\delta \right\|_{L^{\infty}((0,T);H^{2}(\Omega))}^{2}+\left\|\div\pmb\tau_{t} \right\|_{L^{2}((0,T);H^{2}(\Omega))}^{2} \right) \no
 \end{eqnarray}
 provided that $\varDelta t\leq  \min\left\lbrace \dfrac{ \theta_fK_{1}\beta_{1}^{2}h^{2}}{4\gamma\chi_1^{2}K_{2}^{2}},\dfrac{\theta_f\chi_3h^{2}K_{1}}{4\chi_1^{2}K_{2}^{2}}\right\rbrace  $ when $ \theta=0 $, but it holds for all $ \Delta t>0 $ when $ \theta=1 $, the \reff{3.53} holds. The proof is complete.
\end{proof}
\begin{theorem}
	The solution of the (MFEA) satisfies the following error estimates:
	\begin{align}
		&\max_{0\leq n\leq N}\left[\sqrt{\gamma}\left\|\nabla(\pmb\tau(t_{n+1})-\pmb\tau_{h}^{n+1})\right\|_{L^{2}(\Omega )}\right.\\
		&\left.+\sqrt\chi_2\left\|\varpi(t_{n+1})-\varpi_{h}^{n+1}\right\|_{L^{2}(\Omega )}+\sqrt{\chi_3}\left\|\delta(t_{n+1})-\delta_{h}^{n+1}\right\|_{L^{2}(\Omega )} \right]\no\\
		&\leq\check{C}_{1}(T)\varDelta t+\check{C}_{2}(T)h^{2},\nonumber\\
		&\left(\varDelta t\sum_{n=0}^{N}\dfrac{K}{\theta_f}\left\|\nabla p(t_{n+1})-\nabla p_{h}^{n+1}\right\|_{L^{2}(\Omega )}^{2} \right)^{\frac{1}{2}}\leq \check{C}_{1}(T)\varDelta t+\check{C}_{2}(T)h
	\end{align}
	provided that $ \varDelta t=O(h^{2}) $ when $ \theta=0 $ and $ \varDelta t\geq0 $ when $ \theta=1$. Here
	\begin{align*}
		&\check{C}_{1}(T)=\hat{C}_{1}(T),\\
		&\check{C}_{2}(T)=\hat{C}_{2}(T)+\left\|\delta \right\|_{L^{\infty}((0,T);H^{2}(\Omega))}+\left\|\varpi \right\|_{L^{\infty}((0,T);H^{2}(\Omega))}+\left\|\nabla\boldsymbol{u} \right\|_{L^{\infty}((0,T);H^{2}(\Omega))}.
	\end{align*}
\end{theorem}
\begin{proof}
	The above estimates follow immediately from an application of the triangle inequality on
	\begin{align*}
		&\pmb\tau(t_{n})-\pmb\tau_{h}^{n}=Y_{\pmb\tau}^{n}+Z_{\pmb\tau}^{n},~~~~~~~~~~~~~~~~~~~~~\delta(t_{n})-\delta_{h}^{n}=Y_{\delta}^{n}+Z_{\delta}^{n}=F_{\delta}^{n}+G_{\delta}^{n},\\
		&\varpi(t_{n})-\varpi_{h}^{n}=Y_{\varpi}^{n}+Z_{\varpi}^{n}=F_{\varpi}^{n}+G_{\varpi}^{n},~~~~~~~p(t_{n})-p_{h}^{n}=Y_{p}^{n}+Z_{p}^{n}=F_{p}^{n}+G_{p}^{n}.
	\end{align*}
	and appealing to \reff{3.34}, \reff{3.35}, \reff{3.36} and Theorem \ref{2021}. The proof is complete.
\end{proof}
\section{Numerical experiments}\label{sec-4}
~

{\bf Test 1.} Let $\Omega=(0,1)\times (0,1)$, $\Gamma_1=\{(1,x_2); 0\leq x_2\leq 1\}$,
$\Gamma_2=\{(x_1,0); 0\leq x_1\leq 1\}$, $\Gamma_3=\{(0,x_2); 0\leq x_2\leq 1\}$,
$\Gamma_4=\{(x_1,1); 0\leq x_1\leq 1\}$, and $T=1$. The source functions are as follows:
\begin{align*}
	\mathbf{F} &=\lambda^{*}\pi^{2}(\sin(\pi x_{1}),\sin(\pi x_{2}))^{T}+(\beta+\gamma)\pi^{2} t(\sin(\pi x_{1}),\sin(\pi x_{2}))^T\\
	&+b_0 t\pi\cos(\pi x_1+\pi x_2)(1,1)^T,\\
	\phi &=a_0\sin(\pi x_{1}+\pi x_{2})+\frac{2K}{\theta_f} t\pi^2\sin(\pi x_1+\pi x_2)+b_0\pi(\cos(\pi x_1)+\cos(\pi x_2)),
\end{align*}
and the boundary and initial conditions are 
\begin{alignat*}{2}
	p &= t\sin(\pi x_{1}+\pi x_{2})  &&\qquad\mbox{on }\partial\Omega_T,\\
	\tau_1 &= t\sin(\pi x_{1}) &&\qquad\mbox{on }\Gamma_j\times (0,T),\, j=1,3,\\
	\tau_2 &= t\sin(\pi x_{2}) &&\qquad\mbox{on }\Gamma_j\times (0,T),\, j=,2,4,\\
	\lambda^{*}\div\pmb\tau_{t}\bn+\sigma\bn-b_{0} \emph{p}\bf{n} &= \mathbf{F}_1 &&\qquad \mbox{on } \p\Ome_T,\\
	\pmb\tau(x,0) = \mathbf{0},  \quad p(x,0) &=0 &&\qquad\mbox{in } \Ome,
\end{alignat*}
where
\begin{align*}
	\mathbf{F}_1(x,t)&=\lambda^{*}\pi(\cos(\pi x_{1})+\cos(\pi x_{2}))(1,1)^T+\gamma\pi t(\cos(\pi x_{1}),\cos(\pi x_{2}))^T\\
	&+\beta\pi t(\cos(\pi x_{1})+\cos(\pi x_{2}))(1,1)^T-b_0 t\sin(\pi x_1+\pi x_2)(1,1)^T.
\end{align*}
The exact solution of this problem is
\[
\pmb\tau(x,t)=t\bigl( \sin(\pi x_1), \sin(\pi x_2) \bigr)^T,\qquad p(x,t)=t\sin(\pi x_{1}+\pi x_{2}).
\]
\begin{table}[H]
	\begin{center}
		%\smallskip
		\caption{Values of parameters} \label{tab101}
		\begin{tabular}{l c c }
			\hline
			Parameters &Description  &Values  \\ \hline
			$\lambda^{*}$&  Coefficient of secondary consolidation          &1e-5\\
			$\nu $ &  Poisson ratio   & 0.25      \\ %\hline
			$b_0$ &  Biot-Willis constant   & 1e-5   \\ %\hline
			$E$ &  Young's modulus   & 25 \\ %\hline
			$\beta$ &  Lam$\acute{e}$ constant   & 10  \\%\hline
			$ K $ & Permeability tensor  & (1e-3) $\bf I$ \\% \hline
			$ \gamma $ &  Lam$\acute{e}$ constant  & 10\\%\hline
			$ a_0 $ &  Constrained specific storage coefficient  & 0.2\\\hline	
		\end{tabular}
	\end{center}
\end{table}
\begin{table}[!htbp]
	\begin{center}
		\caption{Spatial errors and convergence rates of $\pmb\tau$} \label{tab102}
		\begin{tabular}{l c c c c }
			\hline
			$h$ & $\|\pmb\tau-\pmb\tau_h\|_{L^2}$ &CR &$ \|\pmb\tau-\pmb\tau_h\|_{H^1}$ & CR\\ \hline
			$h=1/4$ & 2.6318e-3  &    & 7.9301e-2       &     \\ %\hline
			$h=1/8$ & 3.1932e-4  & 3.043  & 1.8635e-2     &  2.0893     \\ %\hline
			$h=1/16$ & 3.9427e-5  & 3.0177  & 4.5654e-3      & 2.0292      \\ %\hline
			$h=1/32$ & 4.9094e-6  & 3.0056   & 1.1336e-3     &  2.0098      \\ \hline
		\end{tabular}
	\end{center}
\end{table}
\begin{table}[!htbp]
	\begin{center}
		\caption{Spatial errors and convergence rates of $p$} \label{tab103}
		\begin{tabular}{l c c c c }
			\hline
			$h$ & $\|p-p_h\|_{L^2}$ &CR &$ \|p-p_h\|_{H^1}$ & CR\\ \hline
			$h=1/4$ & 2.6672e-2  &    & 7.3216e-1       &     \\ %\hline
			$h=1/8$ & 5.6605e-3  & 2.2363  & 3.5970e-1     &  1.0254     \\ %\hline
			$h=1/16$ & 1.3277e-3  & 2.0920  & 1.7857e-1     & 1.0103      \\ %\hline
			$h=1/32$ & 3.2584e-4  & 2.0267   & 8.9098e-2    &  1.0030     \\ \hline
		\end{tabular}
	\end{center}
\end{table}
As for the convergence order of time, we define 
\begin{eqnarray}
	\rho_{h,\Delta t}=\dfrac{\left\|v^{h,\Delta t} -v^{h,\frac{\Delta t}{2}}\right\|_{L^{2}} }{\left\|v^{h,\frac{\Delta t}{2}} -v^{h,\frac{\Delta t}{4}}\right\|_{L^{2}}},\no
\end{eqnarray}
where $ v=\pmb\tau,p $. In particular, $ \rho_{h,\Delta t}\approx2 $ when the corresponding order of convergence in time is of $ O(\Delta t) $, one can refer to \cite{20211223}.
\begin{table}[!htbp]
	\begin{center}
		\caption{Spatial errors and convergence rates of $ \pmb\tau $ and $p$ with $ h=\dfrac{1}{8} $} \label{tab103-1}
		\begin{tabular}{l c c c c }
			\hline
			$\Delta t$ & $\|\pmb\tau-\pmb\tau_h\|_{L^2}$ &$ \rho_{h,\Delta t} $ &$ \|p-p_h\|_{L^2}$ & $ \rho_{h,\Delta t} $\\ \hline
			$\Delta t=1/10$ & 5.1594e-9  &    & 4.5958e-5       &     \\ %\hline
			$\Delta t=1/20$ & 2.5796e-9  & 2.0001  & 2.3380e-5     &  1.9657     \\ %\hline
			$\Delta t=1/40$ & 1.2898e-9  & 2.0000  & 1.1793e-5     & 1.9825      \\ %\hline
			$\Delta t=1/80$ & 6.4489e-10  & 2.0000  &5.9228e-6     &  1.9911     \\ \hline
		\end{tabular}
	\end{center}
\end{table}
\begin{figure}[H]
	\begin{minipage}[t]{0.45\linewidth}
		\centering
		\includegraphics[height=1.8in,width=2in]{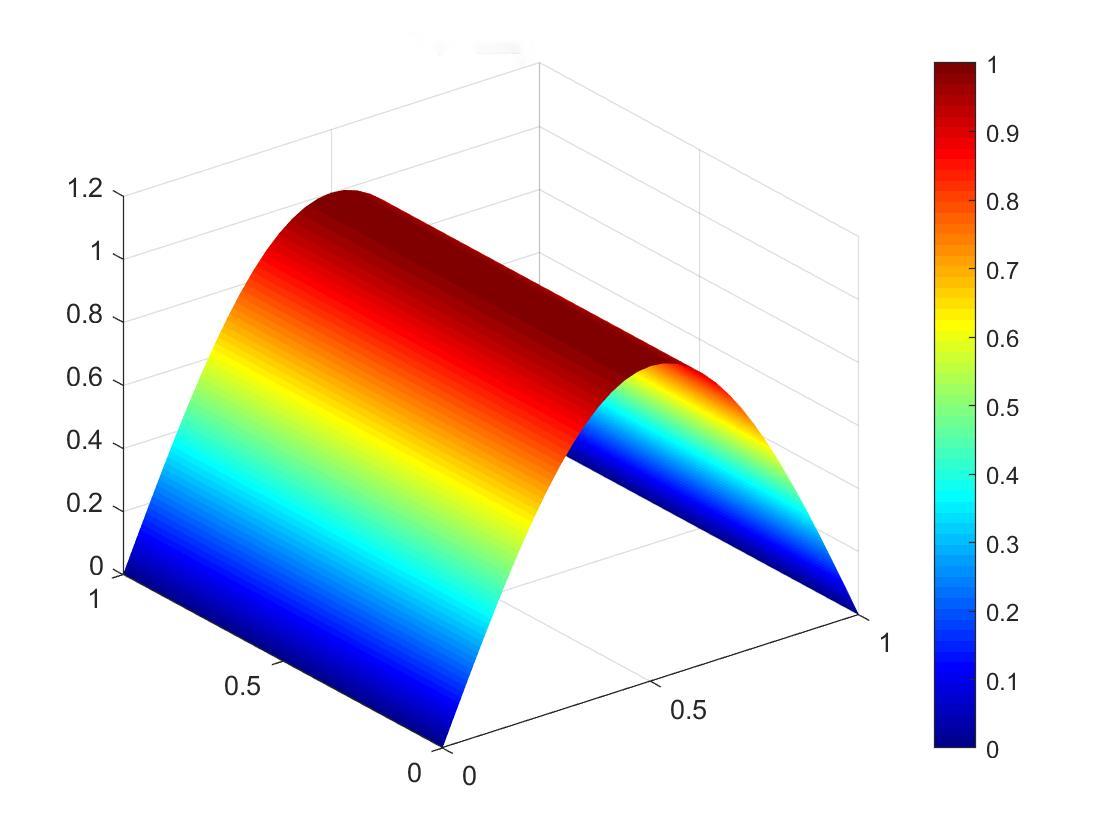}
		\caption{The numerical  displacement\ $( \tau_{1})_{h}^{n+1}$ at the terminal time $T$ with the parameters of Table \ref{tab101}.}\label{figure_1}
	\end{minipage}
	\hspace{0.5in}
	\begin{minipage}[t]{0.45\linewidth}
		\centering
		\includegraphics[height=1.8in,width=2in]{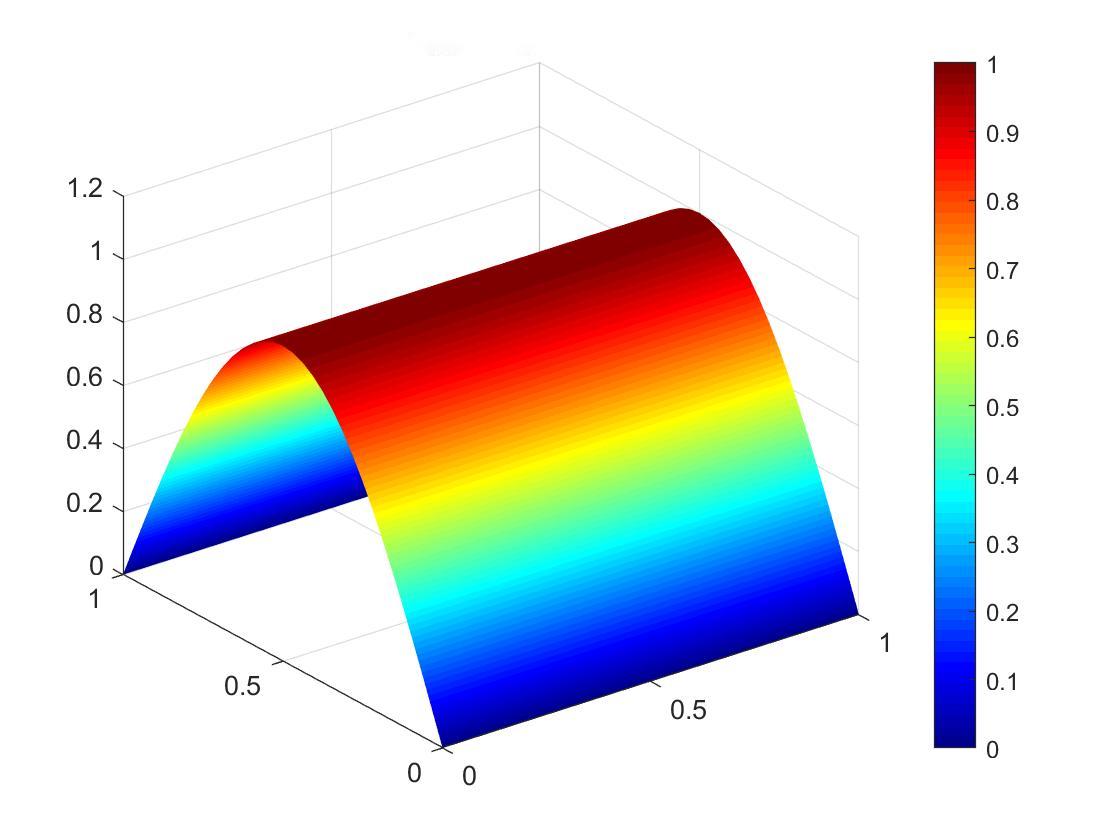}
		\caption{ The numerical  displacement\ $( \tau_{2})_{h}^{n+1}$ at the terminal time $T$ with the parameters of Table \ref{tab101}.}\label{figure_2}
	\end{minipage}
\end{figure}
\begin{figure}[H]
	\begin{minipage}[t]{0.45\linewidth}
		\centering
		\includegraphics[height=1.8in,width=2in]{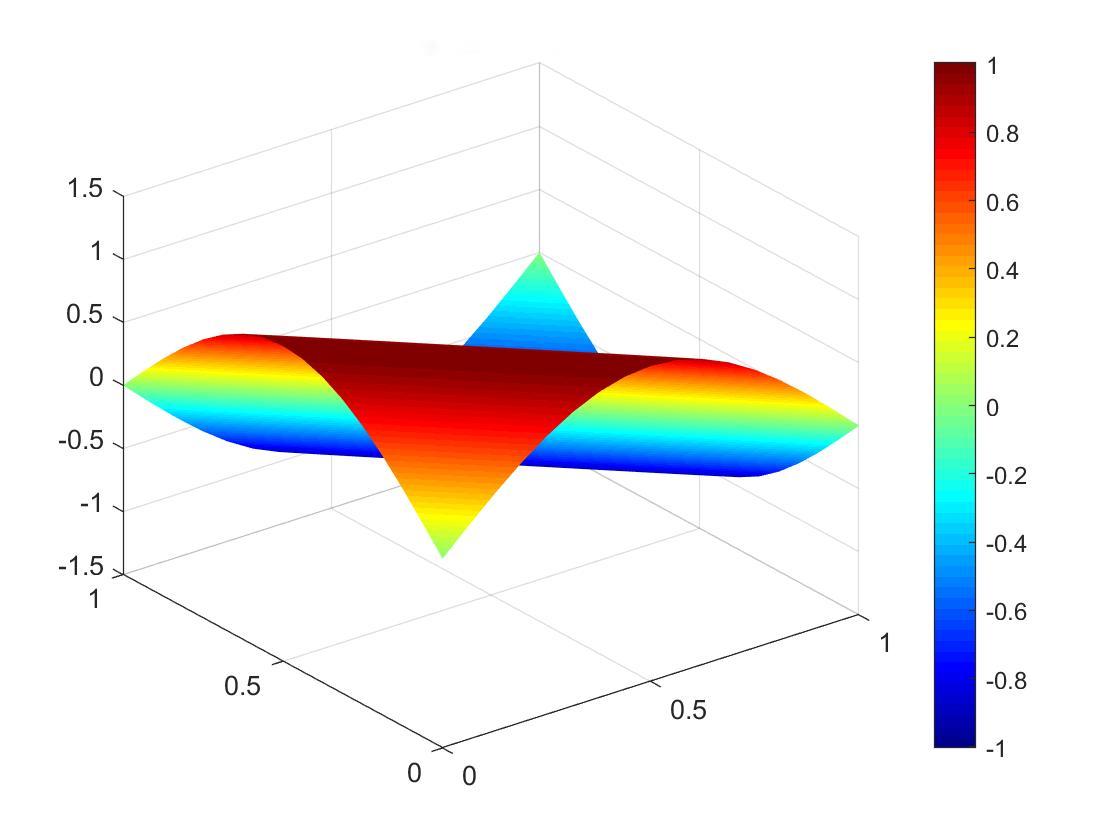}
		\caption{The numerical  pressure\ $p_{h}^{n+1}$ at the terminal time $T$ with the parameters of Table \ref{tab101}.}\label{figure_3}
	\end{minipage}
	\hspace{0.5in}
	\begin{minipage}[t]{0.45\linewidth}
		\centering
		\includegraphics[height=1.8in,width=2.5in]{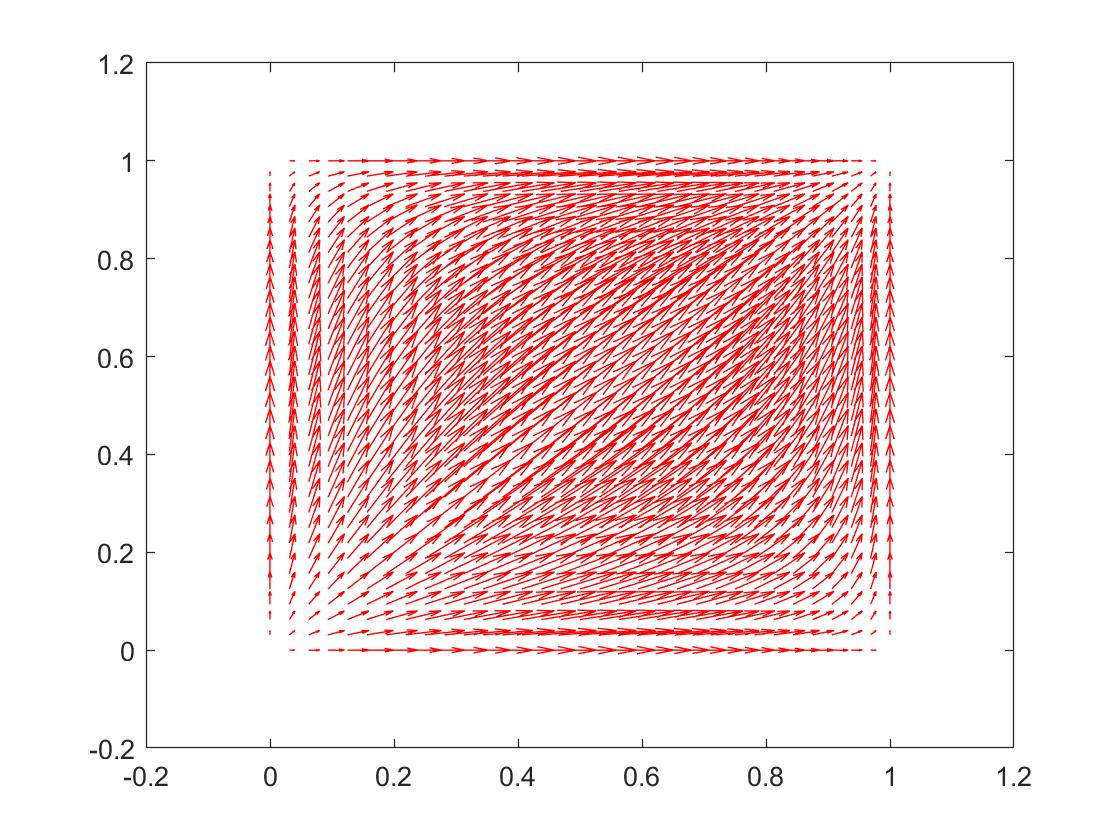}
		\caption{ Arrow plot of the computed displacement $ \pmb\tau $ with the parameters of Table \ref{tab101}.}\label{figure_4}
	\end{minipage}
\end{figure}

\medskip
{\bf Test 2.} The $ \Omega $ and $ T $ are the same as Test 1. The source functions are as follows:
\begin{align*}
	\mathbf{F} &=\lambda^{*}e^{t}(\sin x_{1},\sin x_{2})^{T}+(\beta+\gamma) e^{t}(\sin x_{1},\sin x_{2})^T+b_0 t\pi(\cos(\pi x_{1})\sin(\pi x_{2}),\sin(\pi x_{1})\cos(\pi x_{2}))^T,\\
	\phi &=a_0\sin(\pi x_{1})\sin(\pi x_{2})+\frac{2K}{\theta_f}\pi^{2}t\sin(\pi x_{1})\sin(\pi x_{2})+b_0 e^{t}(\cos x_{1}+\cos x_{2}),
\end{align*}
and the boundary and initial conditions are 
\begin{alignat*}{2}
	p &= t\sin(\pi x_{1})\sin(\pi x_{2})  &&\qquad\mbox{on }\partial\Omega_T,\\
	\tau_1 &= e^{t}\sin x_{1} &&\qquad\mbox{on }\Gamma_j\times (0,T),\, j=1,3,\\
	\tau_2 &=e^{t}\sin x_{2}  &&\qquad\mbox{on }\Gamma_j\times (0,T),\, j=,2,4,\\
	\lambda^{*}\div\pmb\tau_{t}\bn+\sigma\bn-b_{0} \emph{p}\bf{n} &= \mathbf{F}_1 &&\qquad \mbox{on } \p\Ome_T,\\
	\pmb\tau(x,0) = \mathbf{0},  \quad p(x,0) &=0 &&\qquad\mbox{in } \Ome.
\end{alignat*}
where
\begin{align*}
	\mathbf{F}_1(x,t)= \lambda^{*}(\cos x_1+\cos x_2) (n_1,n_2)^T e^t +\gamma e^{t}(\cos x_{1},\cos x_{2})^{T}\\
	+\beta e^{t}(\cos x_{1}+\cos x_{2})(n_{1},n_{2})^{T}-b_0(n_1,n_2)^Tt \sin(\pi x_{1})\sin(\pi x_{2}).
\end{align*}
The exact solution of this problem is
\[
\pmb\tau(x,t)=e^{t} \bigl( \sin x_1, \sin x_2 \bigr)^T,\qquad p(x,t)=t \sin(\pi x_{1})\sin(\pi x_{2}).
\]
\begin{table}[H]
	\begin{center}
		%\smallskip
		\caption{Values of parameters} \label{tab104}
		\begin{tabular}{l c c }
			\hline
			Parameters &Description  &Values  \\ \hline
			$\lambda^{*}$& Coefficient of secondary consolidation          &1e-5\\
			$\nu $ &  Poisson ratio   & 0.25      \\ %\hline
			$b_0$ &  Biot-Willis constant   & 1e-5   \\ %\hline
			$E$ &  Young's modulus   & 25 \\ %\hline
			$\beta$ &  Lam$\acute{e}$ constant   & 10  \\%\hline
			$ K $ & Permeability tensor  & (1e-3) $\bf I$ \\% \hline
			$ \gamma $ &  Lam$\acute{e}$ constant  & 10\\%\hline
			$ a_0 $ &  Constrained specific storage coefficient  & 0.2\\\hline	
		\end{tabular}
	\end{center}
\end{table}
\begin{table}[!htbp]
	\begin{center}
		\caption{Spatial errors and convergence rates of $\pmb\tau$} \label{tab105}
		\begin{tabular}{l c c c c }
			\hline
			$h$ & $\|\pmb\tau-\pmb\tau_h\|_{L^2}$ &CR &$ \|\pmb\tau-\pmb\tau_h\|_{H^1}$ & CR\\ \hline
			$h=1/4$ & 2.6708e-4  &    & 7.8215e-3       &     \\ %\hline
			$h=1/8$ & 3.3136e-5  & 3.0108  & 1.9334e-3     &  2.0163     \\ %\hline
			$h=1/16$ & 4.1402e-6  & 3.0006  & 4.8037e-4      & 2.0089      \\ %\hline
			$h=1/32$ & 5.1792e-7  & 2.9989   & 1.1971e-4     &  2.0046      \\ \hline
		\end{tabular}
	\end{center}
\end{table}
\begin{table}[!htbp]
	\begin{center}
		\caption{Spatial errors and convergence rates of $p$} \label{tab106}
		\begin{tabular}{l c c c c }
			\hline
			$h$ & $\|p-p_h\|_{L^2}$ &CR &$ \|p-p_h\|_{H^1}$ & CR\\ \hline
			$h=1/4$ & 3.5800e-2  &    & 9.0720e-1      &     \\ %\hline
			$h=1/8$ & 7.8029e-3  & 2.1979  & 4.4210e-1     &  1.0370     \\ %\hline
			$h=1/16$ & 1.8634e-3  & 2.0661  & 2.1887e-1     & 1.0143      \\ %\hline
			$h=1/32$ & 4.6956e-4  & 1.9886   & 1.0910e-1     &  1.0044     \\ \hline
		\end{tabular}
	\end{center}
\end{table}
\begin{table}[!htbp]
	\begin{center}
		\caption{Spatial errors and convergence rates of $ \pmb\tau $ and $p$ with $ h=\dfrac{1}{10} $} \label{tab106-1}
		\begin{tabular}{l c c c c }
			\hline
			$\Delta t$ & $\|\pmb\tau-\pmb\tau_h\|_{L^2}$ &$ \rho_{h,\Delta t} $ &$ \|p-p_h\|_{L^2}$ & $ \rho_{h,\Delta t} $\\ \hline
			$\Delta t=1/10$ & 2.1471e-9  &    & 3.8352e-5      &     \\ %\hline
			$\Delta t=1/20$ &1.0948e-9  & 1.9612  & 1.9469e-5    & 1.9699    \\ %\hline
			$\Delta t=1/40$ & 5.5289e-10  & 1.9801  &9.8098e-6     & 1.9846     \\ %\hline
			$\Delta t=1/80$ & 2.7783e-10  & 1.9900  &4.9240e-6     &  1.9922     \\ \hline
		\end{tabular}
	\end{center}
\end{table}
\begin{figure}[H]
	\begin{minipage}[t]{0.45\linewidth}
		\centering
		\includegraphics[height=1.8in,width=2in]{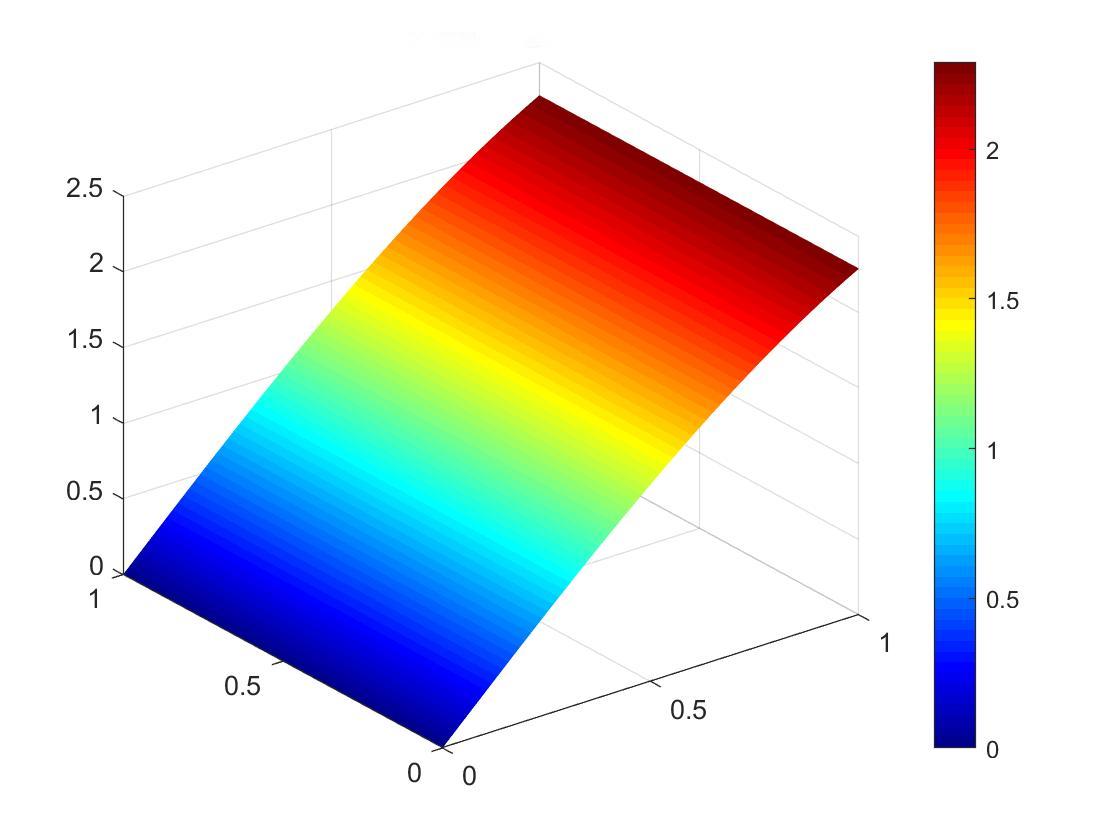}
		\caption{The numerical  displacement\ $( \tau_{1})_{h}^{n+1}$ at the terminal time $T$ with the parameters of Table \ref{tab104}.}\label{figure_5}
	\end{minipage}
	\hspace{0.5in}
	\begin{minipage}[t]{0.45\linewidth}
		\centering
		\includegraphics[height=1.8in,width=2in]{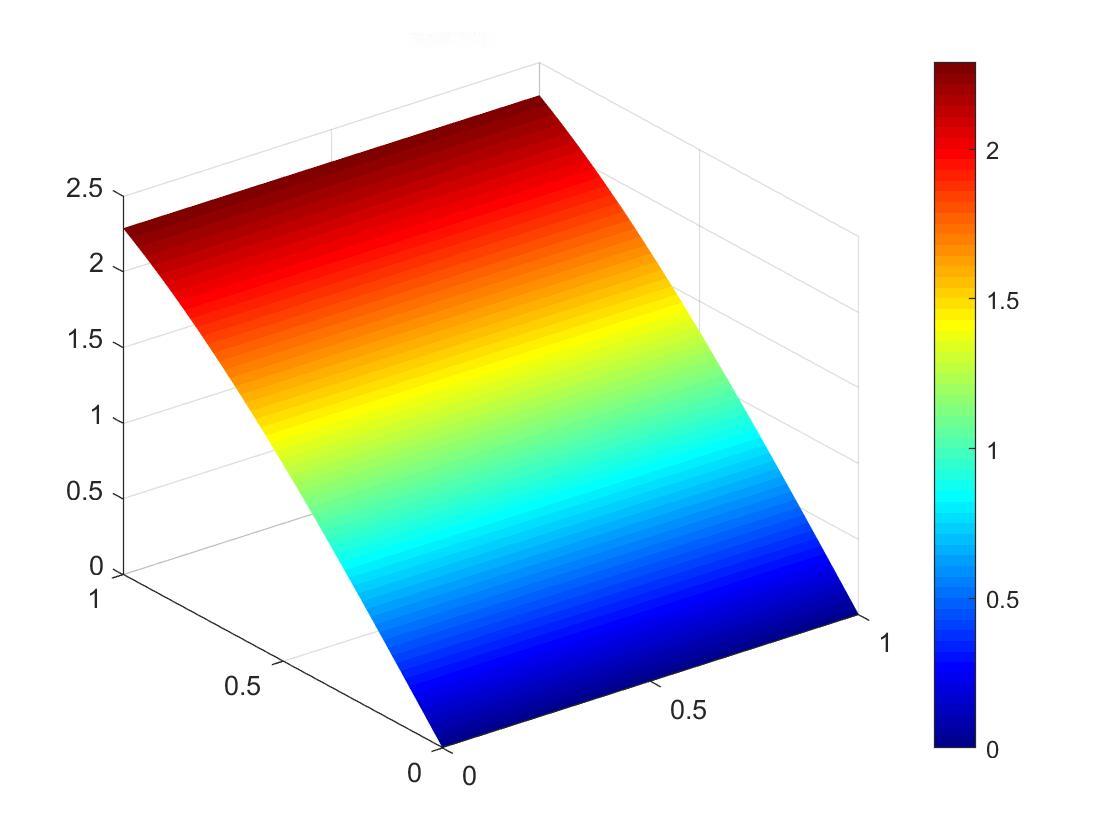}
		\caption{ The numerical  displacement\ $( \tau_{2})_{h}^{n+1}$ at the terminal time $T$ with the parameters of Table \ref{tab104}.}\label{figure_6}
	\end{minipage}
\end{figure}
\begin{figure}[H]
	\begin{minipage}[t]{0.45\linewidth}
		\centering
		\includegraphics[height=1.8in,width=2in]{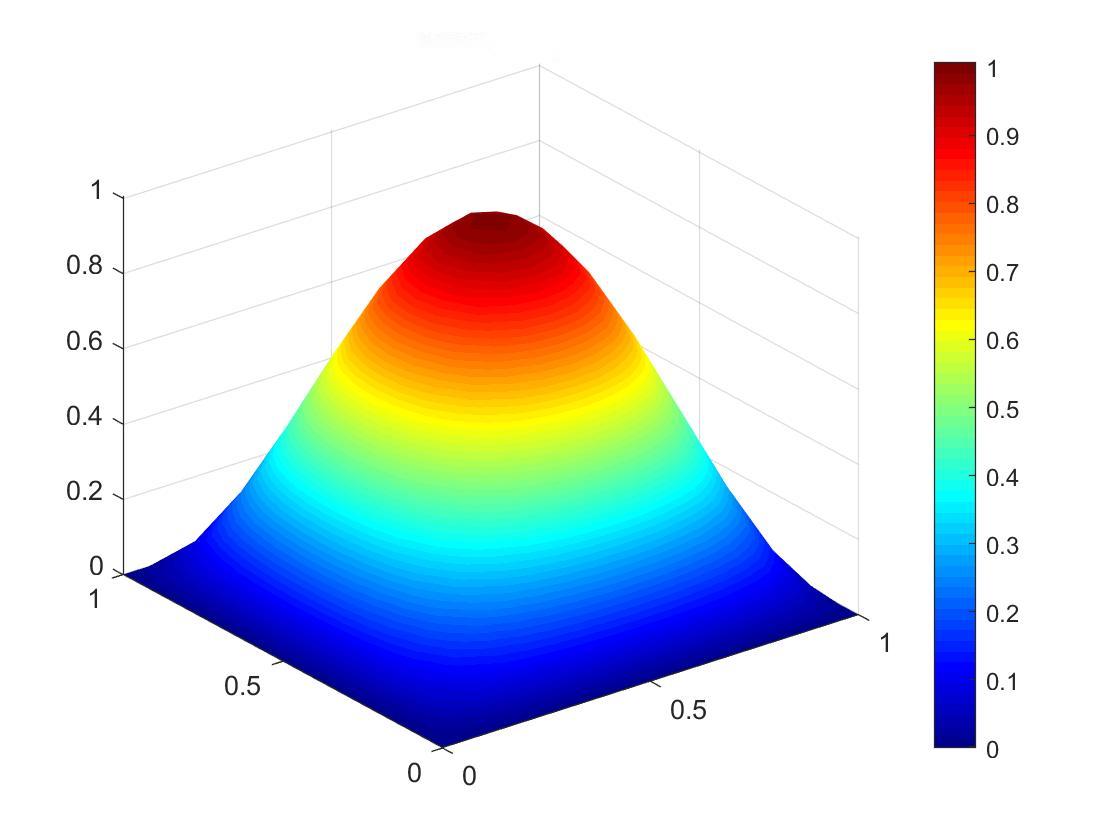}
		\caption{The numerical  pressure\ $p_{h}^{n+1}$ at the terminal time $T$ with the parameters of Table \ref{tab104}.}\label{figure_7}
	\end{minipage}
	\hspace{0.5in}
	\begin{minipage}[t]{0.45\linewidth}
		\centering
		\includegraphics[height=1.8in,width=2.5in]{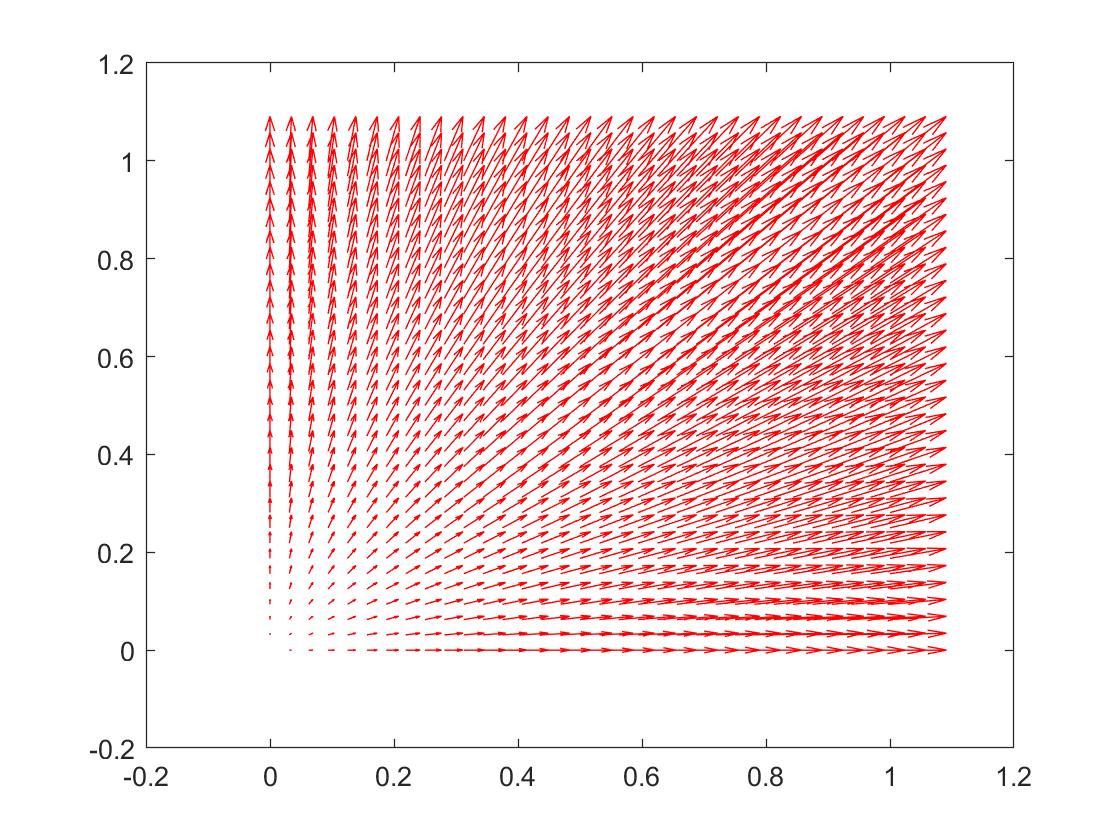}
		\caption{ Arrow plot of the computed displacement $ \pmb\tau $ with the parameters of Table \ref{tab104}.}\label{figure_8}
	\end{minipage}
\end{figure}
Table \ref{tab102} and Table \ref{tab103} display the error of displacement\ $\pmb\tau$ and the pressure\ $p$ with\ $L^2(\Omega)$-norm and\ $H^1(\Omega)$-norm in space at the terminal time $T$ with the parameters of Table  \ref{tab101} of Test 1, which are consistent with the theoretical result. Table \ref{tab103-1} display the error of displacement\ $\pmb\tau$ with\ $L^2(\Omega)$-norm and the pressure\ $p$ with\ $L^2(\Omega)$-norm at $ h=\dfrac{1}{8} $ with the parameters of Table \ref{tab101} of Test 1, which verifies the theoretical result.

Figure  \ref{figure_1} and Figure \ref{figure_2} show the numerical solution of displacement \ $(\tau_{1})_{h}^{n+1}$ and \ $(\tau_{2})_{h}^{n+1}$ at the terminal time\ $T$ with the parameters of Table  \ref{tab101} of Test 1, Figure \ref{figure_3} shows the numerical solution of pressure\ $p_{h}^{n+1}$ at the terminal time\ $T$ with the parameters of Table  \ref{tab101} of Test 1. Figure  \ref{figure_4} show the arrow plot of the computed displacement $ \pmb\tau $ corresponding to the parameters of Table \ref{tab101} of Test 1. Table \ref{tab105}--Table \ref{tab106-1} and Figure  \ref{figure_5}--Figure \ref{figure_8} of Test 2 have a similar description as Test 1.

{\bf Test 3.} This is a benchmark problem, which occurs ``locking" (cf. \cite{PW09}). The $ \Omega $ and $ T $ are the same as Test 1. The source functions are $\mathbf{F} =0, 
	\phi =0$, and the boundary and initial conditions are 
\begin{alignat*}{2}
	p &= 0  &&\qquad\mbox{on }\partial\Omega_T,\\
	\tau_1 &=0 &&\qquad\mbox{on }\Gamma_j\times (0,T),\, j=1,3,\\
	\tau_2 &=0  &&\qquad\mbox{on }\Gamma_j\times (0,T),\, j=,2,4,\\
	\lambda^{*}\div\pmb\tau_{t}\bn+\sigma\bn-b_{0} \emph{p}\bf{n} &= \mathbf{F}_1:=(0,b_0 p) &&\qquad \mbox{on } \p\Ome_T,\\
	\pmb\tau(x,0) = \mathbf{0},  \quad p(x,0) &=0 &&\qquad\mbox{in } \Ome,
\end{alignat*}
where
\begin{align*}
	&p=
	\left\lbrace \begin{aligned}	
		&\sin t  &&when~ x_{1}\in[0.2,0.8)\times(0,T),\\
		&0&&others,
	\end{aligned}\right.
\end{align*}
\begin{table}[H]
	\begin{center}
		%\smallskip
		\caption{Values of parameters}\label{tab109}
		\begin{tabular}{l c c }
			\hline
			Parameters &Description  &Values  \\ \hline
			$\lambda^{*}$& Coefficient of secondary consolidation          &1e-5\\
			$\nu $ &  Poisson ratio   & 0.045      \\ %\hline
			$b_0$ &  Biot-Willis constant   & 1e-5   \\ %\hline
			$E$ &  Young's modulus   & 20909.091 \\ %\hline
			$\beta$ &  Lam$\acute{e}$ constant   & 1e3  \\%\hline
			$ K $ & Permeability tensor  & (1e-7) $\bf I$ \\% \hline
			$ \gamma $ &  Lam$\acute{e}$ constant  & 1e4\\%\hline
			$ a_0 $ &  Constrained specific storage coefficient  & 2e-10\\\hline	
		\end{tabular}
	\end{center}
\end{table}
\begin{figure}[H]
	\begin{minipage}[t]{0.45\linewidth}
		\centering
		\includegraphics[height=1.8in,width=2in]{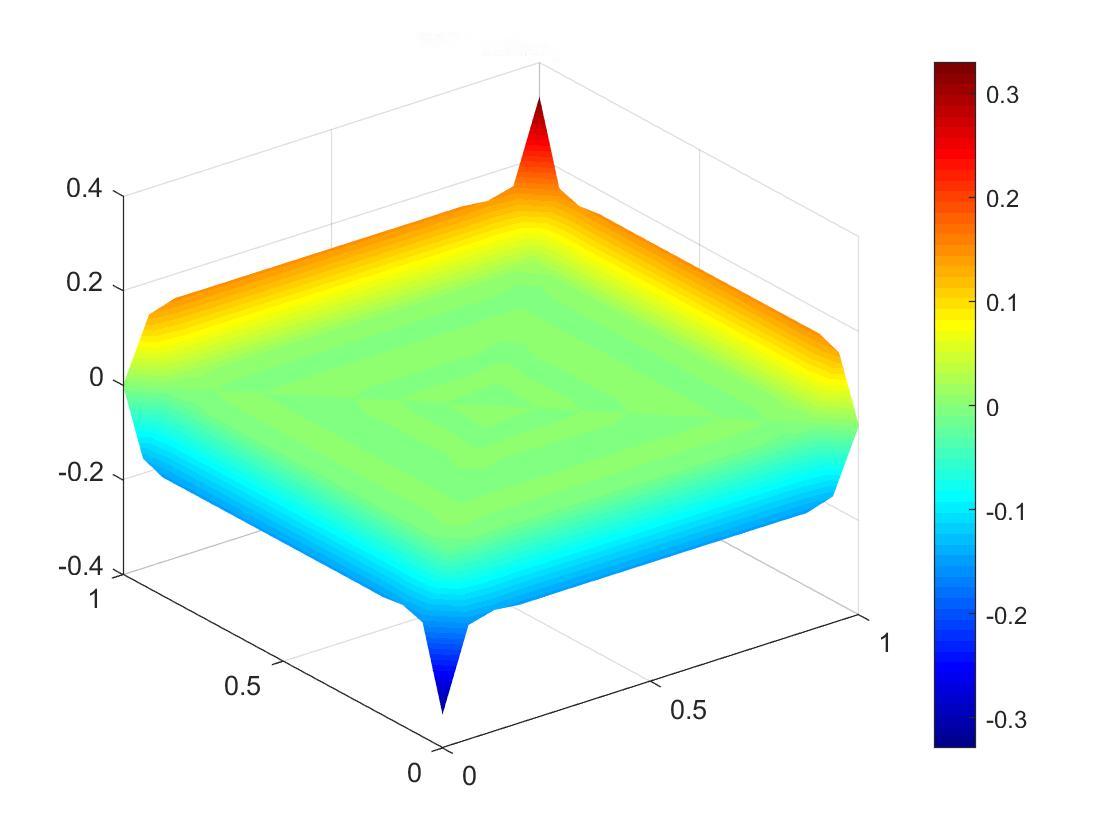}
		\caption{The numerical  pressure\ $p_{h}^{n+1}$ at the terminal time $T$ with the parameters of Table \ref{tab109} for the original model.}\label{figure_9}
	\end{minipage}
	\hspace{0.5in}
	\begin{minipage}[t]{0.45\linewidth}
		\centering
		\includegraphics[height=1.8in,width=2in]{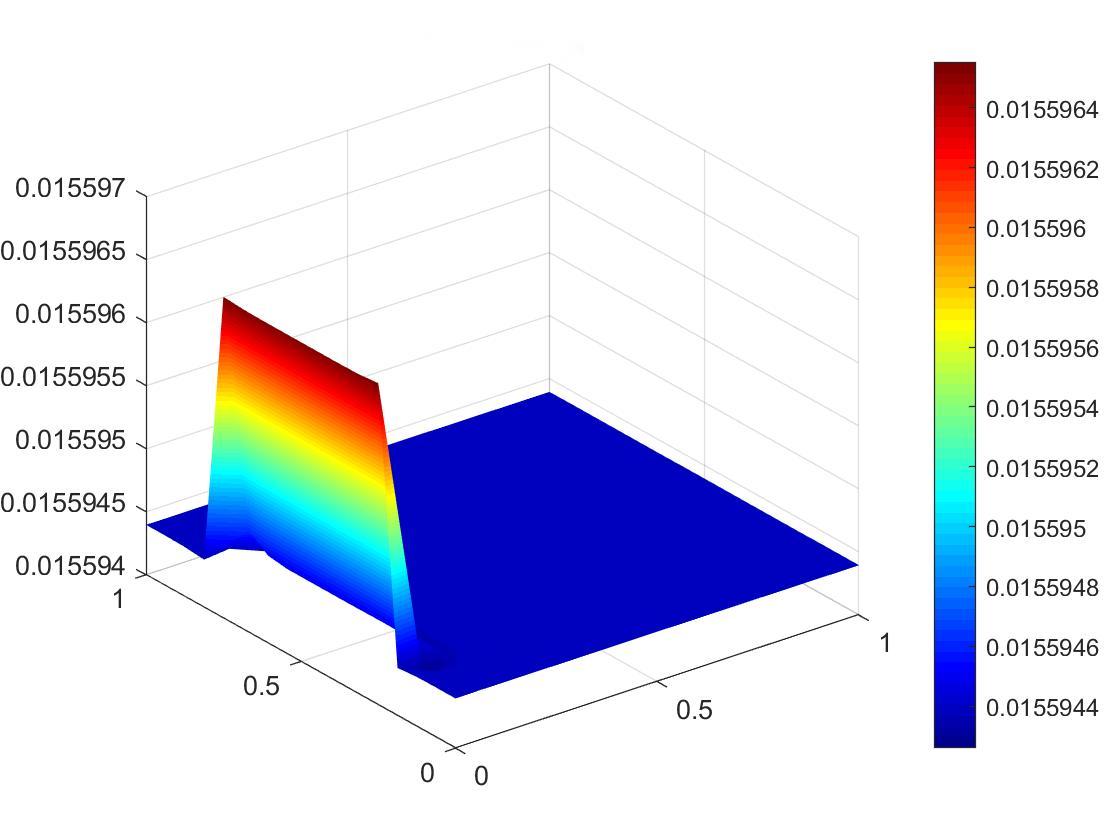}
		\caption{ The numerical  pressure\ $p_{h}^{n+1}$ at the terminal time $T$ with the parameters of Table \ref{tab109} for the reformulated model.}\label{figure_10}
	\end{minipage}
\end{figure}
\begin{figure}[H]
	\centering
	\includegraphics[height=1.8in,width=2.5in]{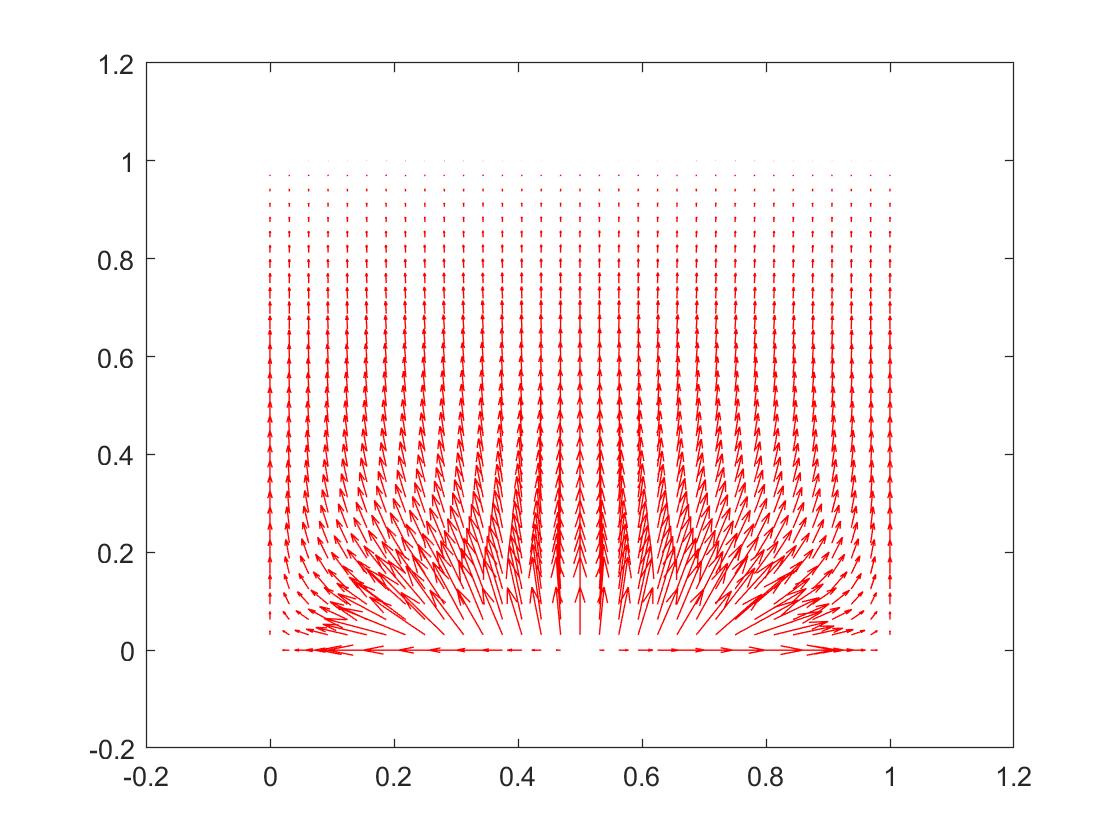}
	\caption{ Arrow plot of the computed displacement $ \pmb\tau $ with the parameters of Table \ref{tab109} for the reformulated model.}\label{figure_101}
\end{figure}
Figure  \ref{figure_9}-Figure\ref{figure_101} show the numerical solution of pressure \ $p_{h}^{n+1}$ for the original model and the reformulated model and the arrow plot of the computed displacement $ \pmb\tau $ for the reformulated model corresponding to the parameters of Table \ref{tab109} of Test 3. It is easy to find that there is  no ``locking phenomenon". 

{\bf Test 4.} This problem is a real two- dimensional footing problem (cf. \cite{20220101}). The simulation domain is a 100 by 100 meters block of porous soil, $\Omega=[-50,50]\times [0,100]$, $ T=0.01s$. At he base of this domain the soil is assumed to be fixed while at some centered upper part of the domain a uniform load of intensity $ \sigma_{0}=10^4 N/m^2$ is applied in a strip of length $40$ m. The whole domain is assumed free to drain. The boundary condition are given as follows
\begin{alignat*}{2}
	p &= 0  &&\qquad\mbox{on }\partial\Omega_T,\\
	\sigma_{xy}=0,\qquad \lambda^{*}(\div\pmb\tau)_{t}+\sigma_{yy} &=-\sigma_{0} &&\qquad\mbox{on }\Gamma_1\times (0,T),\\
	\sigma_{xy}=0,\qquad \lambda^{*}(\div\pmb\tau)_{t}+\sigma_{yy} &=-\sigma_{0} &&\qquad\mbox{on }\Gamma_2\times (0,T),\\
	\pmb\tau &= \mathbf{0} &&\qquad\mbox{on } \partial\Ome\setminus(\Gamma_1\cup\Gamma_1),
\end{alignat*}
where $ \sigma_{xy}=\frac{\gamma}{2}(\frac{\p\tau_{1}}{\p x_{2}}+\frac{\p\tau_{2}}{\p x_{1}}) $, $ \sigma_{yy}=\gamma\frac{\p\tau_{2}}{\p x_{2}}+\beta(\frac{\p\tau_{1}}{\p x_{1}}+\frac{\p\tau_{2}}{\p x_{2}}) $ and
\begin{alignat*}{2}
	\Gamma_1=\left\lbrace(x_{1},x_{2})\in\p\Omega,\left| x_{1}\right|\leq 20,~x_{2}=100  \right\rbrace,\qquad\Gamma_2=\left\lbrace(x_{1},x_{2})\in\p\Omega,\left| x_{1}\right|> 20,~x_{2}=100  \right\rbrace. 
\end{alignat*}
The material properties of the porous medium are given in Table \ref{tab1010}.
\begin{table}[H]
	\begin{center}
		%\smallskip
		\caption{Values of parameters}\label{tab1010}
		\begin{tabular}{l c c c }
			\hline
			Parameters &Description  &Values  &Unit\\ \hline
			$\lambda^{*}$& Coefficient of secondary consolidation          &1e-2 &-\\
			$\nu $ &  Poisson ratio   & 0.2  &-    \\ %\hline
			$b_0$ &  Biot-Willis constant   & 1&-   \\ %\hline
			$E$ &  Young's modulus   & 3e4 &$ N/m^2 $\\ %\hline
			$\beta$ &  Lam$\acute{e}$ constant   & 8.333e3 &$ N/m^2 $ \\%\hline
			$ K $ & Permeability tensor  & (1e-15) $\bf I$ &$ m^2 $\\% \hline
			$ \gamma $ &  Lam$\acute{e}$ constant  & 1.25e4&$ N/m^2 $\\%\hline
			$ a_0 $ &  Constrained specific storage coefficient  & 2e-8&-\\
			$ \theta_{f} $	& Fluid viscosity &1e-3&Pa s\\\hline
		\end{tabular}
	\end{center}
\end{table}
\begin{figure}[H]
	\begin{minipage}[t]{0.45\linewidth}
		\centering
		\includegraphics[height=1.8in,width=2in]{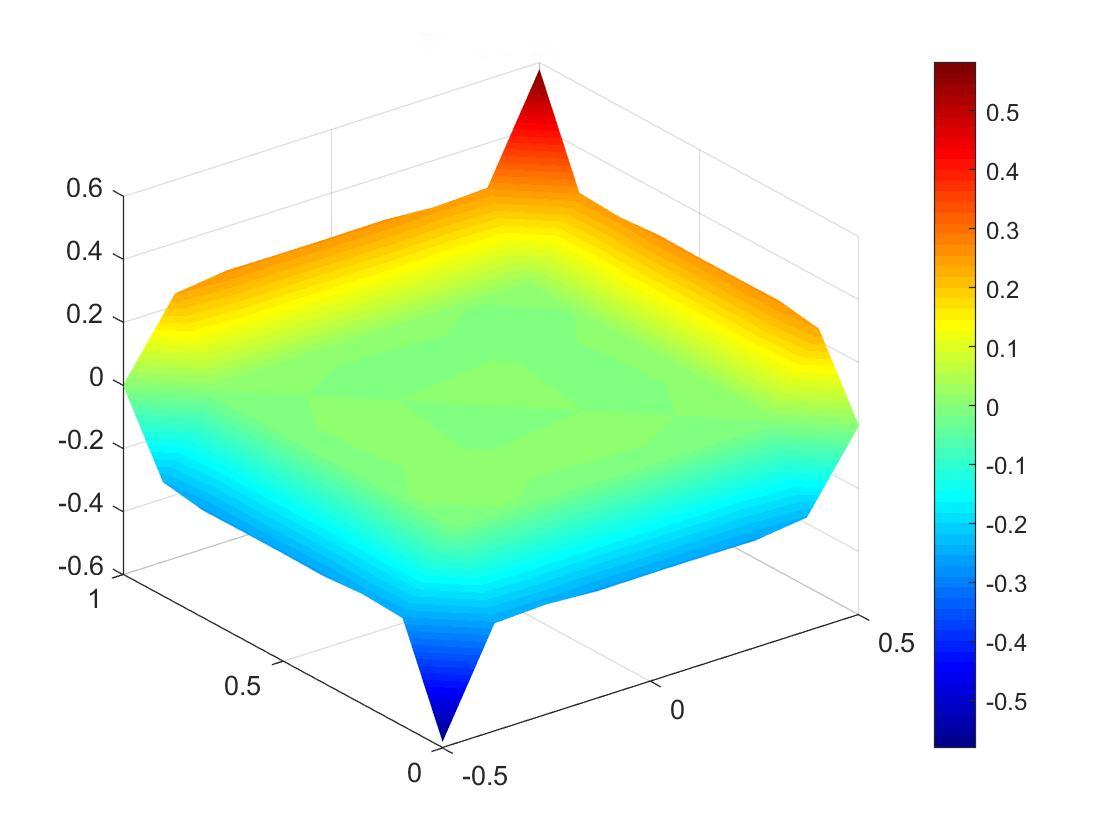}
		\caption{The numerical  pressure\ $p_{h}^{n+1}$ at the terminal time $T$ with the parameters of Table \ref{tab1010} for the original model.}\label{figure_11}
	\end{minipage}
	\hspace{0.5in}
	\begin{minipage}[t]{0.45\linewidth}
		\centering
		\includegraphics[height=1.8in,width=2in]{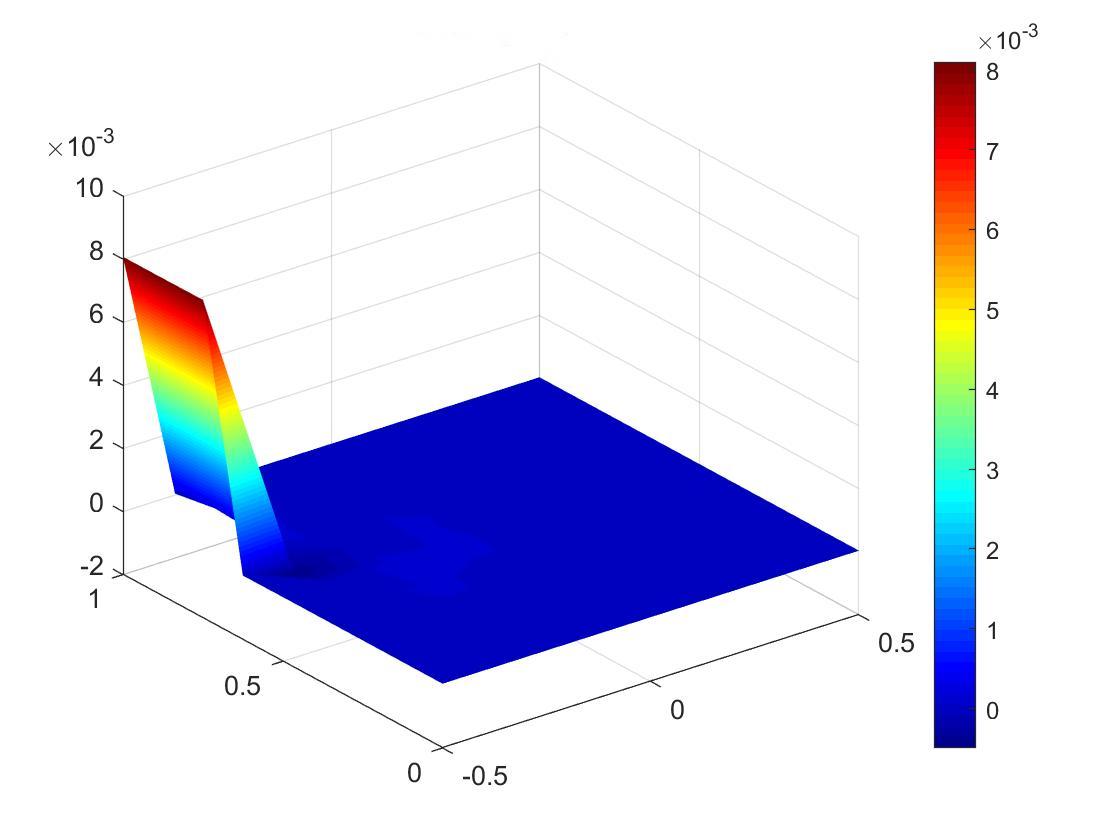}
		\caption{ The numerical  pressure\ $p_{h}^{n+1}$ at the terminal time $T$ with the parameters of Table \ref{tab1010} for the reformulated model.}\label{figure_12}
	\end{minipage}
\end{figure}
\begin{figure}[H]
	\centering
	\includegraphics[height=1.8in,width=2.5in]{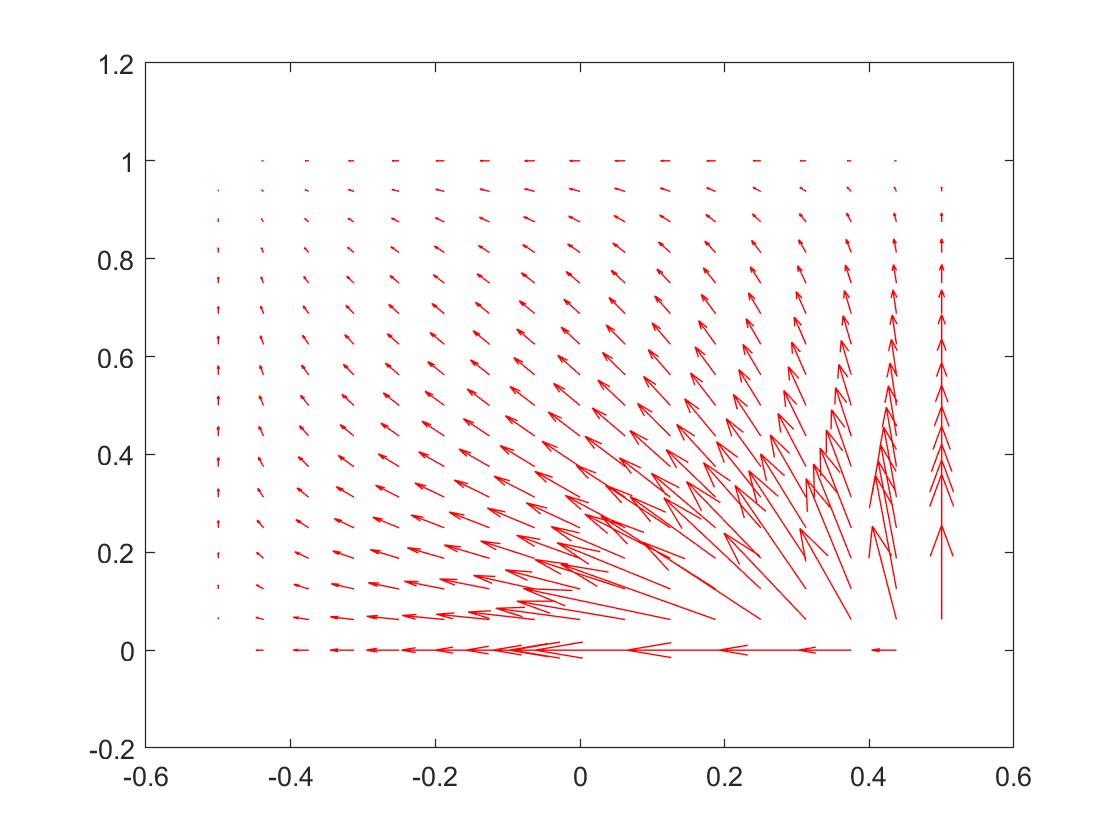}
	\caption{ Arrow plot of the computed displacement $ \pmb\tau $ with the parameters of Table \ref{tab1010} for the reformulated model.}\label{figure_121}
\end{figure}
From Figure \ref{figure_11}-Figure \ref{figure_121}, we find that the arrows near the boundary get along very well with those on the boundary. We conclude that the numerical solution with the reformulated model has no oscillations in numerical pressure. 

\section{Conclusion}
In this paper, we propose a new multiphysics finite element method for a Biot model with secondary consolidation in soil dynamics. To better describe the processes of deformation and diffusion underlying in the original model, we introduce new variables
$q=\div \pmb\tau, \varpi=a_0p+b_0 q, \delta:=b_0 p -\lam q-\lambda^{*}q_{t}$ to reformulate Biot model with secondary consolidation so that we successfully transform the fluid-solid coupling problem into a fluid coupled problem and and the parabolic problem into a Stokes problem, where the multiphysics approach is different from the introduced variables in \cite{fgl14}. Then, we give the energy law and prior error estimate of the weak solution. Also, we design a fully discrete time-stepping scheme to use multiphysics finite element method with $P_2-P_1-P_1$ element pairs for the space variables and backward Euler method for the time variable, and we derive the discrete energy laws and the optimal convergence order error estimates. To the best of our knowledge, it is a complete new method and the first time to give the optimal convergence order error estimates for the new proposed method for a Biot model with secondary consolidation. Also, we show some numerical examples to verify the rationality of theoretical analysis and there is no ``locking phenomenon".

\end{document}